\documentclass[10pt,a4paper,oneside,final]{article}
\usepackage{amsmath}
\usepackage{amssymb}
\usepackage{amsthm}
\usepackage{color}
\usepackage{fullpage}
\usepackage{hyperref}
\usepackage[capitalise]{cleveref}
\usepackage[pdftex]{graphicx}
\usepackage[UKenglish]{babel}

\newcommand{\R}{{\mathbb{R}}}

\newcommand{\Z}{{\mathbb{Z}}}
\newcommand{\N}{{\mathbb{N}}}
\newcommand{\de}{{\mathrm{d}}}
\DeclareMathOperator{\dv}{div}

\newcommand{\dist}{{\mathrm{dist}}}
\newcommand{\id}{\mathrm{id}}
\newcommand{\setchar}[1]{\mathbf{1}_{#1}}

\newcommand{\grid}[1]{\mathrm{Grid}_{#1}}
\newcommand{\lebesgue}{\mathcal{L}}

\newcommand{\hd}{\mathcal{H}}
\newcommand{\hdone}{\mathcal{H}^1}

\newcommand{\rca}{\mathrm{rca}}
\newcommand{\fbm}{{\mathrm{fbm}}}
\newcommand{\prob}{{\mathcal{P}}}
\newcommand{\pushforward}[2]{{{#1}_{\#}#2}}
\newcommand{\restr}{{\mbox{\LARGE$\llcorner$}}}
\newcommand{\spt}{{\mathrm{spt}}}
\newcommand{\diffForms}[1]{{\mathcal{D}^{#1}}}
\newcommand{\currents}[1]{{\mathcal{D}_{#1}}}
\newcommand{\mass}[1]{M(#1)}
\newcommand{\flatNorm}[1]{|#1|^\flat}
\newcommand{\flatChains}[1]{\mathbf F_{#1}}

\newcommand{\weakstarto}{\stackrel{*}{\rightharpoonup}}
\newcommand{\Wd}[1]{\mathrm{W}_{#1}}
\newcommand{\Wdone}{\Wd{1}}
\newcommand{\AC}{\mathrm{AC}}
\newcommand{\Lip}{\mathrm{Lip}}

\newcommand{\cont}{{\mathrm{C}}}
\newcommand{\contbdd}{\cont^0}
\newcommand{\contsmooth}{\cont_{c}^{\infty}}
\newcommand{\flux}{{\mathcal{F}}}

\newcommand{\reSpace}{\Gamma}
\newcommand{\reMeasure}{P_{\reSpace}}

\newcommand{\TPM}{\textrm{TPM}}

\newcommand{\JEn}{{\mathcal{J}}}

\newcommand{\JEnXia}[1][\tau]{#1_{\mathrm{F}}}

\newcommand{\JEnMMS}[1][\tau]{#1_{\mathrm{P}}}

\newcommand{\dtau}{d_\tau}
\newcommand{\Bcal}{\mathcal{B}}

\newcommand{\transportPath}{mass flux}
\newcommand{\transportPaths}{mass fluxes}
\newcommand{\TransportPath}{Mass flux}

\numberwithin{equation}{section}

\theoremstyle{plain}
\newtheorem{theorem}{Theorem}[section]
\newtheorem{lemma}[theorem]{Lemma}
\newtheorem{proposition}[theorem]{Proposition}
\newtheorem{corollary}[theorem]{Corollary}
\newtheorem{problem}[theorem]{Problem}

\theoremstyle{definition}
\newtheorem{definition}[theorem]{Definition}
\newtheorem*{TODO}{\textcolor{red}{TODO}}

\theoremstyle{remark}
\newtheorem{remark}[theorem]{Remark}
\newtheorem{example}[theorem]{Example}

\newcommand{\keywords}[1]{\noindent\textbf{Keywords:}\enspace#1}
\newcommand{\subjclass}[1]{\bigskip\noindent\emph{2010 MSC:}\enspace#1}

\usepackage{enumitem}
\usepackage{marginnote}
\newcommand{\notinclude}[1]{}

\begin{document}

\title{General transport problems with branched minimizers as functionals of $1$-currents with prescribed boundary}
\author{Alessio Brancolini\footnote{Institute for Numerical and Applied Mathematics, University of M\"unster, Einsteinstra\ss{}e 62, D-48149 M\"unster, Germany}\ \footnote{Email address: \texttt{alessio.brancolini@uni-muenster.de} (corresponding author)} \and Benedikt Wirth\footnotemark[1]\ \footnote{Email address: \texttt{benedikt.wirth@uni-muenster.de}}}
\maketitle

\begin{abstract}
A prominent model for transportation networks is branched transport,
which seeks the optimal transportation scheme to move material from a given initial to a given final distribution.
The cost of the scheme encodes a higher transport efficiency the more mass is moved together, which automatically leads to optimal transportation networks with a hierarchical branching structure.
The two major existing model formulations use either \transportPaths{} (vector-valued measures, Eulerian formulation) or patterns (probabilities on the space of particle paths, Lagrangian formulation).

In the branched transport problem the transportation cost is a fractional power of the transported mass.
In this paper we instead analyse the much more general class of transport problems in which the transportation cost is merely a nonnegative increasing and subadditive function
(in a certain sense this is the broadest possible generalization of branched transport).
In particular, we address the problem of the equivalence of the above-mentioned formulations in this wider context.
However, the newly-introduced class of transportation costs lacks strict concavity which complicates the analysis considerably.
New ideas are required, in particular, it turns out convenient to state the problem via $1$-currents.

Our analysis also includes the well-posedness, some network properties, as well as a metrization and a length space property of the model cost,
which were previously only known for branched transport.
Some already existing arguments in that field are given a more concise and simpler form.
%

\bigskip\keywords{optimal transport, optimal networks, branched transport, irrigation, urban planning, Wasserstein distance, geometric measure theory, currents}

\subjclass{49Q20, 49Q10, 90B10}
\end{abstract}


\section{Introduction}
The classical theory for cost-efficient transportation of an amount of material from a given initial to a given final mass distribution is the theory of optimal transport.
Proposed by Monge in the end of the 18\textsuperscript{th} century, a great impulse was given by Kantorovich in the mid 20\textsuperscript{th} century.
The optimal transport problem assigns the minimum possible transport cost to each pair of initial and final mass distributions.
Here, the movement of an amount of mass from a position $A$ to a position $B$ contributes a transportation cost which is proportional to the transported mass and which may depend in a rather general way on $A$ and $B$.

In the setting of classical optimal transport, all mass particles move independently from each other.
If, however, the transportation cost is only subadditive in the transported mass, which models an efficiency gain if mass is transported in bulk,
then material particles start interacting, and an optimal transportation scheme will move the particles along paths that together form ramified network structures.
This can be used to model for instance biological networks such as vascular systems in plants and animals.

If the transportation cost (per transport unit distance) is a fractional power of the transported mass $w$,
\begin{equation*}
\tau(w)=w^\alpha\quad\text{for some }\alpha\in(0,1),
\end{equation*}
we obtain the \emph{branched transport models} by Xia \cite{Xia-Optimal-Paths} and by Maddalena, Morel, and Solimini \cite{Maddalena-Morel-Solimini-Irrigation-Patterns}.
The model by Maddalena, Morel, and Solimini employs a Lagrangian formulation based on \emph{irrigation patterns} $\chi$ that describe the position of each mass particle $p$ at time $t$ by $\chi(p,t)$.
The model by Xia on the other hand uses a Eulerian formulation in which only the flux of particles is described, discarding its dependence on the time variable $t$.
The equivalence between both model formulations was shown in \cite{Maddalena-Solimini-Transport-Distances,Maddalena-Solimini-Synchronic};
a comprehensive reference is the monograph \cite{BeCaMo09}.

The motivation for generalizing the choice $\tau(w)=w^\alpha$ comes from work by the current authors \cite{BrWi15-equivalent} (and the subsequent studies \cite{BrRoWi16,BrWi15-micropatterns}) where it is shown that the \emph{urban planning model} (introduced in \cite{Brancolini-Buttazzo}) can be formulated in the same setting as branched transport, just using a different (but no longer strictly concave) transportation cost.

In this work we generalize both branched transport and urban planning models and their analysis by replacing their transportation costs by the most general class of transportation cost functions as described below.
Furthermore, the lack of strict subadditivity forces us to provide a model description in terms of $1$-currents, which in the special case of branched transport had already been conceived by Xia \cite{Xia-Interior-Regularity}.
The potent measure geometric tools help to gain a more intuitive understanding of the models and greatly reduce the effort in comparing different model formulations.
Of course, this comes at the cost of introducing the measure geometric machinery, building on classical as well as more recent results by White \cite{Wh99b}, Smirnov \cite{Sm93}, \v{S}ilhav\'y \cite{Si07} or Colombo et al.\ \cite{CoRoMa17}.

Our analysis also serves as a preparation for a further study in which we will introduce yet another model formulation quite different from the ones considered here and very much akin to the original formulation of urban planning.

\subsection{General transportation costs}
We will be concerned with transportation costs of the following form.
\begin{definition}[Transportation cost]\label{def:transportation_cost}
The \emph{transportation cost} is a function $\tau : [0,\infty) \to [0,\infty)$ such that
\begin{enumerate}
 \item $\tau(0) = 0$ and $\tau(w)>0$ for $w>0$,
 \item $\tau$ is non-decreasing,
 \item $\tau$ is subadditive,
 \item $\tau$ is lower semi-continuous.
\end{enumerate}
\end{definition}

The transportation cost $\tau(w)$ has the interpretation of the cost per transport unit distance for transporting an amount of mass $w$.
Note that any non-decreasing concave function $\tau\neq0$ with $\tau(0)=0$ can be chosen as a transportation cost.
As mentioned above, the generalization of the \notinclude{branched transport }choice $\tau(w)=w^\alpha$ to more general subadditive costs is motivated by \cite{BrWi15-equivalent},
which is concerned with a different model for transport networks, the urban planning model, and which provides a model formulation analogous to branched transport, just with a different transportation cost.
The most well-known choices for $\tau$ are summarized in the following example, and with this article we advocate the use of more general $\tau$ that may be tailored to applications.

\begin{example}[Transportation cost]
Classical optimal transport, branched transport, urban planning, and a variant of the Steiner problem can be retrieved using
\begin{enumerate}
\item the \emph{Wasserstein cost} $\tau(w)=aw$ for some $a>0$,
\item the \emph{branched transport cost} $\tau(w)=w^\alpha$ for some $\alpha\in(0,1)$,
\item the \emph{urban planning cost} $\tau(w)=\min\{aw,w+\varepsilon\}$ for some $a>1$, $\varepsilon>0$, or
\item the \emph{discrete cost} $\tau(w)=1$ if $w>0$ and $\tau(0)=0$.
\end{enumerate}
\end{example}

Note that the properties required in \cref{def:transportation_cost} are dictated by first principles:
Transporting mass has a positive cost, where the cost increases with increasing mass.
Furthermore, the transport cost only jumps when the maximum capacity of a transportation means is reached (for instance, if a lorry is fully loaded), which implies lower semi-continuity.
In addition, the transportation cost is subadditive,
since there is always the option of splitting the mass into several parts and transporting those separately.
A direct consequence is that the average transportation cost per particle is bounded below as follows.

\begin{lemma}[\protect{\cite[Thm.\,5 and its proof]{La62}}]\label{thm:averageCost}
Let $\tau$ be a transportation cost and define $$\lambda^\tau(m)=\inf\left\{\tfrac{\tau(w)}w\,:\,w\in\left(\tfrac m2,m\right]\right\}>0$$ for $m>0$.
Then $\tau(w)\geq\lambda^\tau(m)w$ for all $w\in[0,m]$.
\end{lemma}

In principle, the condition of lower semi-continuity may be dropped, however,
in the \transportPath{}-based model formulation this would lead to the same model as taking the lower semi-continuous envelope of $\tau$,
while in the pattern-based formulation this would lead to non-existence of optimal networks.

An important feature of our generalization is that now also non-concave and non-strictly subadditive transportation costs are allowed, which will complicate the analysis in parts but covers cases of interest such as urban planning.
As such, this work makes use of arguments from \cite{Xia-Optimal-Paths,Maddalena-Morel-Solimini-Irrigation-Patterns,Bernot-Caselles-Morel-Traffic-Plans,Maddalena-Solimini-Transport-Distances,Maddalena-Solimini-Synchronic,BrWi15-equivalent},
but all is transferred into this more general setting and in several places streamlined.
In particular, any reference to the specific form of the transportation cost (which is exploited in all of the above works) is eliminated.

\subsection{Summary of main results, techniques, and perspectives}

Given two probability measures $\mu_+$ and $\mu_-$ on $\R^n$, denoting the material source and the sink,
the Eulerian model formulation will describe the mass transport from $\mu_+$ to $\mu_-$ via a vector-valued measure $\flux$ (a \transportPath{}) with $\dv\flux=\mu_+-\mu_-$, which represents the material flux through each point (cf.\,\cref{def:mass_fluxes}).
If $\flux$ can be represented as a weighted directed graph $G$, whose edge weight $w(\vec e)$ represents the mass flux through edge $\vec e$,
then its cost functional will be defined as (cf.\,\cref{def:costXia})
\begin{equation*}
 \JEnXia(G) = \sum_{\text{edges }\vec e\text{ of }G}\tau(w(\vec e))|\vec e|\,.
\end{equation*}
Otherwise, $\flux$ will be approximated in an appropriate sense by sequences $G_k$ of graphs transporting $\mu_+^k$ to $\mu_-^k$, and its cost will be defined via relaxation as
\begin{equation*}
 \JEnXia(\flux) = \inf\left\{\liminf_{k \to \infty} \JEnXia(G_k) \ : \ (\mu_+^k,\mu_-^k,\flux_{G_k}) \weakstarto (\mu_+,\mu_-,\flux)\right\}\,.
\end{equation*}
The Lagrangian model on the other hand will describe the mass transport via an irrigation pattern $\chi:[0,1]\times[0,1]\to\R^n$ with $\chi(p,t)$ being the position of mass particle $p$ at time $t$ (cf.\,\cref{def:reference_space}).
Its cost will essentially be defined as
\begin{displaymath}
 \JEnMMS(\chi) = \int_{[0,1]\times[0,1]} \frac{\tau(m_\chi(\chi(p,t)))}{m_\chi(\chi(p,t))}|\dot\chi(p,t)|\,\de t\de p\,,
\end{displaymath}
where integration is with respect to the Lebesgue measure and $m_\chi(x)$ denotes the total amount of all mass particles $p$ travelling through $x$ (cf.\,\cref{def:patternCost}).

We are interested in the following optimization problems.
\begin{problem}[Flux and pattern optimization problem]
Given $\mu_+$ and $\mu_-$, the problems of finding an associated optimal \transportPath{} and irrigation pattern are
\begin{equation}\label{eqn:optProblems}
\begin{aligned}
&\min\{\JEnXia(\flux)\ :\ \flux\text{ is a \transportPath{} from }\mu_+\text{ to }\mu_-\}\,,\\
&\min\{\JEnMMS(\chi)\ :\ \chi\text{ is an irrigation pattern moving }\mu_+\text{ to }\mu_-\}\,.
\end{aligned}
\end{equation}
\end{problem}

The main results of the paper are the following.

\paragraph{Equivalence of formulations.}

In \cref{thm:existenceFluxAdmissible} and \cref{thm:modelEquivalence} we prove that under appropriate growth condition on the transportation cost $\tau$ (cf.\,\cref{def:admissible}) both problems in \eqref{eqn:optProblems} are well-posed, that is, they admit minimizers.
Moreover, both problems in \eqref{eqn:optProblems} have the same minimum (cf.\,\cref{thm:modelEquivalence}) and a minimizer of one of them induces a minimizer of the other.
Indeed, optimal \transportPath{} $\flux$ and irrigation pattern $\chi$ are related via
\begin{equation*}
\int_{\R^n}\varphi\cdot\de\flux=\int_{[0,1]\times[0,1]}\varphi(\chi(p,t))\cdot\dot\chi(p,t)\,\de t\de p \quad\text{for all appropriate test functions }\varphi.
\end{equation*}

\paragraph{Structure of the optimal \transportPath{} $\flux$.}

We study the structure of the optimal \transportPath{} $\flux$. We prove that it can be decomposed into $\flux=\theta\hdone\restr S+\flux^\perp$ for a countably $1$-rectifiable set $S\subset\R^n$, $\theta:S\to\R^n$, and a diffuse part $\flux^\perp$ (cf.\,\cref{thm:GilbertFlux}) such that the cost turns into
\begin{equation*}
\JEnXia(\flux)
=\int_{S}\tau(|\theta(x)|)\,\de\hdone(x)+\tau'(0)|\flux^\perp|(\R^n)\,.
\end{equation*}
The novelty in this formula is to introduce the diffuse part, which takes into account that a transport cost may have finite initial slope so that small amounts of mass may be moved more conveniently in a ``Monge--Kantorovich'' way (typically close to the the initial or the final measure).
Moreover, the optimal irrigation pattern $\chi$ has $m_\chi>0$ only on a countably $1$-rectifiable set $S\subset\R^n$ and its cost is (cf.\,\cref{thm:GilbertPatterns})
\begin{equation*}
\JEnMMS(\chi)
=\int_{S}\tau(m_{\chi}(x))\,\de\hdone(x)+\tau'(0)\int_{[0,1]}\hdone(\chi(p,[0,1])\setminus S)\,\de p\,.
\end{equation*}

\paragraph{Length space property.}

We also consider the minimum value $\dtau(\mu_+,\mu_-)$ of \eqref{eqn:optProblems}, which induces a metric on the space of probability measures that metrizes weak-$*$ convergence (cf.\,\cref{thm:metrization}).
Endowed with this metric, the space of probability measures is a length space, i.e. the metric $\dtau$ is induced by shortest paths (cf.\,\cref{thm:lengthSpaceProperty}).

\paragraph{Other properties of the network.}

We show that it is in general no longer true (unlike in branched transport) that optimal transportation networks have a tree-like structure; however, we prove several weaker properties of optimal transportation networks, for example the absence of loops.

\paragraph{1-current setting.}

In several places the key idea is to recast the Eulerian formulation as an optimization problem on $1$-currents with prescribed boundary, as done by Xia in \cite{Xia-Interior-Regularity} to study the interior regularity of optimal transport paths.

\subsection{Notation}
Throughout the article, we will use the following notation.
\begin{itemize}
 \item $\lebesgue^m$ denotes the $m$-dimensional \emph{Lebesgue measure}.

 \item $\hd^m$ denotes the $m$-dimensional \emph{Hausdorff measure}.

 \item $\fbm(\R^n)$ denotes the set of \emph{nonnegative finite Borel measures} on $\R^n$. Notice that these measures are countably additive and also regular by \cite[Thm.\,2.18]{Ru87}.
 The total variation measure of $\mu\in\fbm(\R^n)$ is defined as $|\mu|(A)=\sup\{\sum_{i=1}^\infty|\mu(A_i)|\ :\ A_i\subset A\text{ measurable and disjoint }\}$.
 The total variation norm of $\mu$ then is $|\mu|(\R^n)$.

 \item $\rca(\R^n;\R^n)$ denotes the set of \emph{$\R^n$-valued regular countably additive measures} on $\R^n$.
 The total variation measure of $\flux\in\rca(\R^n;\R^n)$ and its total variation norm are $|\flux|$ and $|\flux|(\R^n)$, respectively.

 \item \emph{Weak-$*$ convergence} on $\fbm(\R^n)$ or $\rca(\R^n;\R^n)$ is indicated by $\weakstarto$.

 \item The \emph{support} $\spt\mu$ of a measure $\mu$ in $\fbm(\R^n)$ or $\rca(\R^n;\R^n)$ is the smallest closed set $B$ with $|\mu|(\R^n\setminus B)=0$.

 \item The \emph{restriction} of a measure $\mu$ in $\fbm(\R^n)$ or $\rca(\R^n;\R^n)$ to a measurable set $B$ is the measure defined by $\mu\restr B(A)=\mu(A\cap B)$ for all measurable sets $A$.

 \item The \emph{pushforward} of a measure $\mu$ on $X$ under a measurable map $T:X\to Y$ is the measure defined by $\pushforward T\mu(A)=\mu(T^{-1}(A))$ for all measurable sets $A$.

 \item The \emph{Dirac mass} in $x\in\R^n$ is the measure $\delta_x(A)=1$ if $x\in A$ and $\delta_x(A)=0$ else.
 
 \item The \emph{Wasserstein-$1$-metric} $\Wdone(\mu_+,\mu_-)$ between two measures $\mu_+,\mu_-\in\fbm(\R^n)$ of equal mass is defined as $\Wdone(\mu_+,\mu_-)=\sup\{\int_{\R^n}f\,\de(\mu_+-\mu_-)\ :\ f\text{ Lipschitz with constant }\leq1\}$.
 It metrizes weak-$*$ convergence on the space of nonnegative finite Borel measures with equal mass.
 
 \item $I=[0,1]$ denotes the \emph{unit interval}.

 \item $\AC(I;\R^n)$ denotes the set of \emph{absolutely continuous functions} $f:I\to\R^n$.

 \item $\Lip(I;\R^n)$ denotes the set of \emph{Lipschitz functions} $f:I\to\R^n$.

 \item The \emph{characteristic function} of a set $A$ is defined as $\setchar{A}(x)=1$ if $x\in A$ and $\setchar A(x)=0$ else.

 \item $\Lambda^m(V)$ and $\Lambda_m(V)$ denote the vector spaces of \emph{$m$-vectors} and \emph{$m$-covectors} (sometimes called alternating $m$-linear forms) in the vector space $V$, respectively (using the notation from \cite[1.3-4]{Fe69}).
 In detail, $\Lambda^m(V)$ is the quotient space of the $m$-fold tensor product of $V$ with respect to the identification $x\otimes x=0$ for all $x\in V$, and $\Lambda_m(V)$ is its dual.
 If $V$ is equipped with an inner product, then an inner product of two $m$-vectors $v_1\wedge\ldots\wedge v_m,w_1\wedge\ldots\wedge w_m\in\Lambda^m(V)$ is defined by $(v_1\wedge\ldots\wedge v_m,w_1\wedge\ldots\wedge w_m)_{\Lambda^m(V)}=\det((v_i\cdot w_j)_{ij}$,
 which induces a norm $\|\cdot\|$ on $\Lambda^m(V)$.
 The corresponding operator norm on $\Lambda_m(V)$ is also denoted $\|\cdot\|$.

 \item $\cont^m$ (with varying domain and range specifications) denotes the vector space of \emph{$m$ times boundedly and continuously differentiable functions} with norm $\|\cdot\|_{\cont^m}$ being the supremum over the domain of all absolute derivatives up to order $m$.
 For instance, $\cont^0$ denotes the space of bounded continuous functions.
 
 \item $\cont_c$ and $\contsmooth$ denote the set of \emph{continuous} and \emph{smooth functions with compact support}, respectively.

\end{itemize}

The article is organized as follows.
\Cref{sec:Xia} introduces the Eulerian model formulation and examines properties of optimal transportation networks.
Furthermore, the metrization and the length space property are shown in that section.
\Cref{sec:MMS} then introduces the Lagrangian model formulation via irrigation patterns,
and the equivalence between Eulerian and Lagrangian formulation is proved in \cref{sec:equivalence}.

\section{Eulerian model for transportation networks}\label{sec:Xia}

We start by recapitulating the model formulation due to Xia \cite{Xia-Optimal-Paths}
and subsequently analyse its properties in our more general setting.
In large but not all parts we can follow the original arguments by Xia.

\subsection{Model definition}

Here, transportation networks are described with the help of graphs.
First only transport between discrete mass distributions is considered, and then general transportation problems are obtained via a relaxation technique.

\begin{definition}[\TransportPath{}]\label{def:mass_fluxes}
\begin{enumerate}
\item A \emph{discrete finite mass} shall be a nonnegative measure of the form $\mu = \sum_{i = 1}^k a_i\delta_{x_i}$ with $k\in\N$, $a_i>0$, $x_i\in \R^n$.

\item Let $\mu_+,\mu_-$ be discrete finite masses with $\mu_+(\R^n)=\mu_-(\R^n)$.
A \emph{discrete \transportPath{}} between $\mu_+$ and $\mu_-$ is a weighted directed graph $G$ with vertices $V(G)$, straight edges $E(G)$, and edge weight function $w : E(G) \to [0,\infty)$ such that the following \emph{mass preservation conditions} hold,
\begin{equation}\label{eqn:massPreservation}
\mu_+(\{v\}) + \sum_{\substack{e \in E(G)\\e^- = v}} w(e) = \mu_-(\{v\}) + \sum_{\substack{e \in E(G)\\e^+ = v}} w(e)
\qquad\text{for all }v\in V(G)\cup\spt\mu_+\cup\spt\mu_-,
\end{equation}
where $e^+$ and $e^-$ denote the initial (source) and final (sink) point of edge $e$.

\item The \emph{flux} associated with a discrete \transportPath{} $G$ is given by
\begin{equation}\label{eqn:graphFlux}
 \flux_G = \sum_{e \in E(G)} w(e)\mu_e\,,
\end{equation}
where every edge $e\in E(G)$ with direction $\hat e=\frac{e^--e^+}{|e^--e^+|}$ was identified with the vector measure $\mu_e = (\hdone\restr e)\, \hat e$.
\Cref{eqn:massPreservation} is equivalent to $\dv\flux_G = \mu_+ - \mu_-$ (in the distributional sense).

\item A vector measure $\flux\in\rca(\R^n;\R^n)$ is a \emph{\transportPath{}} between two nonnegative measures $\mu_+$ and $\mu_-$ (also known as \emph{transport path}), if there exist sequences of discrete finite masses $\mu_+^k$, $\mu_-^k$ with $\mu_+^k \weakstarto \mu_+$, $\mu_-^k \weakstarto \mu_-$
and a sequence of fluxes $\flux_{G_k}$ with $\flux_{G_k} \weakstarto \flux$, $\dv \flux_{G_k} = \mu_+^k - \mu_-^k$.
A sequence $(\mu_+^k,\mu_-^k,\flux_{G_k})$ satisfying these properties is called \emph{approximating graph sequence}, and we write $(\mu_+^k,\mu_-^k,\flux_{G_k}) \weakstarto (\mu_+,\mu_-,\flux)$.
Note that $\dv\flux=\mu_+-\mu_-$ follows by continuity with respect to weak-$*$ convergence.
\end{enumerate}
\end{definition}

In the above, $\mu_+$ has the interpretation of the initial material distribution or mass source, while $\mu_-$ represents the final distribution or sink.
The edge weight $w(e)$ indicates the amount of mass flowing along edge $e$
so that \eqref{eqn:massPreservation} expresses mass conservation on the way from initial to final distribution.
Indeed, \eqref{eqn:massPreservation} implies that the Dirac locations of $\mu_+$ and $\mu_-$ form vertices as well and that at every vertex $v$ the total mass influx equals the total outflux.
Thus, a \transportPath{} essentially encodes how the mass moves from $\mu_+$ to $\mu_-$, and it can be associated with a cost.

\begin{definition}[Cost functional]\label{def:costXia}
\begin{enumerate}
\item Given a transportation cost $\tau$, the \emph{cost function} of a discrete \transportPath{} $G$ between $\mu_+$ and $\mu_-$ is
\begin{displaymath}
 \JEnXia(G) = \sum_{e \in E(G)} \tau(w(e))\,l(e)\,,
\end{displaymath}
where $l(e)$ is the length of edge $e$.

\item The \emph{cost function} of a \transportPath{} $\flux$ between $\mu_+$ and $\mu_-$ is defined as
\begin{equation}\label{eq:functional_XiaEn}
 \JEnXia(\flux) = \inf\left\{\liminf_{k \to \infty} \JEnXia(G_k) \ : \ (\mu_+^k,\mu_-^k,\flux_{G_k}) \weakstarto (\mu_+,\mu_-,\flux)\right\}\,.
\end{equation}
We furthermore abbreviate
\begin{equation*}
\JEn^{\tau,\mu_+,\mu_-}[\flux]
=\begin{cases}
\JEnXia(\flux)&\text{if }\dv\flux=\mu_+-\mu_-,\\
\infty&\text{else.}
\end{cases}
\end{equation*}

\item Given $\mu_+,\mu_- \in \fbm(\R^n)$, the \emph{transport problem} is to find the solution $\flux$ of
\begin{equation*}
\dtau(\mu_+,\mu_-)
= \min_{\flux\in\rca(\R^n;\R^n)}\JEn^{\tau,\mu_+,\mu_-}[\flux]\,,
\end{equation*}
where $\dtau$ is called \emph{cost distance}.
\end{enumerate}
\end{definition}

We close this section with two lemmas showing that we may always restrict ourselves to probability measures with support on $[-1,1]^n$.
Therefore, throughout the article and without loss of generality we will assume any source or sink to lie in
\begin{equation*}
\prob=\{\mu\in\fbm(\R^n)\,:\,\mu(\R^n)=1,\,\spt\mu\subset[-1,1]^n\}\,.
\end{equation*}

\begin{lemma}[Mass rescaling]\label{lem:initial_and_final_measures_can_be_rescaled}
Let $\mu_+,\mu_-\in\fbm(\R^n)$ with $\mu_+(\R^n)=\mu_-(\R^n)=m$, and let $\tau$ be a transportation cost and $\flux\in\rca(\R^n;\R^n)$ a \transportPath{}.
We have
\begin{gather*}
\JEn^{\tau,\mu_+,\mu_-}[\flux]=\JEn^{\overline\tau,\overline\mu_+,\overline\mu_-}[\overline\flux]\quad\text{for}\\
\overline\mu_\pm=\mu_\pm/m\,,\;
\overline\tau(w)=\tau(mw)\,,\;
\overline\flux=\flux/m\,.
\end{gather*}
In particular, $\overline\tau$ is a valid transportation cost.
Furthermore, in \eqref{eq:functional_XiaEn} we may restrict to approximating graph sequences with $\mu_\pm^k(\R^n)=\mu_\pm(\R^n)$.
\end{lemma}

\begin{proof}
That $\overline\tau$ represents a valid transportation cost is straightforward to check.
Likewise, it is easy to see that there is a one-to-one relation between approximating graph sequences $(\mu_+^k,\mu_-^k,\flux_{G_k})\weakstarto(\mu_+,\mu_-,\flux)$
and approximating graph sequences $(\overline\mu_+^k,\overline\mu_-^k,\flux_{\overline G_k})\weakstarto(\overline\mu_+,\overline\mu_-,\overline\flux)$
via $\overline\mu_\pm^k=\mu_\pm^k/m$ and $\overline G_k=G_k/m$ (the latter means that all edge weights are divided by $m$).
Furthermore, $\JEnXia(G_k)=\JEnXia[\overline\tau](\overline G_k)$, which together with the above directly implies the first statement.

As for the last statement, consider an approximating graph sequence $(\mu_+^k,\mu_-^k,\flux_{G_k})\weakstarto(\mu_+,\mu_-,\flux)$ and set $\lambda_k=\mu_\pm(\R^n)/\mu_\pm^k(\R^n)$.
Due to the continuity of $\mu_\pm^k(\R^n)$ with respect to weak-$*$ convergence we have $\lambda_k\to1$ as $k\to\infty$.
Now it is straightforward to see that another valid approximating graph sequence is obtained as
$(\tilde\mu_+^k,\tilde\mu_-^k,\tilde G_k)=(\lambda_k\mu_+^k,\lambda_k\mu_-^k,\lambda_k\flux_{G_k})$ if $\lambda_k\leq1$ and $(\tilde\mu_+^k,\tilde\mu_-^k,\tilde G_k)=(\mu_+^k+(1-\lambda_k)\delta_0,\mu_-^k+(1-\lambda_k)\delta_0,\flux_{G_k})$ else.
However, $\liminf_{k \to \infty} \JEnXia(\tilde G_k)\leq\liminf_{k \to \infty} \JEnXia(G_k)$.
\end{proof}

\begin{lemma}[Domain rescaling]\label{thm:domainRescaling}
Let $\mu_+,\mu_-\in\fbm(\R^n)$ with $\spt\mu_+,\spt\mu_-\subset[-s,s]^n$, and let $\tau$ be a transportation cost and $\flux\in\rca(\R^n;\R^n)$ a \transportPath{}.
We have
\begin{gather*}
\JEn^{\tau,\mu_+,\mu_-}[\flux]=s\JEn^{\tau,\overline\mu_+,\overline\mu_-}[\overline\flux]\quad\text{for}\quad
\overline\mu_\pm=\pushforward{(\tfrac1s\id)}{\mu_\pm}\,,\;
\overline\flux=\pushforward{(\tfrac1s\id)}{\flux}\,.
\end{gather*}
Furthermore, in \eqref{eq:functional_XiaEn} we may restrict to approximating graph sequences with $\spt\mu_\pm^k,V(G_k)\subset[-s,s]^n$.
\end{lemma}
\begin{proof}
Again there is a one-to-one relation between approximating graph sequences $(\mu_+^k,\mu_-^k,\flux_{G_k})\weakstarto(\mu_+,\mu_-,\flux)$
and approximating graph sequences $(\overline\mu_+^k,\overline\mu_-^k,\flux_{\overline G_k})\weakstarto(\overline\mu_+,\overline\mu_-,\overline\flux)$
via $\overline\mu_\pm^k=\pushforward{(\tfrac1s\id)}{\mu_\pm^k}$ and $\overline G_k=\pushforward{(\tfrac1s\id)}{G_k}$ (the latter means that all vertex coordinates and edges are rescaled by $\frac1s$).
Furthermore, $\JEnXia(G_k)=s\JEnXia(\overline G_k)$, which implies the first statement.

As for the last statement, it is straightforward to see that for any approximating graph sequence $(\mu_+^k,\mu_-^k,G_k)\weakstarto(\mu_+,\mu_-,\flux)$
we may project the Dirac locations of $\mu_\pm^k$ and the vertices of $G_k$ orthogonally onto $[-s,s]$, resulting in a modified approximating graph sequence with non-greater cost.
Indeed, the edge lengths (and thus also the cost functional) are at most decreased, and $\mu_\pm^k\weakstarto\mu_\pm$ still holds after the modification.
\end{proof}

\subsection{Existence of minimizers and their properties}

We will see that under certain growth conditions there will always be an optimal \transportPath{} between any two measures $\mu_+,\mu_-\in\fbm(\R^n)$ with bounded support.
To this end we first show that optimal discrete \transportPaths{} never have cycles.

\begin{lemma}[Acyclicity of discrete \transportPaths{}]\label{thm:no_cycles_lemma}
For any discrete \transportPath{} $G$ there exists an acyclic discrete \transportPath{} $G_\lambda$ with same initial and final measure and $\JEnXia(G_\lambda)\leq\JEnXia(G)$.
\end{lemma}
\begin{proof}
Suppose that there is a single cycle $L\subset E(G)$, that is, $L=\{e_1,\ldots,e_l\}$ for a sequence $e_1,\ldots,e_l$ of edges with positive weight such that $e_k^-=e_{k+1}^+$, $k=1,\ldots,l-1$, and $e_l^-=e_1^+$\notinclude{ loop of edges with consistent direction and positive weight}.
For $\lambda=\min\{w(e) : e \in L\}$ consider the graph $G_\lambda$ whose edge weights are given by
\begin{displaymath}
 w_\lambda(e) = \begin{cases}
                 w(e) & e \notin L,\\
                 w(e)-\lambda & e \in L.
                \end{cases}
\end{displaymath}
Note that the initial and final measure of $G_\lambda$ are the same as of $G$ and that $G_\lambda$ no longer contains a cycle, since one edge in $L$ has weight $0$ and can thus be removed.
By the monotonicity of $\tau$ we have $\tau(w-\lambda)\leq\tau(w)$ so that
\begin{equation*}
 \JEnXia(G_\lambda)-\JEnXia(G)
 = \sum_{e \in L} [\tau(w(e)-\lambda)-\tau(w(e))]l(e)
 \leq0\,.
\end{equation*}
In case of multiple cycles we just repeat this procedure until all cycles are removed.
\end{proof}

For completeness, let us at this point also prove a stronger property of optimal discrete \transportPaths{} in case of concave transportation costs $\tau$, namely their tree structure.
By tree we shall here understand an acyclic directed graph such that from any vertex $x$ to any other vertex $y$ there exists at most one path $\{e_1,\ldots,e_l\}\subset E(G)$ with $e_1^+=x$, $e_l^-=y$, and $e_k^-=e_{k+1}^+$ for $k=1,\ldots,l-1$\notinclude{consistent with the edge orientations}.
Note that with this convention a tree may be composed of multiple disjoint trees.
\notinclude{Note that the trees do not have to be trunk trees!}

\begin{lemma}[Tree structure of discrete \transportPaths{}]\label{thm:treeStructure}
For any discrete \transportPath{} $G$ and concave transportation cost $\tau$ there exists a tree $G_\lambda$ with same initial and final measure and $\JEnXia(G_\lambda)\leq\JEnXia(G)$.
\end{lemma}

\begin{proof}
Suppose that there is a subset $L\subset E(G)$ that forms a loop (not necessarily with consistent edge orientation, that is, two neighbouring edges may also meet at both their end points or their starting points), and choose an orientation.
Let $L_+\subset L$ be the subset of edges with same orientation and $L_-=L\setminus L_+$.
Assume that the loop orientation was chosen so that $\sum_{e \in L_+}\tau'(w(e))l(e) \leq \sum_{e \in L_-}\tau'(w(e))l(e)$ (else reverse the orientation),
where $\tau'$ shall denote an element of the supergradient of $\tau$.
Next, for $\lambda=\min\{w(e) : e \in L_-\}$ consider the graph $G_\lambda$ whose multiplicity is given by
\begin{displaymath}
 w_\lambda(e) = \begin{cases}
                 w(e) & e \notin L,\\
                 w(e)+\lambda & e \in L_+,\\
                 w(e)-\lambda & e \in L_-.
                \end{cases}
\end{displaymath}
Note that the initial and final measure of $G_\lambda$ are the same as of $G$ and that $G_\lambda$ no longer contains the loop, since one edge in $L_-$ has weight $0$ and can thus be removed.
By the concavity of $\tau$ we have $\tau(w\pm\lambda)\leq\tau(w)\pm\tau'(w)\lambda$ so that
\begin{multline*}
 \JEnXia(G_\lambda)-\JEnXia(G)
 = \sum_{e \in L_+} [\tau(w(e)+\lambda)-\tau(w(e))]l(e) + \sum_{e \in L_-} [\tau(w(e)-\lambda)-\tau(w(e))]l(e)\\
 \leq\lambda\left(\sum_{e \in L_+}\tau'(w(e))l(e) - \sum_{e \in L_-}\tau'(w(e))l(e)\right)\leq0\,.
\end{multline*}
In case of multiple loops we just repeat this procedure until all loops are removed so that the resulting graph has a tree structure.
\end{proof}

\begin{remark}[Strict concavity]
If $\tau$ is strictly concave, the same proof shows that every optimal discrete \transportPath{} must have a tree structure.
\end{remark}


\begin{remark}[Necessity of concavity]\label{rem:notree}
If $\tau$ is not concave, \cref{thm:treeStructure} is false, and optimal discrete \transportPaths{} may not have a tree structure.
Indeed, for $\delta>0,l>2$ and $a\in(0,1)$ let
\begin{equation*}
\mu_+=a\delta_{(0,0)}+(1-a)\delta_{(0,1)}\,,\qquad
\mu_-=a\delta_{(l,0)}+(1-a)\delta_{(l,1)}\,,
\end{equation*}
as well as $\tau(w)=\lceil\frac w\delta\rceil$
as illustrated in \cref{fig:nontree}.
Note that we choose $\delta$ and $a$ such that $\tau(a)=\tau(1-a)-\delta$ and there is $\varepsilon<\delta$ with $\tau(a+\varepsilon)=\tau(a)=\tau(1-a-\varepsilon)$ (\cref{fig:nontree} right).
In that case, only two tree topologies are possible, displayed in \cref{fig:nontree} left (note that in principle a third possible tree topology exists which---ignoring edge directions---looks like $G_2$ rotated by $\frac\pi2$, however, its central edge would necessarily have zero weight so that this topology would be equivalent to $G_1$).
The first one has cost $\JEnXia(G_1)=l\tau(a)+l\tau(1-a)$, while the second one has larger cost $\JEnXia(G_2)>\JEnXia(G_1)$ if $\delta$ is small enough due to its longer edges and $\tau(1)\geq\tau(a)+\tau(1-a)-\delta$.
However, the nontree discrete \transportPath{} $G_3$ has the strictly smaller cost
$\JEnXia(G_3)=2\delta+l\tau(a+\varepsilon)+l\tau(1-a-\varepsilon)=2\delta+l\tau(a)+l\tau(1-a)-l\delta<\JEnXia(G_1)$.
\end{remark}

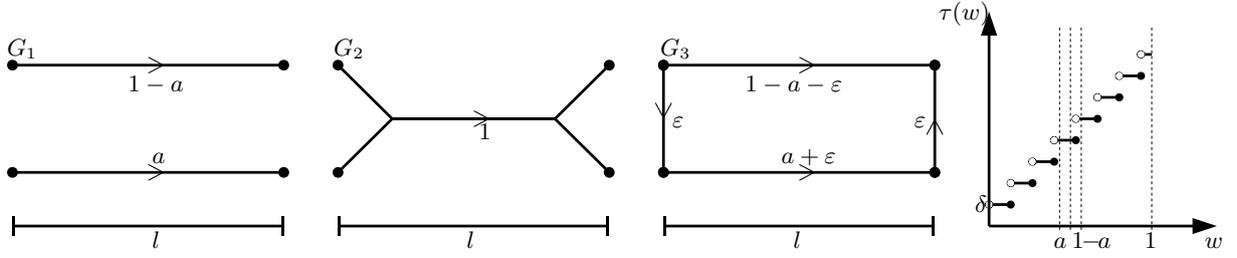
\begin{figure}
\centering
\setlength\unitlength\linewidth
\begin{picture}(1,.15)
\put(0,0){\includegraphics[width=\unitlength]{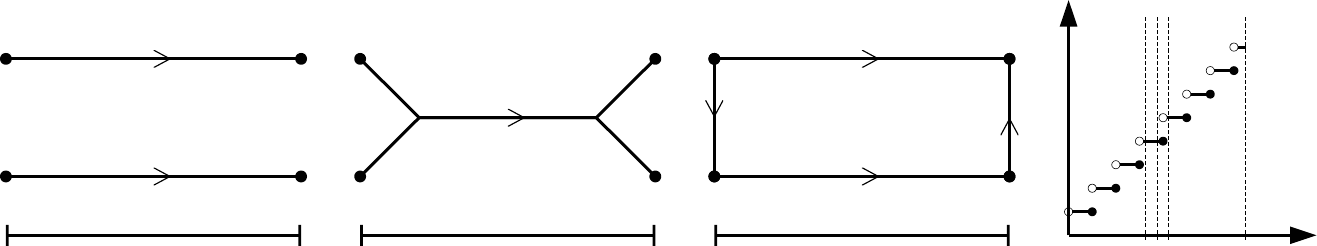}}
\put(.12,-.01){\small$l$}
\put(.38,-.01){\small$l$}
\put(.65,-.01){\small$l$}
\put(.8,.02){\small$\delta$}
\put(.12,.06){\small$a$}
\put(.1,.12){\small$1-a$}
\put(.39,.08){\small$1$}
\put(.64,.06){\small$a+\varepsilon$}
\put(.61,.12){\small$1-a-\varepsilon$}
\put(.55,.09){\small$\varepsilon$}
\put(.75,.09){\small$\varepsilon$}
\put(.77,.18){\small$\tau(w)$}
\put(.99,-.01){\small$w$}
\put(.94,-.01){\small$1$}
\put(.865,-.01){\small$a$}
\put(.88,-.01){\small$1\!\!-\!\!a$}
\put(0,.15){\small$G_1$}
\put(.27,.15){\small$G_2$}
\put(.54,.15){\small$G_3$}
\end{picture}
\caption{Illustration of the counterexample from \cref{rem:notree}. The non-tree graph has the smallest cost.}
\label{fig:nontree}
\end{figure}

As a consequence of the above, the mass flux through each edge of an optimal discrete \transportPath{} can be bounded above.

\begin{lemma}[Maximal mass flux]\label{thm:bounded_maximal_flux}
Let $G$ be an acyclic discrete \transportPath{} between $\mu_+$ and $\mu_-$. Then $w(e)\leq\mu_+(\R^n)$ for all $e\in E(G)$.
\end{lemma}

\begin{proof}
Define the set $E_0\subset E(G)$ of edges $e$ emanating from a vertex $x\in\spt\mu_+$ without influx, $e^+=x$.
$E_0$ is nonempty since otherwise for every source point $x\in\spt\mu_+$ one could find some source point $y(x)\in\spt\mu_+$ from where mass flows to $x$.
Since $\spt\mu_+$ is finite, the sequence $x,y(x),y(y(x)),\ldots$ must contain at least one vertex multiple times so that the graph would have a cycle.
Now inductively define $E_i$, $i=1,2,\ldots$, as follows.
Given $E_i$, we seek a vertex $v\in V(G)$ such that all incoming edges to $v$ are in $E_i$.
All those edges we replace by the outgoing edges of $v$ to obtain $E_{i+1}$.
It is straightforward to show by induction that each edge lies in at least one $E_i$
and that the total flux through the edges is bounded by $\sum_{e\in E_i}w(e)\leq1$ for all $i$.
\end{proof}

Now we are in a position to show that either the transport cost between given $\mu_+,\mu_-\in\fbm(\R^n)$ is infinite, or a minimizer exists.

\begin{theorem}[Existence]\label{thm:existence_of_minimizers_fluxes}
Given $\mu_+,\mu_-\in\fbm(\R^n)$ with bounded support, the minimization problem
\begin{displaymath}
 \min_\flux \JEn^{\tau,\mu_+,\mu_-}[\flux]
\end{displaymath}
either has a solution, or $\JEn^{\tau,\mu_+,\mu_-}$ is infinite.
\end{theorem}

\begin{proof}
By \cref{lem:initial_and_final_measures_can_be_rescaled} and \cref{thm:domainRescaling} we may assume $\mu_+,\mu_-\in\prob$.
Let $\flux_i\in\rca(\R^n;\R^n)$, $i=1,2,\ldots$, be a minimizing sequence with $\JEn^{\tau,\mu_+,\mu_-}[\flux_i]\to\inf_\flux \JEn^{\tau,\mu_+,\mu_-}[\flux]$, and assume the infimum cost to be finite (else there is nothing to show).
By \cref{def:mass_fluxes} there exists a triple of measures and a discrete \transportPath{} $(\mu_+^i,\mu_-^i,G_i)$ such that
\begin{equation*}
 \JEnXia(G_i) \leq \JEnXia(\flux_i) + 2^{-i}\,,\qquad
 \Wdone(\mu_+^i,\mu_+) + \Wdone(\mu_-^i,\mu_-) \leq 2^{-i}\,,\qquad
 \dv\flux_{G_i}=\mu_+^i-\mu_-^i\,.
\end{equation*}
Thanks to \cref{thm:no_cycles_lemma} we can modify the $G_i$ to become acyclic without violating the above properties.
Due to \cref{thm:bounded_maximal_flux} we know $w(e)\leq\mu_+(\R^n)=1$ for every edge $e \in E(G_i)$
so that by \cref{thm:averageCost} we have
$
\tau(w(e))\geq\lambda^\tau(1) w(e).
$
Therefore we obtain that
\begin{equation*}
|\flux_{G_i}|(\R^n)=\sum_{e\in E(G_i)}w(e)l(e)\leq\tfrac1{\lambda^\tau(1)}\sum_{e\in E(G_i)}\tau(w(e))l(e)=\tfrac1{\lambda^\tau(1)}\JEnXia(G_i)
\end{equation*}
is uniformly bounded.
Furthermore, by \cref{thm:domainRescaling} we may assume the $G_i$ or $\flux_{G_i}$ to lie inside $[-1,1]^n$.
Thus, we can extract a \mbox{weakly-*} converging subsequence (still indexed by $i$ for simplicity) so that we have $(\mu_+^i,\mu_-^i,\flux_{G_i})\weakstarto(\mu_+,\mu_-,\flux)$ for some $\flux\in\rca(\R^n;\R^n)$ with $\dv\flux=\mu_+-\mu_-$ and
\begin{equation*}
\JEn^{\tau,\mu_+,\mu_-}[\flux]=\JEnXia(\flux)\leq\liminf_{i\to\infty}\JEnXia(G_i)=\inf_{\tilde\flux}\JEn^{\tau,\mu_+,\mu_-}[\tilde\flux]\,.\qedhere
\end{equation*}
\end{proof}

Under certain growth conditions on $\tau$ (depending on the space dimension $n$) one can always guarantee the existence of a finite cost \transportPath{} and thus existence of minimizers.
We will call the corresponding transportation costs admissible.

\begin{definition}[Admissible transportation costs]\label{def:admissible}
A transportation cost $\tau$ is called \emph{admissible}, if it is bounded above by a concave function $\beta:[0,\infty)\to[0,\infty)$ with $\int_0^1\frac{\beta(w)}{w^{2-1/n}}\,\de w<\infty$.
\end{definition}

\begin{remark}[Invariance under mass rescaling]
The definition of admissibility is invariant under the transformation $\tau\mapsto\overline\tau$ from \cref{lem:initial_and_final_measures_can_be_rescaled}
and thus independent of the total mass of sources and sinks.
\end{remark}

\begin{remark}[Continuity]
Obviously, admissible transportation costs are continuous in 0 and thus automatically continuous everywhere by \cite[Thm.\,16.2.1]{Ku09}.
\end{remark}

\begin{example}[Admissible transportation costs]
\begin{enumerate}
\item The Wasserstein cost and the urban planning cost are admissible.
\item The branched transport cost $\tau(w)=w^\alpha$ is admissible for $\alpha>1-\frac1n$.
\item The transportation cost $\tau(w)=\frac{w^{1-1/n}}{|\log w|^\gamma}$ with $\gamma>1$ is admissible.
\end{enumerate}
\end{example}

For proving existence of finite cost \transportPaths{} we will need to express the admissibility in a different, less compact form.

\begin{lemma}[Admissible transportation costs]\label{thm:admTau}
A transportation cost $\tau$ is admissible if and only if it is bounded above by a concave function $\beta:[0,\infty)\to[0,\infty)$ with
\begin{equation*}
S^\beta(n) = \sum_{k = 1}^{\infty} S^\beta(n,k) < \infty\,,\,where\,
S^\beta(n,k) = 2^{(n-1)k}\beta\left(2^{-nk}\right)\,.
\end{equation*}
\end{lemma}
\begin{proof}

We need to show that $S^\beta(n)<\infty$ is equivalent to $\int_0^1\frac{\beta(w)}{w^{2-1/n}}\,\de w<\infty$ for any concave function $\beta:[0,\infty)\to[0,\infty)$.
To this end we first note that such a function $\beta$ must be non-decreasing (as so must be any nonnegative concave function on the positive halfline)
and use this to show that $S^\beta(n)<\infty$ is equivalent to $\int_0^\infty2^{(n-1)x}\beta\left(2^{-nx}\right)\,\de x<\infty$.
Indeed, this follows from

\begin{multline*}
2^{1-n}S^\beta(n)
\leq2^{1-n}\sum_{k=1}^{\infty}\int_{k-1}^{k}2^{(n-1)(x+1)}\beta\left(2^{-nx}\right)\,\de x
=\int_0^\infty2^{(n-1)x}\beta\left(2^{-nx}\right)\,\de x\\
=2^{n-1}\sum_{k=0}^{\infty}\int_{k}^{k+1}2^{(n-1)(x-1)}\beta\left(2^{-nx}\right)\,\de x
\leq2^{n-1}S^\beta(n)+2^0\beta(2^0)\,.
\end{multline*}
Now the change of variables $w=2^{-nx}$ yields 
\begin{equation*}
\int_0^\infty2^{(n-1)x}\beta\left(2^{-nx}\right)\,\de x
=-\frac{1}{n\log 2}\int_1^0\frac{\beta(w)}{w^{2-1/n}}\,\de w
\end{equation*}
so that $S^\beta(n)<\infty$ if and only if the right-hand side is finite, as desired.
\end{proof}


We will prove existence of finite cost networks by construction using the following components.
\begin{definition}[$n$-adic \transportPaths{}]\label{def:nadicGraph}
For a given measure $\mu\in\fbm(\R^n)$ we define the following.
\begin{enumerate}
\item An \emph{elementary $n$-adic \transportPath{}} for $\mu$ of scale $s$, centred at $x$, is defined as $G_{\mu,s,x}$ with
\begin{align*}
V(G_{\mu,s,x})&=\textstyle\{x\}\cup\{x+v\,:\,v\in\{-s,s\}^n\}\,,\qquad\\
E(G_{\mu,s,x})&=\textstyle\{[x,x+v]\,:\,v\in\{-s,s\}^n\}\,,\qquad\\
w([x,x+v])&=\mu(x+v+(-s,s]^n)\quad\text{for every }v\in\{-s,s\}^n\,,
\end{align*}
where $[a,b]$ denotes the straight edge from $a\in\R^n$ to $b\in\R^n$.
\item A \emph{$n$-adic \transportPath{}} for $\mu$ of $k$ levels and scale $s$, centred at $x$, is defined inductively as
\begin{equation*}
G_{\mu,s,x}^1=G_{\mu,s,x}\,,\qquad
G_{\mu,s,x}^{k}=G_{\mu,s,x}\cup\bigcup_{v\in\{-s,s\}^n}G_{\mu,s/2,x+v}^{k-1}\,,
\end{equation*}
where the union of graphs is obtained by taking the union of all vertices and all wheighted directed edges.
We will write $G_\mu^k=G_{\mu,1,0}^{k}$.
The \emph{$k$\textsuperscript{th} level} of $G_\mu^{k}$ is defined as the graph
\begin{equation*}
F_\mu^{k}=G_\mu^{k}\setminus G_\mu^{k-1}
\text{ if }k>1\text{ and }F_{\mu}^{k}=G_{\mu}^{1}\text{ else}\,.
\end{equation*}
\item Let $L(G_{\mu}^{k})=\{2^{1-k}v\,:\,v\in\{-2^{k}+1,-2^{k}+3,\ldots,2^k-3,2^k-1\}^n\}$ denote the leaves of $G_{\mu}^{k}$.
The \emph{$k$-level approximation} of $\mu$ is defined as
\begin{equation*}
P^k(\mu)=\sum_{v\in L(G_{\mu}^{k})}\mu(v+(-2^{1-k},2^{1-k}]^n)\delta_{v}\,.
\end{equation*}
\end{enumerate}
\end{definition}

An illustration of the \transportPaths{} in two dimensions is provided in \cref{fig:nadicGraph}.
It is straightforward to see that $F_{\mu}^{k}$ is a discrete \transportPath{} between $P^{k-1}(\mu)$ and $P^k(\mu)$.
Likewise, the $n$-adic \transportPath{} $G_{\mu}^k$ is a discrete \transportPath{} between $P^0(\mu)=\mu((-2,2]^n)\delta_0$ and $P^k(\mu)$.

\begin{figure}
\centering
\setlength\unitlength{.4\linewidth}
\begin{picture}(1,.4)
\put(0,0){\includegraphics[width=\unitlength]{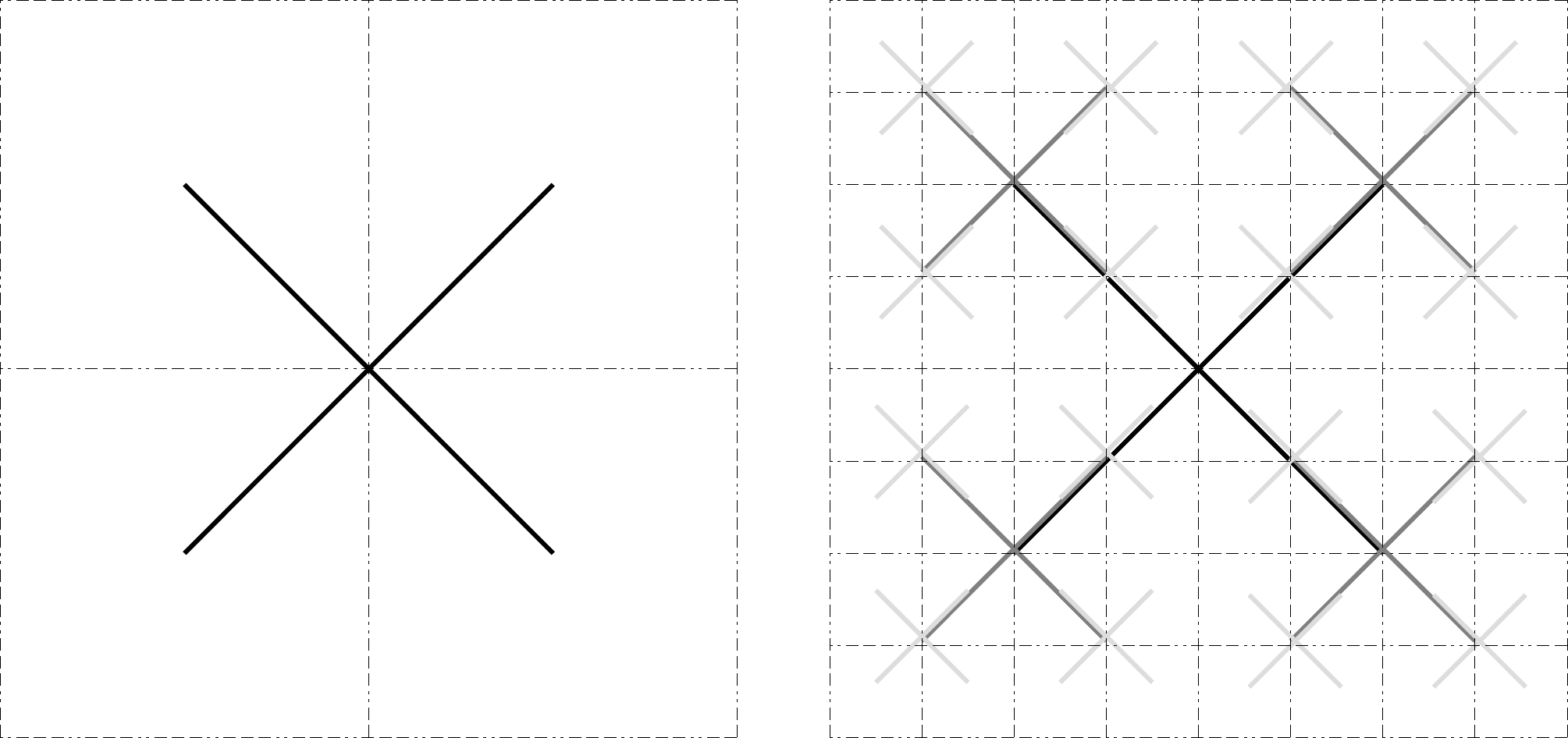}}
\put(.1,.01){$2s$}
\put(.22,.25){$x$}
\put(.75,.25){$x$}
\end{picture}
\caption{Sketch of the elementary dyadic \transportPath{} $G_{\mu,s,x}$ (left) and of the dyadic graph $G_{\mu}^3$ (right).
The flux through the shortest (lowest level) edges equals the mass of $\mu$ inside the little square around each edge.}
\label{fig:nadicGraph}
\end{figure}
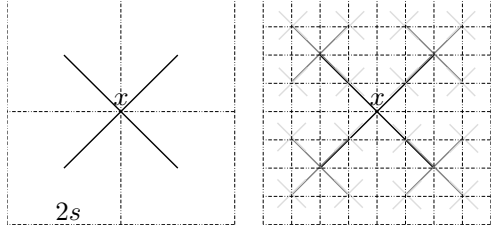

\begin{remark}[Convergence of $k$-level approximation]\label{rem:measApprox}
If $\mu$ has support inside $(-2,2]^n$, then $P^k(\mu)\weakstarto\mu$ as $k\to\infty$.
Indeed, we have
\begin{equation*}
\Wdone(\mu,P^k(\mu))\leq\mu(\R^n)2^{1-k}\sqrt n\,,
\end{equation*}
since every mass particle has to travel at most $2^{1-k}\sqrt n$.
The convergence now follows from the fact that $\Wdone$ metrizes weak-$*$ convergence.
\end{remark}

\begin{proposition}[Cost of $n$-adic \transportPath{}]\label{thm:nadicGraphCost}
If $\tau$ is an admissible transportation cost with upper bound $\beta$ and $\mu\in\prob$, then $\JEnXia(F_{\mu}^k)\leq2\sqrt nS^\beta(n,k)$ and $\JEnXia(G_{\mu}^k)\leq2\sqrt nS^\beta(n)$.
\end{proposition}
\begin{proof}
The cost of $\JEnXia(F_{\mu}^k)$ can be calculated as
\begin{equation*}
\JEnXia(F_{\mu}^k)
=\sum_{e\in E(F_{\mu}^k)}l(e)\tau(w(e))
=\sqrt n2^{1-k}\sum_{e\in E(F_{\mu}^k)}\tau(w(e))
\leq2\sqrt n2^{-k}\sum_{e\in E(F_{\mu}^k)}\beta(w(e))\,.
\end{equation*}
Since $\beta$ is concave and $\sum_{e\in E(F_{\mu}^k)}w(e)=1$ we have
\begin{equation*}
\sum_{e\in E(F_{\mu}^k)}\beta(w(e))
=2^{nk}\sum_{e\in E(F_{\mu}^k)}2^{-nk}\beta(w(e))
\leq2^{nk}\beta\left(\sum_{e\in E(F_{\mu}^k)}2^{-nk}w(e)\right)
\leq2^{nk}\beta(2^{-nk})
\end{equation*}
by Jensen's inequality, where $2^{nk}$ is the number of edges in $E(F_{\mu}^k)$.
Thus we obtain
\begin{align*}
\JEnXia(F_{\mu}^k)
&\leq2\sqrt n2^{(n-1)k}\beta(2^{-nk})
=2\sqrt nS^\beta(n,k)\quad\text{and}\\
\JEnXia(G_{\mu}^k)
&=\sum_{j=1}^k\JEnXia(F_{\mu}^j)
\leq2\sqrt nS^\beta(n)\,.\qedhere
\end{align*}
\end{proof}

We are now in a position to prove the existence of a \transportPath{} with finite cost.

\begin{corollary}[Existence of finite cost \transportPaths{}]\label{thm:finiteCostGraph}
If $\tau$ is an admissible transportation cost and $\mu_+,\mu_-\in\fbm(\R^n)$ with $\mu_+(\R^n)=\mu_-(\R^n)$ and bounded support,
then there exists a \transportPath{} $\flux$ with $\JEn^{\tau,\mu_+,\mu_-}[\flux]<\infty$.
\end{corollary}
\begin{proof}
By \cref{lem:initial_and_final_measures_can_be_rescaled} and \cref{thm:domainRescaling} we may assume $\mu_+,\mu_-\in\prob$.
Consider the sequence of graphs given by
\begin{equation*}
G^k=-G_{\mu_+}^k\cup G_{\mu_-}^k\,,\quad k=1,2,\ldots,
\end{equation*}
where $-G$ for a graph $G$ shall be the same graph with reversed edges.
Obviously, $G^k$ is a discrete \transportPath{} between the $k$-level approximations $\mu_+^k=P^k(\mu_+)$ and $\mu_-^k=P^k(\mu_-)$ of $\mu_+$ and $\mu_-$.
The total variation of $\flux_{G^k}$ is uniformly bounded,
\begin{equation*}
|\flux_{G^k}|(\R^n)
\leq\sum_{j=1}^k\sum_{e\in E(F_{\mu}^j)}l(e)w(e)
=\sum_{j=1}^k\sqrt n2^{1-j}
\leq2\sqrt n\,,
\end{equation*}
and we have $\spt\flux_{G^k}\subset[-2,2]^n$
so that a subsequence of $\flux_{G^k}$ (for simplicity still indexed by $k$) converges to some $\flux\in\rca(\R^n;\R^n)$.
Thus we have $(\mu_+^k,\mu_-^k,\flux_{G^k})\weakstarto(\mu_+,\mu_-,\flux)$ (by \cref{rem:measApprox})
and $\JEnXia(G^k)\leq\JEnXia(-G_{\mu_+}^k)+\JEnXia(G_{\mu_+}^k)\leq4\sqrt nS^\beta(n)$ (by \cref{thm:nadicGraphCost}),
which by \cref{def:costXia} implies the desired result.
\end{proof}

\begin{corollary}[Existence]\label{thm:existenceFluxAdmissible}
Given $\mu_+,\mu_-\in\fbm(\R^n)$ with $\mu_+(\R^n)=\mu_-(\R^n)$ and bounded support and an admissible $\tau$, the minimization problem
$\min_\flux \JEn^{\tau,\mu_+,\mu_-}[\flux]$
has a solution.
\end{corollary}

\subsection{Cost distance metrizes weak-$*$ convergence}

In this section we generalize \cite[Theorem 4.2]{Xia-Optimal-Paths} and show
that the cost distance $\dtau$ metrizes the weak-$*$ topology on the space of probability measures with uniformly bounded support.
First we show that up to a constant factor, $\dtau$ can be bounded below by the Wasserstein distance.

%

\begin{lemma}[Lower Wasserstein bound]\label{lem:D}
For some constant $\lambda > 0$ we have $\dtau(\mu_+,\mu_-)\geq\lambda\Wdone(\mu_+,\mu_-)$ for all probability measures $\mu_+,\mu_-$.
\end{lemma}

\begin{proof}
By \cref{def:costXia} and \cref{lem:initial_and_final_measures_can_be_rescaled} it is sufficient to show $\JEnXia(G)\geq\lambda\Wdone(\mu_+,\mu_-)$ for any discrete \transportPath{} between discrete probability measures $\mu_+$ and $\mu_-$.
By \cref{thm:bounded_maximal_flux} we can suppose $w(e)\leq1$ for each edge $e$ of $G$.
Due to \cref{thm:averageCost} we have $\tau(w)\geq\lambda^\tau(1) w$ for all $w\in[0,1]$.
Thus, denoting the Wassertein cost by $\tilde\tau(w)=\lambda^\tau(1) w$, we have
$\JEnXia(G)\geq\JEnXia[\tilde\tau](G)\geq d_{\tilde\tau}(\mu_+,\mu_-)=\lambda^\tau(1)\Wdone(\mu_+,\mu_-)$.
\end{proof}

\notinclude{\begin{corollary}
The $\dtau$ convergence implies the weak convergence.
\end{corollary}

\begin{proof}
The $\dtau$ convergence implies the convergence w.r.t. the Wasserstein distance and the latter the weak convergence (see \cite{Villani-Topics-Optimal-Transport}).
\end{proof}
}

We now prove that $\dtau$ is a distance.
The proof involves constructions using the following components (compare \cref{fig:discreteGraphs}).
In contrast to \cite{Xia-Optimal-Paths}, our construction involves smoothing the measures, thereby avoiding the need of adjusting the underlying grid of discrete measure approximations.

\begin{figure}
\setlength\unitlength\linewidth
\begin{picture}(1,.22)
\put(.016,-.029){\includegraphics[width=.234\unitlength,trim=100 0 100 0,clip]{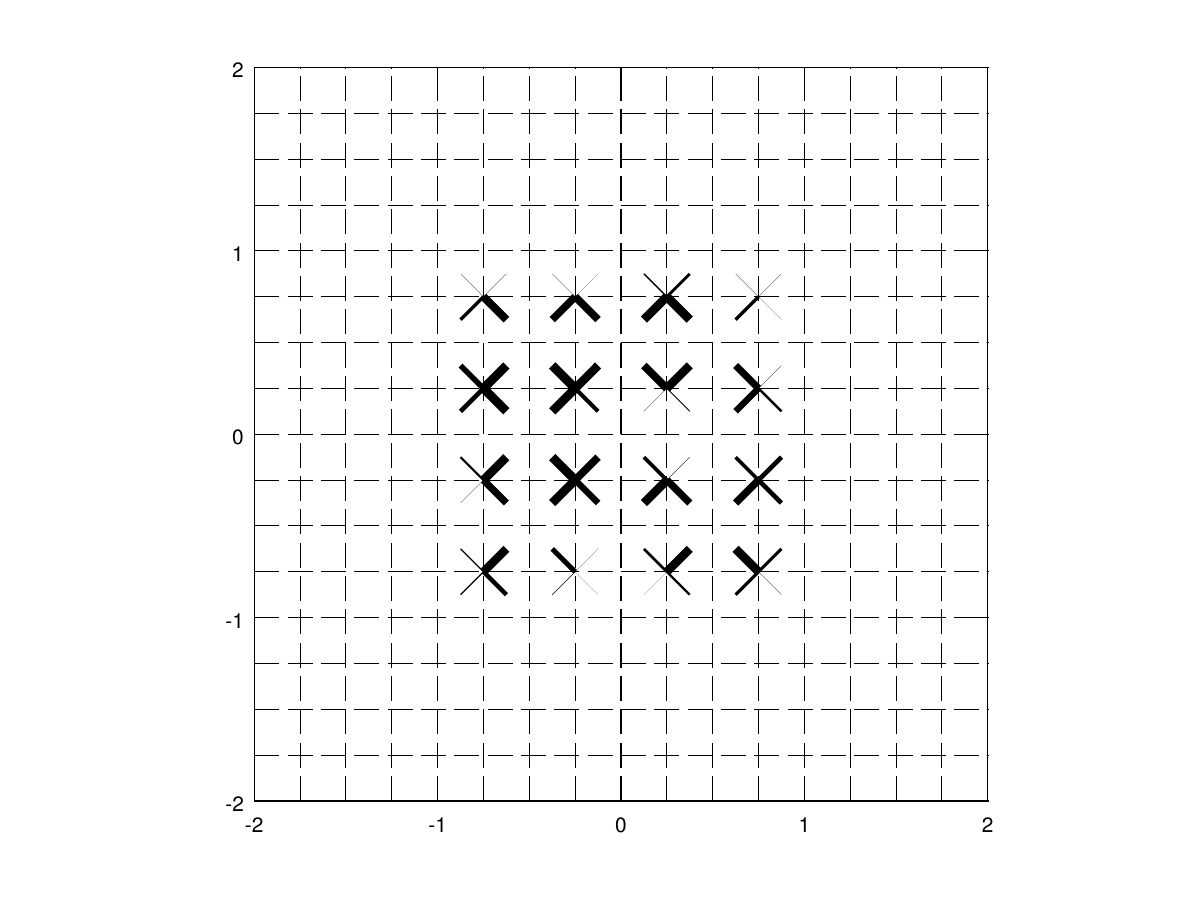}}
\put(.266,-.029){\includegraphics[width=.234\unitlength,trim=100 0 100 0,clip]{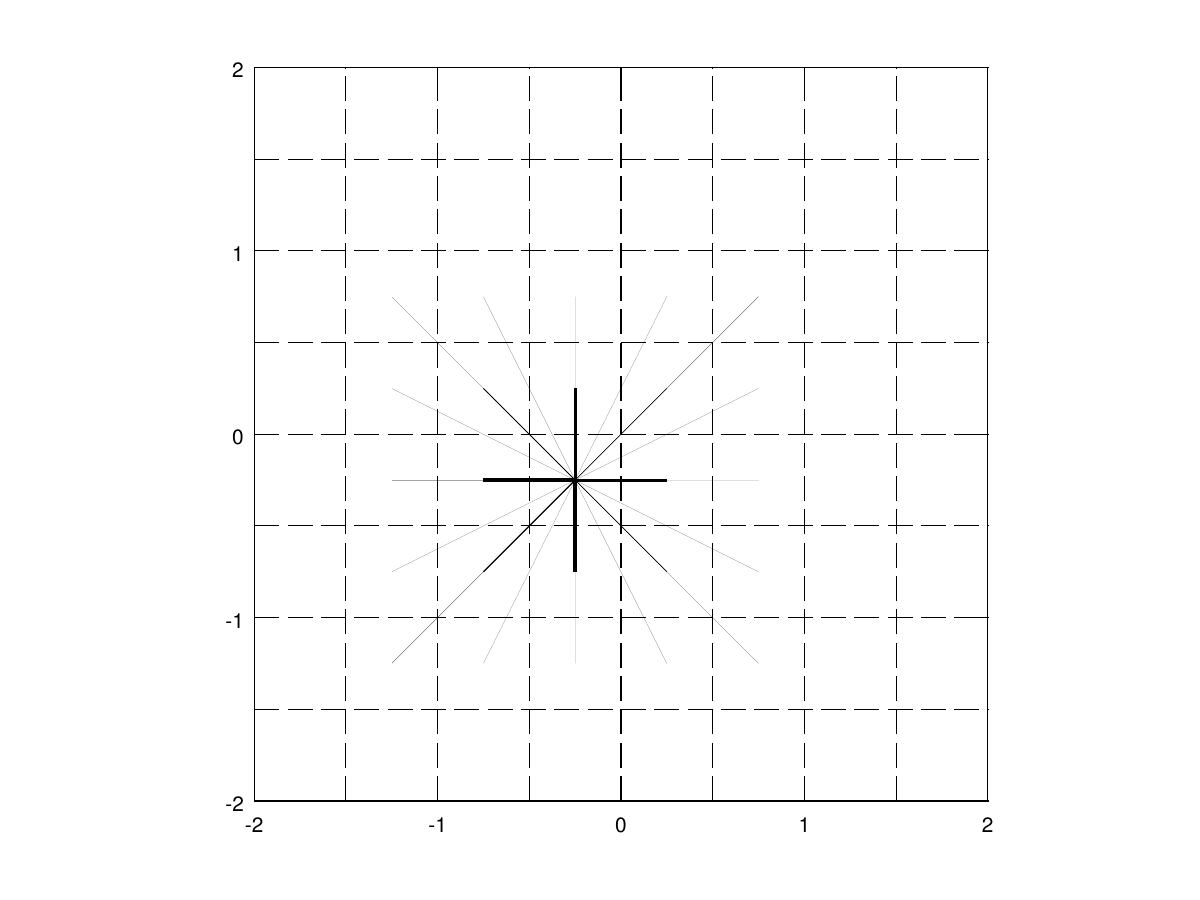}}
\put(.516,-.029){\includegraphics[width=.234\unitlength,trim=100 0 100 0,clip]{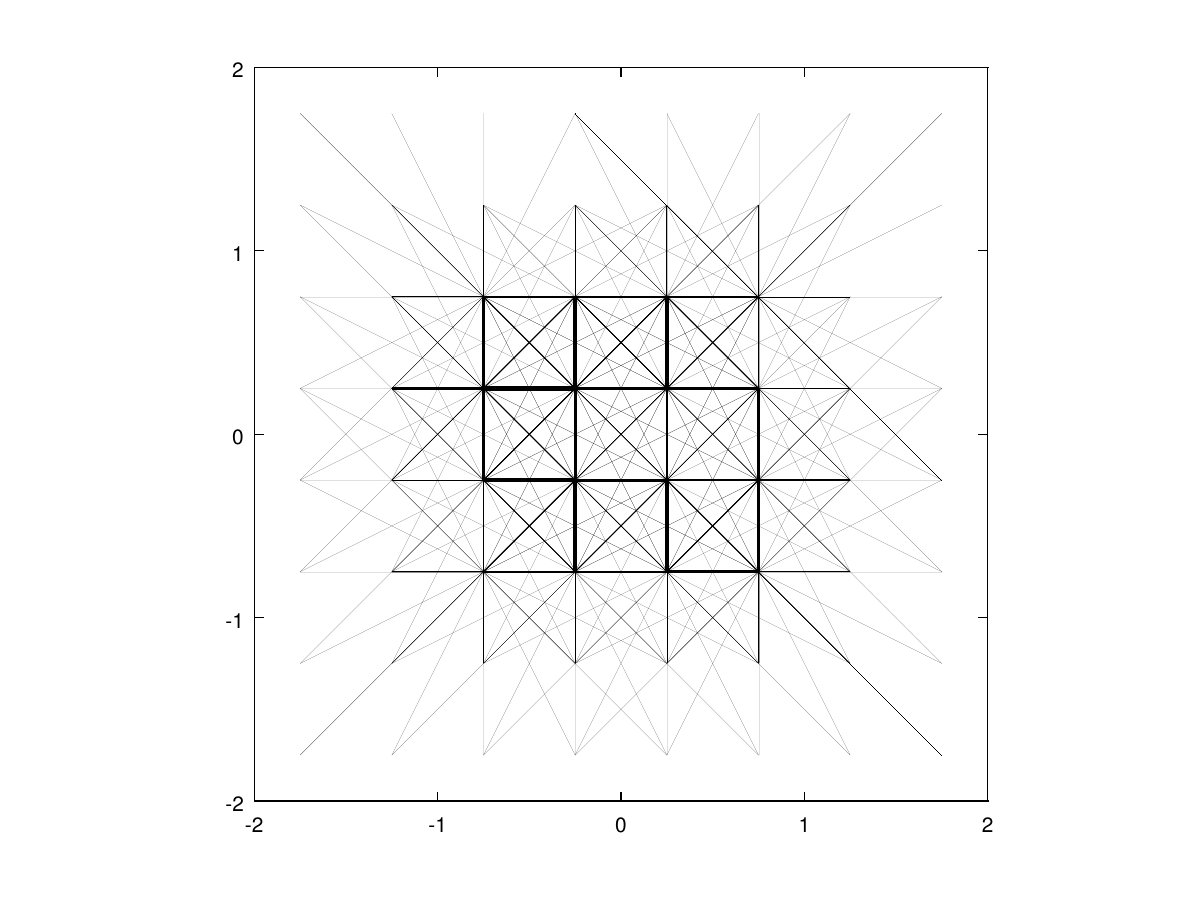}}
\put(.766,-.029){\includegraphics[width=.234\unitlength,trim=100 0 100 0,clip]{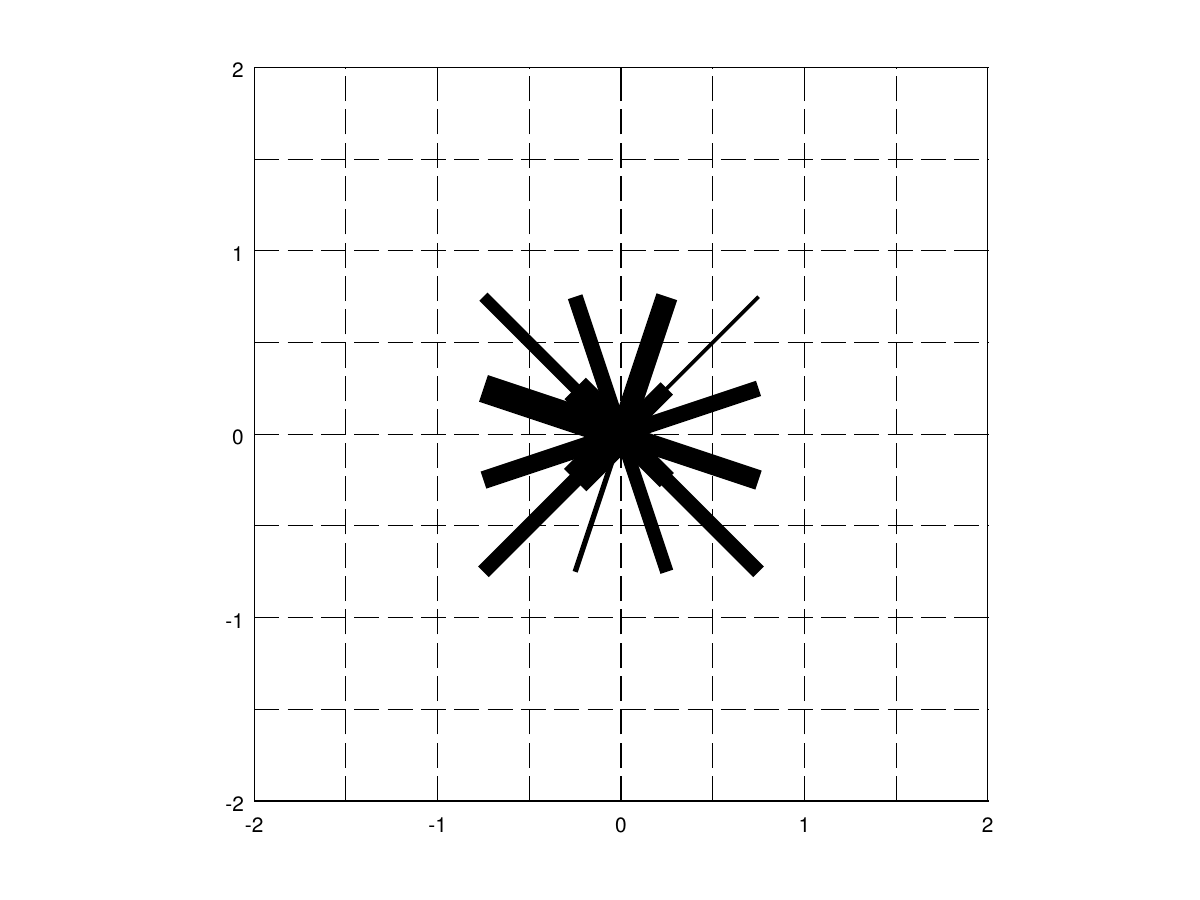}}
\put(0.03,0){\includegraphics[width=.22\unitlength]{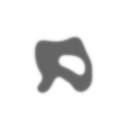}}
\put(0.28,0){\includegraphics[width=.22\unitlength]{discreteGraphsMu}}
\put(0.53,0){\includegraphics[width=.22\unitlength]{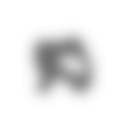}}
\put(0.78,0){\includegraphics[width=.22\unitlength]{discreteGraphsMu}}
\put(.035,.19){\colorbox{white}{$G_\mu^{3,4}$}}
\put(.285,.19){\colorbox{white}{part of $G_\mu^{3;\delta}$}}
\put(.535,.19){\colorbox{white}{$G_\mu^{3;\delta}$}}
\put(.785,.19){\colorbox{white}{$G_{P^3(\mu),\delta_0}$}}
\end{picture}
\caption{Illustration of the discrete \transportPaths{} from \cref{thm:discreteTransportPaths}.
The linethickness is proportional to the transported mass, the dashed lines just indicate the underlying grids.
All figures show the measure $\mu$ in the background except for the third, which shows $K_\delta*\mu$.}
\label{fig:discreteGraphs}
\end{figure}

\begin{proposition}[Some discrete \transportPaths{}]\label{thm:discreteTransportPaths}
Let $K_\delta$ be a convolution kernel with compact support in $B_{\frac{\delta}{3}}(0)$ (the ball of radius $\frac\delta3$ around the origin), $\delta<1$, let $\mu,\nu\in\fbm(\R^n)$ be probability measures with support in $(-2,2)^n$, and let $\tau$ be an admissible transportation cost with upper bound $\beta$.
\begin{itemize}
\item There exists a discrete \transportPath{} $G_\mu^{k,m}$, $k<m$, between $P^k(\mu)$ and $P^m(\mu)$ with cost $$\JEnXia(G_\mu^{k,m})\leq2\sqrt n\sum_{j=k+1}^mS^\beta(n,j)\,.$$
\item There exists a discrete \transportPath{} $G_\mu^{k;\delta}$ between $P^k(\mu)$ and $P^k(K_\delta*\mu)$ with cost $$\JEnXia(G_\mu^{k;\delta})\leq2^{(n+1)l_\delta+1}S^\beta(n,k+l_\delta)\,,\qquad l_\delta=\lceil\log_2(2^{k-1}\delta+2\sqrt n)/n\rceil\,.$$
\item If $\mu,\nu$ are finite discrete masses, there is a discrete \transportPath{} $G_{\mu,\nu}$ between them with cost $$\JEnXia\leq2\sqrt n\sum_{x\in\spt\mu\cup\spt\nu}\tau(|\mu-\nu|(\{x\}))\,.$$
\end{itemize}
\end{proposition}
\begin{proof}
\begin{enumerate}
\item The graph $G_\mu^{k,m}$ can obviously be chosen as $G_\mu^{k,m}=F_{\mu}^{k+1}\cup\ldots\cup F_{\mu}^{m}$ with cost $$\JEnXia(G_\mu^{k,m})=\sum_{j=k+1}^m\JEnXia(F_{\mu}^{j})=2\sqrt n\sum_{j=k+1}^mS^\beta(n,j)\,.$$
\item Any mass particle of $\mu$ is moved at most by $\delta$ during the convolution. Furthermore, the projection $P^k$ moves each particle at most by $2^{1-k}\sqrt n$.
Thus the maximum distance of the same mass particle in $P^k(\mu)$ and $P^k(K_\delta*\mu)$ is $\delta+2\cdot2^{1-k}\sqrt n$.
Thus as $G_\mu^{k;\delta}$ we may choose a subset of the fully connected graph on the grid underlying the projection $P^k$, where each mass particle travels along the edge from its location in $P^k(\mu)$ to its location in $P^k(K_\delta*\mu)$.
The number of grid points underlying the projection $P^k$ is $2^{nk}$, and since each edge in $G_\mu^{k;\delta}$ is no longer than $\delta+2^{2-k}\sqrt n$,
any grid point is connected to at most $[2(\delta+2^{2-k}\sqrt n)/2^{2-k}]^n$ other grid points, or to no more than $2^{nl_\delta}$ other grid points.
Thus, letting $w(e)$ be just the total amount of particles travelling along $e$ and using Jensen's inequality, the cost can be estimated as
\begin{multline*}
\JEnXia(G_\mu^{k;\delta})
=\sum_{e\in E(G_\mu^{k;\delta})}l(e)\tau(w(e))
\leq(\delta+2^{2-k}\sqrt n)\sum_{e\in E(G_\mu^{k;\delta})}\beta(w(e))\\
\leq(\delta+2^{2-k}\sqrt n)2^{n(k+l_\delta)}\sum_{e\in E(G_\mu^{k;\delta})}\tfrac{\beta(w(e))}{2^{n(k+l_\delta)}}
\leq(\delta+2^{2-k}\sqrt n)2^{n(k+l_\delta)}\beta\Big(\sum_{e\in E(G_\mu^{k;\delta})}\tfrac{w(e)}{2^{n(k+l_\delta)}}\Big)\\
=(2^k\delta+4\sqrt n)2^{l_\delta}S^\beta(n,k+l_\delta)
\leq2^{(n+1)l_\delta+1}S^\beta(n,k+l_\delta)\,.
\end{multline*}
\item As $G_{\mu_1,\mu_2}$ we chose the discrete \transportPath{} $G$ which transports each point mass in $\mu$ to the origin along a straight line and then moves mass from the origin to each point mass in $\nu$ along a straight line.
In detail, $V(G)=\{0\}\cup\spt\mu\cup\spt\nu$, $E(G)=\{[v,0]\,:\,v\in\spt\mu\}\cup\{[0,v]\,:\,v\in\spt\nu\}$, and $w([v,0])=\max\{0,\mu(\{v\})-\nu(\{v\})\}$, $w([0,v])=\max\{0,\nu(\{v\})-\mu(\{v\})\}$.
The cost can be calculated as
$$
\JEnXia(G)
\leq\sum_{v\in V(G)\setminus\{0\}}|v|\tau(|\mu(\{v\})-\nu(\{v\})|)
\leq2\sqrt n\sum_{x\in\spt\mu\cup\spt\nu}\tau(|\mu-\nu|(\{x\}))
\,.\qedhere
$$
\end{enumerate}
\end{proof}

\begin{corollary}[Upper Wasserstein bound]\label{thm:UpperWassersteinBound}
Let $\mu_i\weakstarto\mu$ as $i\to\infty$ for $\mu_i,\mu\in\prob$.
Then $\dtau(\mu_i,\mu)\to0$ for any admissible $\tau$.
\end{corollary}
\begin{proof}
Let $k\geq1$, $\delta>0$, and
consider the discrete \transportPath{} $G_i^m=(-G_{\mu_i}^{k,m})\cup G_{\mu_i}^{k;\delta}\cup G_{P^k(K_\delta*\mu_i),P^k(K_\delta*\mu)}\cup(-G_{\mu}^{k;\delta})\cup G_{\mu}^{k,m}$ between $P^m(\mu_i)$ and $P^m(\mu)$,
then $(P^m(\mu_i),P^m(\mu),\flux_{G_i^m})$ for $m=k+1,k+2,\ldots$ is an approximating graph sequence so that
\begin{align*}
\dtau(\mu_i,\mu)
&\leq\liminf_{m\to\infty}\JEnXia(G_i^m)\\
&\leq\liminf_{m\to\infty}\JEnXia(G_{\mu_i}^{k,m})+\JEnXia(G_{\mu_i}^{k;\delta})+\JEnXia(G_{P^k(K_\delta*\mu_i),P^k(K_\delta*\mu)})+\JEnXia(G_{\mu}^{k;\delta})+\JEnXia(G_{\mu}^{k,m})\\
&\leq4\sqrt n\sum_{j=k+1}^\infty S^\beta(n,j)+2\cdot2^{(n+1)l_\delta+1}S^\beta(n,k+l_\delta)+2\sqrt n\sum_{x\in\spt P^k(K_\delta*(\mu_i+\mu))}\tau(|P^k(K_\delta*(\mu_i-\mu))|(\{x\}))\,.
\end{align*}
Since $K_\delta*\mu_i\to K_\delta*\mu$ strongly in $\fbm(\R^n)$ and thus $P^k(K_\delta*\mu_i)-P^k(K_\delta*\mu)\to0$ strongly as $i\to\infty$, we obtain
$\limsup_{i\to\infty}\dtau(\mu_i,\mu)\leq4\sqrt n\sum_{j=k+1}^\infty S^\beta(n,j)+2\cdot2^{(n+1)l_\delta+1}S^\beta(n,k+l_\delta)$.
The result now follows by letting $\delta=2^{-k}$ so that $l_\delta\leq2$ and $k\to\infty$.
\end{proof}

Note that above the convolution was used since it turns weak into strong convergence and thus allows a simple cost estimate from above.

\begin{corollary}[Metrization property]\label{thm:metrization}
The cost distance $\dtau$ for an admissible transportation cost $\tau$ metrizes weak-$*$ convergence on $\prob$.
\end{corollary}
\begin{proof}
We have already shown that $\mu_i\weakstarto\mu$ is equivalent to $\dtau(\mu_i,\mu)=0$ (where one direction follows from \cref{thm:UpperWassersteinBound}, the other from \cref{lem:D} and the metrization of weak-$*$ convergence by $\Wdone$).
It only remains to show that $\dtau$ satisfies the triangle inequality (the symmetry is straightforward).
To this end, consider $\mu,\nu,\xi\in\prob$ as well as approximating graph sequences
\begin{align*}
(\mu_i,\xi_i^l,\flux_{G_i^l})&\quad\text{with }\dtau(\mu,\xi)=\lim_{i\to\infty}\JEnXia(G_i^l)\,,\\
(\xi_i^r,\nu_i,\flux_{G_i^r})&\quad\text{with }\dtau(\xi,\nu)=\lim_{i\to\infty}\JEnXia(G_i^r)\,.
\end{align*}
Now by \cite[Thm.\,18]{MaWi19} (which only requires \cref{thm:UpperWassersteinBound} in its proof rather than the full metrization property, even though the proof refers to the metrization property),
we may even choose $\xi_i^r=\xi_i^l$ so that $(\mu_i,\nu_i,\flux_{G_i^l\cup G_i^r})$ is an approximating graph sequence for transport from $\mu$ to $\nu$ with
\begin{equation*}
\dtau(\mu,\nu)\leq\liminf_{i\to\infty}\JEnXia(G_i^l\cup G_i^r)\leq\liminf_{i\to\infty}\JEnXia(G_i^l)+\JEnXia(G_i^r)=\dtau(\mu,\xi)+\dtau(\xi,\nu)\,.\qedhere
\end{equation*}
\end{proof}


\subsection{Length space property}

Here we show that for an admissible transportation cost $\tau$, the space $\prob$ with metric $\dtau$ is a length space, that is, the distance between any two elements $\mu_+,\mu_-\in\prob$ is induced by a shortest connecting path.
To this end we show that any $\mu_+,\mu_-\in\prob$ admit a point $\mu\in\prob$ in between, which means $\mu\neq\mu_+,\mu_-$ and $\dtau(\mu_+,\mu_-)=\dtau(\mu_+,\mu)+\dtau(\mu,\mu_-)$.
A simple construction of this point uses the following concept.

\begin{definition}[Graph paths]
Let $G$ be an acyclic discrete \transportPath{} between $\mu_+,\mu_-\in\fbm(\R^n)$.
\begin{enumerate}
\item A \emph{path} in $G$ is a sequence $\xi = (e_1,\ldots,e_k)$ of edges $e_1,\ldots,e_k\in E(G)$ such that $e_i^+ = e_{i-1}^-$ for $i = 2,\ldots,k$, where $e^+$ and $e^-$ denote the initial and final point of edge $e$.
\item A \emph{maximal path} in $G$ is a path that begins in a vertex $v_+\in\spt\mu_+$ and ends in a vertex $v_-\in\spt\mu_-$.
\notinclude{It is not necessarily a path which is not a subsequence of any other path, since a source point may also have an incoming edge.}
The set of maximal paths is denoted $\Xi(G)$.
\item The \emph{weight} $w(\xi)$ of all paths $\xi\in\Xi(G)$ is defined by the system of equations
\begin{displaymath}
 w(e) = \sum_{e \in \xi} w(\xi), \quad e \in E(G)\,,
\end{displaymath}
whose solvability follows from \cite[Lemma 7.1]{Xia-Optimal-Paths}. Notice that the solution may not be unique.
\end{enumerate}
\end{definition}

\notinclude{
\begin{example}[Measure in between]
Let $\mu_+,\mu_-\in\fbm$ with same mass and
let $A\subset\R^n$ with smooth boundary $\partial A$ and $\spt\mu_+\subset A$, $\spt\mu_-\subset\R^2\setminus A$.
Then there exists a measure $\mu$ in between $\mu_+$ and $\mu_-$, i.\,e.\ with $\mu\neq\mu_+,\mu_-$ and $\dtau(\mu_+,\mu_-)=\dtau(\mu_+,\mu)+\dtau(\mu,\mu_-)$, i.e. the equality is realised in the triangle inequality.
\end{example}
\begin{proof}
Let $G^k$ be an approximating graph sequence for $\dtau(\mu_+,\mu_-)$.
Then $G^k=G^k_1\cup G^k_2$ for graphs defined via $\flux_{G^k_1}=\flux_{G^k}\restr A$ and $\flux_{G^k_2}=\flux_{G^k}-\flux_{G^k_1}$.
Obviously, $\dv\flux_{G^k_1}=\mu_+^k-\mu^k$ and $\dv\flux_{G^k_1}=\mu^k-\mu_-^k$, where $\spt\mu^k\subset\partial A$ and $\mu_+^k\weakstarto\mu_+$, $\mu_-^k\weakstarto\mu_-$, $\mu^k\weakstarto\mu$ (the latter follows at least for a subsequence with $\mu\subset\partial A$).
Thus, $G^k_1$ and $G^k_2$ are approximating graph sequences for transport from $\mu_+$ to $\mu$ and from $\mu$ to $\mu_-$, respectively, so that $\dtau(\mu_+,\mu_-)\geq\dtau(\mu_+,\mu)+\dtau(\mu,\mu_-)$.
That this is actually an equality follows from the triangle inequality.
\end{proof}
}

\begin{proposition}[Measure in between]
Let $\mu_+,\mu_-\in\prob$ and $\tau$ be admissible, then there exists a measure $\mu$ in between $\mu_+$ and $\mu_-$ with respect to $\dtau$.
\end{proposition}
\begin{proof}
Let $\flux$ be a minimizer of $\JEn^{\tau,\mu_+,\mu_-}$ so that $\dtau(\mu_+,\mu_-) = \JEn^{\tau,\mu_+,\mu_-}[\flux]=\JEnXia(\flux)$, and let $(\mu_+^k,\mu_-^k,\flux_{G^k})\weakstarto(\mu_+,\mu_-,\flux)$ be an approximating graph sequence with $\lim_{k\to\infty}\JEnXia(G^k)=\dtau(\mu_+,\mu_-)$.
By \cref{thm:no_cycles_lemma} we may assume that $G^k$ is acyclic.
We now split each $G^k$ into a discrete \transportPath{} $G_+^k$ between $\mu_+^k$ and some $\mu^k\in\fbm(\R^n)$ as well as a discrete \transportPath{} $G_-^k$ between $\mu^k$ and $\mu_+^k$ such that $\JEnXia(G^k)=\JEnXia(G_+^k)+\JEnXia(G_-^k)$.
Furthermore, the splitting will be performed such that $\Wdone(\mu_+^k,\mu^k),\Wdone(\mu_-^k,\mu^k)\geq\frac12\Wdone(\mu_+^k,\mu_-^k)$.

Consider the set $\Xi(G^k)$ of maximal paths in $G^k$ and parameterize each path $\xi\in\Xi(G^k)$ by $\theta_\xi:[0,1]\to\R^n$
such that for any two $\xi_1,\xi_2\in\Xi(G^k)$ and any $x\in\theta_{\xi_1}([0,1])\cap\theta_{\xi_2}([0,1])$ we have $\theta_{\xi_1}^{-1}(x)=\theta_{\xi_2}^{-1}(x)$.
(This can for instance be done by first parameterizing each path with constant speed
and then assigning to each vertex $v\in V(G^k)$ the latest arrival time $t_v=\max\{\theta_{\xi}^{-1}(v)\,:\,\xi\in\Xi(G^k),v\in\xi([0,1])\}\in[0,1]$.
Now each path is reparameterized such that the parameterization along a single edge with initial point $e^+$ and end point $e^-$ is linear with $t\mapsto e^++\frac{t-t_{e^+}}{t_{e^-}-t_{e^+}}(e^--e^+)$.)
Now define for $t\in(0,1)$ the graph $G_+(t)$ that contains all edges of $G^k$ whose preimage under the parameterization lies below $t$
(if $t$ lies in the interior of an edge preimage, we split the edge at the corresponding point, thereby introducing a new vertex).
The graph $G_-(t)$ is defined as the complement $G_+(t)=G^k\setminus G_-(t)$.
%
We clearly have $\JEnXia(G^k)=\JEnXia(G_+(t))+\JEnXia(G_-(t))$.
Furthermore, defining $\mu(t)=\mu_+^k-\dv\flux_{G_+(t)}\in\fbm(\R^n)$ it follows by definition that $G_+(t)$ is a discrete \transportPath{} between $\mu_+^k$ and $\mu(t)$, while $G_-(t)$ is a discrete \transportPath{} between $\mu(t)$ and $\mu_-^k$.

Note that $\mu(0)=\mu_+^k$ and $\mu(1)=\mu_-^k$.
Since $t\mapsto\Wdone(\mu_+^k,\mu(t))$ is continuous, this implies the existence of $t^k\in\R$ with $\Wdone(\mu_+^k,\mu(t^k))=\frac12\Wdone(\mu_+^k,\mu_-^k)$.
Now set $G_\pm^k=G_\pm(t^k)$ and $\mu^k=\mu(t^k)$.
By \cref{thm:domainRescaling} we may assume $\flux_{G_\pm^k}$ and $\mu^k$ to have uniformly bounded support.
Furthermore, $|\flux_{G_-^k}|(\R^n)+|\flux_{G_+^k}|(\R^n)=|\flux_{G^k}|(\R^n)$ as well as $\mu^k(\R^n)=\mu_+^k(\R^n)$ are uniformly bounded
so that for a subsequence (again indexed by $k$) we have
\begin{equation*}
(\mu_+^k,\mu^k,\flux_{G_+^k})\weakstarto(\mu_+,\mu,\flux_+)
\quad\text{and}\quad
(\mu^k,\mu_-^k,\flux_{G_-^k})\weakstarto(\mu,\mu_-,\flux_-)\,.
\end{equation*}
Due to the weak-$*$ continuity of $\Wdone$ we have $\Wdone(\mu_+,\mu)=\frac12\Wdone(\mu_+,\mu_-)$ and $\Wdone(\mu,\mu_-)\geq\Wdone(\mu_+,\mu_-)-\Wdone(\mu_+,\mu)=\frac12\Wdone(\mu_+,\mu_-)$.
Finally,
\begin{equation*}
\dtau(\mu_+,\mu)+\dtau(\mu,\mu_-)\leq\liminf_{n\to\infty}\JEnXia(G_+^k)+\liminf_{n\to\infty}\JEnXia(G_-^k)\leq\liminf_{n\to\infty}\JEnXia(G^k)=\dtau(\mu_+,\mu_-)\,,
\end{equation*}
which together with the triangle inequality yields the desired result.
\end{proof}

\begin{corollary}[Length space property]\label{thm:lengthSpaceProperty}
The space $\prob$ with metric $\dtau$ for an admissible $\tau$ is a convex metric space.
Since it is also complete, it is a length space (see \cite[Def.\,14.1 \& Thm.\,14.1]{MR0268781} and Menger's theorem \cite[p.\,24]{MR1074005}).
\notinclude{\url{https://en.wikipedia.org/wiki/Convex_metric_space};\url{https://en.wikipedia.org/wiki/Intrinsic_metric} and examples therein}
\end{corollary}


\subsection{Explicit formula for cost function}
In this section we provide a useful explicit representation of the cost function, which generalizes \cite[Prop.\,4.4]{Xia-Interior-Regularity}.
After introducing the necessary notions, a few parts of the argument roughly follow \cite{Xia-Interior-Regularity} and \cite{CoRoMa17}, though we refer to \cite[Thm.\,8.1]{Wh99} instead of using \cite[Thm.\,2.7]{Xia-Interior-Regularity}.

Below we restrict ourselves to flat chains over $\R$ endowed with the Euclidean norm $|\cdot|$, since these provide the necessary tools.
(Note flat chains can in principle also be considered over different groups than $\R$, for instance over $\R^n$ with the group norm $v\mapsto\tau(|v|)$.)
The below definitions follow \cite[Ch.\,V.1-3]{Wh57}, \cite[Sec.\,4.1-2]{Fe69}, and \cite[\S26-27]{Si83}.

\begin{definition}[Flat chains and currents]\label{def:flatChains}
\mbox{}
\begin{enumerate}
\item An \emph{$m$-dimensional polyhedron} in $\R^n$ is a bounded, oriented, polyhedral subset of an $m$-dimensional hyperplane $H\subset\R^n$, which has nonempty relative interior.
A \emph{polyhedral $m$-chain} in $\R^n$ is an expression of the form $A=\sum_{j=1}^Na_j\sigma_j$
with $a_1,\ldots,a_N\in\R$ and $\sigma_1,\ldots,\sigma_N$ disjoint $m$-dimensional polyhedra in $\R^n$.
A refinement of $A$ is a polyhedral $m$-chain of the form $\sum_{j=1}^N\sum_{k=1}^{K_j}a_j\sigma_j^k$, where $\sigma_j=\sigma_j^1\cup\ldots\cup\sigma_j^{K_j}$ represents a disjoint partition of $\sigma_j$.
Two polyhedral $m$-chains are \emph{equivalent} and identified with each other, if they have a joint refinement.
\item The \emph{boundary} of a polyhedral $m$-chain $A=\sum_{j=1}^Na_j\sigma_j$ is a polyhedral $(m-1)$-chain defined as $\partial A=\sum_{j=1}^Na_j\partial\sigma_j$,
where $\partial\sigma_j$ is the sum of the oriented faces in the relative boundary of $\sigma_j$.
\item The \emph{mass} of a polyhedral $m$-chain $A=\sum_{j=1}^Na_j\sigma_j$ is $\mass A=\sum_{j=1}^N|a_j|\hd^m(\sigma_j)$.
\item The \emph{flat norm} of a polyhedral $m$-chain is defined as $$\flatNorm A=\inf\{\mass{A-\partial D}+\mass D\,:\,D\text{ is polyhedral }(m+1)\text{-chain}\}\,.$$
Essentially, if two chains are close with respect to $\flatNorm\cdot$, they mainly differ by a small deformation.
\item The completion of the vector space of polyhedral $m$-chains under the flat norm is the Banach space $\flatChains{m}$ of \emph{flat $m$-chains}.
The boundary of flat $m$-chains is defined by extending the linear operator $\partial$ continuously with respect to the flat norm.
\item\label{enm:inducedFunctional} Any function $\rho:\R\to\R$ \emph{induces a functional} on polyhedral $m$-chains (still denoted $\rho$) via
$\rho(\sum_{j=1}^Na_j\sigma_j)=\sum_{j=1}^Na_j\rho(a_j)\hd^m(\sigma_j)$.
In turn, any such functional on polyhedral $m$-chains \emph{induces a functional} $\rho:\flatChains{m}\to[-\infty,\infty]$ via relaxation,
that is, $\rho(A)=\inf\{\liminf_{k\to\infty}\rho(A_k)\ :\ A_k\text{ is polyhedral $m$-chain, }\flatNorm{A_k-A}\to0\text{ as }k\to\infty\}$.
Also the notion of mass is extended to $\flatChains{m}$ in this way.
\item If a flat $m$-chain $A$ has finite mass, one can define its \emph{restriction} $A\restr S$ to Borel sets $S$.
For a polygonal chain $A=\sum_{j=1}^Na_j\sigma_j$ and a hypercube $S$ this is nothing else but $A\restr S=\sum_{j=1}^Na_j(\sigma_j\cap S)$;
for general flat chains and Borel sets it is defined via a limiting procedure \cite[Sec.\,4]{Fl66}.
\item $S\subset\R^n$ is \emph{countably $m$-rectifiable} if it is contained in the countable union of Lipschitz images of $\R^m$ up to a subset of $\hd^m$ measure zero.
\item A flat 1-chain $A$ is called \emph{rectifiable} if there exists a countably $m$-rectifiable Borel set $S$ with $A=A\restr S$.
\item Let $\diffForms{m}(\R^n)$ be the set of smooth compactly supported differential forms of degree $m$ with the Fr\'echet topology induced by the seminorms $\|\cdot\|_{\cont^k}$, $k=0,1,2,\ldots$.
An \emph{$m$-dimensional current} (or $m$-current) $A$ in $\R^n$ is a continuous linear functional on $\diffForms{m}(\R^n)$.
The space of $m$-dimensional currents in $\R^n$ is denoted $\currents m(\R^n)$.
\item The \emph{boundary} of an $m$-dimensional current $A$ is the $(m-1)$-dimensional current defined via $\partial A(\psi)=A(\de\psi)$ for all $\psi\in\diffForms{m-1}(\R^n)$, where $\de\psi$ denotes the exterior derivative.
\item The \emph{mass} of $A\in\currents m(\R^n)$ is $\mass A=\sup\{A(\phi)\,:\,\phi\in\diffForms m(\R^n),\|\phi(x)\|\leq1\forall x\in\R^n\}$.
\item $A\in\currents m(\R^n)$ is \emph{representable by integration}
if there exists a Radon measure $\mu_A\in\fbm(\R^n)$ and a $\mu_A$-measurable $m$-vectorfield $\vec A:\R^n\to\Lambda_m(\R^n)$ with $\|\vec A\|=1$ $\mu_A$-almost everywhere
such that $A(\phi)=\int_{\R^n}\langle\phi(x),\vec A(x)\rangle_{\Lambda^m(\R^n),\Lambda_m(\R^n)}\,\de\mu_A$ for all $\phi\in\diffForms m(\R^n)$.
$A$ is called \emph{locally normal}, if $A$ and $\partial A$ are representable by integration.
\item $A\in\currents m(\R^n)$ is called \emph{rectifiable} if it is representable by integration
with $\mu_A=\theta\hd^m\restr S_A$ for a countably $m$-rectifiable $S_A\subset\R^n$ and a nonnegative, $\hd^m\restr S_A$-measurable scalar function $\theta$
and $\vec A(x)=v_1\wedge\ldots\wedge v_m$ for an orthonormal basis $\{v_1,\ldots,v_m\}$ of the approximate tangent space of $S_A$ at $x$.
\end{enumerate}
\end{definition}

\begin{remark}[Chains, currents, and vector measures]\label{rem:chainsRCA}\mbox{}
\begin{enumerate}
\item By the Riesz Representation Theorem, currents of finite mass and with finite mass boundary are locally normal (see also \cite[26.7]{Si83}).
\item By \cite[4.1.23 \& 4.2.23]{Fe69}, locally normal $m$-dimensional currents (after taking the closure with respect to flat convergence of currents) and flat $m$-chains can be identified with each other.
In more detail, every polyhedral $m$-chain $\sum_{j=1}^Na_j\sigma_j$ is identified with the current $\phi\mapsto\sum_{j=1}^Na_j\int_{\sigma_j}\langle\phi(x),v_1^j\wedge\ldots\wedge v_m^j\rangle\,\de\hd^m(x)$
for an orthonormal frame $v_1^j,\ldots,v_m^j$ of the oriented tangent space to $\sigma_j$.
Furthermore, the boundaries of corresponding currents and flat chains correspond to each other, and their masses coincide.
\item Identifying $\R^n=\Lambda^1(\R^n)=\Lambda_1(\R^n)$, $1$-currents $A\in\currents1(\R^n)$ which are representable by integation coincide with vector-valued Radon measures $\tilde A\in\rca(\R^n;\R^n)$.
By $\partial A(\phi)=A(\de\phi)=\tilde A(\nabla\phi)=-\dv\tilde A(\phi)$ for any $\phi\in\contsmooth(\R^n)$ (where we identified $\R=\Lambda_0(\R^n)=\Lambda^0(\R^n)$),
locally normal $1$-currents (and thus flat $1$-chains) are identical to vector-valued Radon measures of compact support whose distributional divergence is a Radon measure (see also \cite[Sec.\,5]{Si07}).
Similarly, $0$-currents $A\in\currents0(\R^n)$ which are representable by integration coincide with Radon measures $\tilde A\in\rca(\R^n;\R)$.
\item If $|\flux_j|(\R^n)$ and $|\dv\flux_j|(\R^n)$ are uniformly bounded, the weak-$*$ convergence $\flux_j\weakstarto\flux$ in $\rca(\R^n;\R^n)$ is equivalent
to the convergence in the flat norm of the corresponding $1$-currents or equivalently the corresponding flat $1$-chains
(this is essentially a version of the compactness theorem for normal currents; see also \cite[Rem.\,2.2]{Xia-Interior-Regularity} or \cite[Thm.\,31.2]{Si83} for the analogous statement on integral currents;
note, though, that flat convergence alone does not imply boundedness of the mass but always implies flat convergence of the boundary,
while weak-$*$ convergence does not imply weak-$*$ convergence of the divergence but automatically implies boundedness of the mass by the uniform boundedness principle%
\notinclude{$A_\varepsilon=\varepsilon^{-1/2}(\delta_{-\varepsilon}+\delta_\varepsilon)\to0$ in flat norm, but mass diverges; also, $\sum_{j=1}^\infty A_{1/j^3}$ is a flat chain with infinite mass; flat convergence of boundary follows from \cite[4.1.12]{Fe69}}%
).
Indeed, identifying $\R^n=\Lambda_1(\R^n)=\Lambda^1(\R^n)$ as well as $\{B\in\R^{n\times n}\,:\,B\text{ is skew}\}=\Lambda_2(\R^n)=\Lambda^2(\R^n)$, by \cite[4.1.12]{Fe69} we have
\begin{multline*}
\flatNorm{\flux_j-\flux}\to0
\quad\Leftrightarrow\quad
\sup\{(\flux_j-\flux)(\phi)\,:\,\phi\in D(K)\}\to0\;\forall K\subset\R^n\\
\quad\text{with}\quad
D(K)=\{\phi\in\contsmooth(K;\R^n)\ :\ \|\phi\|_{\cont^0}\leq1,\,\|\text{skew}D\phi\|_{\cont^0}\leq1\}
\end{multline*}
for $\text{skew}B=\frac12(B-B^T)$.
Since $\contsmooth(\R^n;\R^n)$ is dense in $\cont_c(\R^n;\R^n)$ with respect to the supremum norm,
this implies $(\flux_j-\flux)(\phi)\to0$ for all $\phi\in\cont_c(\R^n;\R^n)$ and thus $\flux_j\weakstarto\flux$.

Now let $\flux_j\weakstarto\flux$ and $|\dv\flux_j|(\R^n)$ be uniformly bounded, then any subsequence contains another subsequence along which $\dv\flux_j\weakstarto\dv\flux$
so that actually $\dv\flux_j\weakstarto\dv\flux$ for the whole sequence as well as $|\dv\flux|(\R^n)<\infty$.
Defining on $\contsmooth(\R^n;\R^n)$ the seminorm
\begin{equation*}
\|\phi\|=\inf_{\substack{\psi\in\contsmooth(\R^n;\R^n),\,\xi\in\contsmooth(\R^n;\R),\\\phi=\psi+\nabla\xi}}\|\psi\|_{\cont^0}+\|\xi\|_{C^0}\,,
\end{equation*}
we obtain for all $\tilde\flux\in\rca(\R^n;\R^n)$ and $\phi\in\contsmooth(\R^n;\R^n)$
\begin{multline*}
\tilde\flux(\phi)
=\inf_{\substack{\psi\in\contsmooth(\R^n;\R^n),\,\xi\in\contsmooth(\R^n;\R),\\\phi=\psi+\nabla\xi}}\tilde\flux(\psi)-\dv\tilde\flux(\xi)\\
\leq\inf_{\substack{\psi\in\contsmooth(\R^n;\R^n),\,\xi\in\contsmooth(\R^n;\R),\\\phi=\psi+\nabla\xi}}|\tilde\flux|(\R^n)\|\psi\|_{\cont^0}+|\dv\tilde\flux|(\R^n)\|\xi\|_{\cont^0}
\leq\|\phi\|\max\{|\tilde\flux|(\R^n)+|\dv\tilde\flux|(\R^n)\}\,.
\end{multline*}
Furthermore, for any bounded open $K\subset\R^n$, $D(K)$ is compact with respect to $\|\cdot\|$, that is,
for any $\varepsilon>0$ there is a finite number of functions $\phi_1,\ldots,\phi_{N_\varepsilon}\in D(K)$
such that for any $\phi\in D(K)$ there is some $i\in\{1,\ldots,N_\varepsilon\}$ with $\|\phi-\phi_i\|<\varepsilon$
(the proof is identical to the proof of \cite[Thm.\,1(i)]{Ki89}, replacing $L^q$ and $L^p$ with $C^0$ as well as $W^{1,p}$ with $C^{0,\alpha}$
and extending $K$ to a compact domain with periodic boundary conditions\notinclude{for necessary elliptic regularity see e.\,g.\ Gilbarg--Trudinger Thm.\,8.33}).
Now let $\varepsilon>0$ and an arbitrary $\phi\in D(K)$ be given and choose $k$ large enough that $|(\flux_j-\flux)(\phi_i)|<\varepsilon$ for all $i\in\{1,\ldots,N_\varepsilon\}$ and all $j>k$.
Then, picking $i\in\{1,\ldots,N_\varepsilon\}$ with $\|\phi-\phi_i\|<\varepsilon$,
\begin{multline*}
|(\flux_j-\flux)(\phi)|
\leq|\flux_j(\phi-\phi_i)|+|(\flux_j-\flux)(\phi_i)|+|\flux(\phi_i-\phi)|\\
\leq\max\{|\flux_j|(\R^n),|\dv\flux_j|(\R^n)\}\|\phi-\phi_i\|+\varepsilon+\max\{|\flux|(\R^n),|\dv\flux|(\R^n)\}\|\phi-\phi_i\|
\leq C\varepsilon
\end{multline*}
for all $j>k$ and a fixed constant $C>0$ independent of $\phi$ and $\varepsilon$.
Since $\varepsilon>0$ and $\phi\in D(K)$ were arbitrary, we have $\sup\{(\flux_j-\flux)(\phi)\,:\,\phi\in D(K)\}\to0$ as $j\to\infty$
and thus $\flatNorm{\flux_j-\flux}\to0$.

Likewise, weak-$*$ convergence of scalar-valued Radon measures is equivalent to flat convergence of the corresponding $0$-currents (or flat $0$-chains).
\end{enumerate}
\end{remark}

The explicit characterization of the cost $\JEnXia$ with the help of $1$-currents or flat $1$-chains will make use of the theory of rectifiable flat chains as well as of slicing,
which is a technique to reduce the dimension of a flat chain (in our case from $1$ to $0$).
Therefore, we first prove a result on flat $0$-chains to be used later in a slicing argument.
Recall that a transportation cost $\tau$ induces a functional $\tau$ on $\flatChains m$ via \cref{def:flatChains}\eqref{enm:inducedFunctional}.

\begin{definition}[Diffuse flat $0$-chain]
We shall call $A\in\flatChains0$ \emph{diffuse} if $A\restr\{x\}=0$ for all $x\in\R^n$.
\end{definition}

\begin{lemma}[Lower semi-continuous envelope on diffuse $0$-chains]\label{thm:diffuse0chainCost}
Let $\tau$ be a transportation cost and $A\in\flatChains0$ be of finite mass (note that $A$ can be identified with a Radon measure by \cref{rem:chainsRCA}) and diffuse,
then $\tau(A)=\tau'(0)\mass{A}$. Here, $\tau'(0)=\lim_{w\searrow0}\frac{\tau(w)}w$ shall denote the right derivative in $0$, which exists, but may be infinite \cite[Thm.\,16.3.3]{Ku09}.
\end{lemma}
\begin{proof}
First let $A_i$ be a polyhedral chain approaching $A$ in the flat norm such that $\mass{A_i}\to\mass{A}$ as $i\to\infty$.
Due to $\tau(w)\leq\tau'(0)w$ we have $\tau(A)\leq\liminf_{i\to\infty}\tau(A_i)\leq\liminf_{i\to\infty}\tau'(0)\mass{A_i}=\tau'(0)\mass{A}$,
so it remains to show the opposite inequality.

Now let $A_i=\sum_{k=1}^{K_i}\theta_k^i\{x_k^i\}\notinclude{=\sum_{k=1}^{K_i}\theta_k^i\delta_{x_k^i}}$ be a polyhedral chain with $\flatNorm{A_i-A}\to0$ such that $\tau(A_i)\to\tau(A)$.
By restricting to a subsequence we may assume $\sum_{i=1}^\infty\flatNorm{A_i-A}<\infty$.
For a proof by contradiction, assume $\tau(A)<\lambda\mass{A}$ for $\lambda<\tau'(0)$.
There exists $\delta>0$ with $\tau(w)>\lambda w$ for all $w\in[0,\delta]$.
Now cover most of the support of $A$ with finitely many \notinclude{open, }pairwise disjoint hypercubes $B_1,\ldots,B_N$
such that $\sum_{i=1}^N\mass{A\restr B_j}=\mass{A}-\varepsilon$ and $\mass{A\restr B_j}\leq\delta$ for $j=1,\ldots,N$.
For instance, one can partition $\R^n$ into halfopen hypercubes of side length $1$
and then repeatedly subdivide all of the hypercubes into $2^n$ hypercubes of half the width until the mass of $A$ on each hypercube is less than $\delta$
(this condition must be met after a finite number of subdivisions,
since otherwise there would be a sequence of nested hypercubes whose side lengths decrease to zero, but whose mass stays above $\delta$, implying a non-allowed Dirac mass inside them).
\notinclude{Since $A$ is diffuse, the mass inside each hypercube does not change when only considering its interior.}\phantomsection%
Now pick $N$ of these open hypercubes such that $A$ is covered up to mass $\varepsilon$.

Since $\sum_{i=1}^\infty\flatNorm{A_{i+1}-A_i}<\infty$, by \cite[Lem.\,(2.1) \& Sec.\,4]{Fl66} the sets $B_j$ can be positioned such that they are ``non-exceptional''
\notinclude{(note that the grid of hypercubes can be shifted in any direction by any real number without violating the condition;
if not, then just repeat the grid refinement process, where the grid is shifted by half a side length)}
so that $\flatNorm{A_i\restr B_j-A\restr B_j}\to0$ as $i\to\infty$ (in fact, this is how $A\restr B_j$ is defined in \cite[Sec.\,4]{Fl66}; note that Fleming uses the notation $A\cap B_j$).
Then
\begin{multline*}
\textstyle
\tau(A_i\restr B_j)
=\sum_{x_k^i\in B_j}\tau(|\theta_k^i|)
\geq\tau\left(\sum_{x_k^i\in B_j}|\theta_k^i|\right)\\
=\tau(\mass{A_i\restr B_j})
\geq\tau(\min\{\mass{A_j\restr B_j},\mass{A\restr B_j}\})
>\lambda\min\{\mass{A_j\restr B_j},\mass{A\restr B_j}\}
\end{multline*}
so that (exploiting the lower semi-continuity of the mass)
\begin{multline*}
\textstyle
\tau(A)
=\liminf_{i\to\infty}\tau(A_i)
\geq\liminf_{i\to\infty}\tau\left(A_i\restr\bigcup_{j=1}^N B_j\right)
\geq\sum_{j=1}^N\liminf_{i\to\infty}\tau(A_i\restr B_j)\\
\textstyle
>\lambda\sum_{j=1}^N\liminf_{i\to\infty}\min\{\mass{A_j\restr B_j},\mass{A\restr B_j}\}
=\lambda\sum_{j=1}^N\mass{A\restr B_j}
\geq\lambda\mass{A}-\varepsilon\,.
\end{multline*}
Since $\varepsilon$ was arbitrary, we obtain $\tau(A)\geq\lambda\mass{A}$, the desired contradiction.
\end{proof}

The main result of this section may be seen as a variant of \cite{CoRoMa17}, here only stated for $1$-dimensional currents
(in contrast, \cite{CoRoMa17} consider functionals on $m$-dimensional currents, but they restrict their analysis to rectifiable currents).
It is the following characterization of the cost functional.

\begin{proposition}[Generalized Gilbert energy]\label{thm:GilbertFlux}
%
Let $\flux$ be a mass flux of bounded support between $\mu_+,\mu_-\in\fbm(\R^n)$.
By \cite[Thm.\,5.5]{Si07} or \cite[Thm.\,4.2]{Sil08} we can write $\flux=\theta\hdone\restr S+\flux^\perp$ for some countably $1$-rectifiable $S\subset\R^n$, an $\hdone\restr S$-measurable map $\theta:S\to\R^n$ tangent to $S$ $\hdone$-almost everywhere,
and $\flux^\perp$ singular with respect to $\hdone\restr R$ for any countably $1$-rectifiable $R\subset\R^n$.
Then

\notinclude{If $\flux$ has bounded support, we have}%
\begin{equation*}
\JEn^{\tau,\mu_+,\mu_-}[\flux]
=\int_{S}\tau(|\theta(x)|)\,\de\hdone(x)+\tau'(0)|\flux^\perp|(\R^n)\,.
\end{equation*}
\notinclude{if $\dv\flux=\mu_+-\mu_-$, $|\flux|(\R^n)<\infty$, and $\flux=\theta\hdone\restr S+\flux^\perp$ for some countably $1$-rectifiable $S\subset\R^n$, an $\hdone\restr S$-measurable map $\theta:S\to\R^n$ tangent to $S$ $\hdone$-almost everywhere,
and $\flux^\perp$ singular with respect to $\hdone\restr R$ for any countably $1$-rectifiable $R\subset\R^n$.
Otherwise, $\JEn^{\tau,\mu_+,\mu_-}[\flux]=\infty$.}%
\end{proposition}
\begin{proof}
Any discrete \transportPath{} $G$ between discrete finite masses $\mu_\pm=\sum_{j=1}^{N^\pm}a_j^\pm\delta_{x_j^\pm}$ can obviously be identified
with a polyhedral chain $C_G=\sum_{e\in E(G)}w(e)e$ with boundary $\partial C_G=-\sum_{j=1}^{N^+}a_j^+\{x_j^+\}+\sum_{j=1}^{N^-}a_j^-\{x_j^-\}$.
Thus, by \cref{def:mass_fluxes}, \cref{def:costXia} and \cref{rem:chainsRCA}, a \transportPath{} $\flux$ between $\mu_\pm$ is a flat $1$-chain with boundary $\mu_--\mu_+$
(identifying $\flux$ and $\mu_\pm$ with their corresponding flat $1$- and $0$-chains, respectively).
Thus we have
\begin{equation*}
\JEn^{\tau,\mu_+,\mu_-}[\flux]=\begin{cases}\tau(\flux)&\text{if }\flux\in\flatChains1,\partial\flux=\mu_--\mu_+,\\\infty&\text{else,}\end{cases}
\end{equation*}
where $\tau$ is defined for polyhedral $1$-chains as $\tau(\sum_{j=1}^Na_j\sigma_j)=\sum_{j=1}^N\tau(|a_j|)\hdone(\sigma_j)$.
We now proceed in steps.
%
%
%
%
%
%

\begin{itemize}
\item Denote the upper $m$-dimensional density of a Radon measure $\gamma$ on $\R^n$ in a point $x\in\R^n$ by $\Theta^{*m}(\gamma,x)=\limsup_{r\searrow0}\gamma(B(x,r))/V_r^m$
for $B(x,r)$ the closed $n$-dimensional ball of radius $r$ centred at $x$ and for $V_r^m$ the volume of the $m$-dimensional ball of radius $r$.
\item Set $S=\{x\in\R^n\,:\,\Theta^{*1}(|\flux|,x)>0\}$ and $\flux^\perp=\flux\restr(\R^n\setminus S)$ (note that $S$ is Borel \cite[Prop.\,(1.1)]{Ed94}\notinclude{http://math.stackexchange.com/questions/1884629/is-the-domain-of-the-symmetric-derivative-of-a-borel-measure-a-borel-set/1885155}).
\item By \cite[Thm.\,6.1(3)]{Wh99b}, the mapping $\mu_\flux:B\mapsto\tau(\flux\restr B)$ for any Borel set $B$ is a Radon measure.
Thus we have $\tau(\flux)=\mu_\flux(\R^n)=\mu_\flux(S\cup(\R^n\setminus S))=\mu_\flux(S)+\mu_\flux(\R^n\setminus S)=\tau(\flux\restr S)+\tau(\flux^\perp)$.
\item We have $\flux^\perp\perp\hdone\restr R$ for any countably $1$-rectifiable $R\subset\R^n$:
Let $R\subset\R^n$ be countably $1$-rectifiable with $\hdone(R)<\infty$, then by \cite[2.10.19]{Fe69} and the definition of $S$
\notinclude{easier to grasp: https://www.math.washington.edu/~morrow/334_15/geoMeasure.pdf p.294}%
we have $|\flux^\perp|(R)=|\flux|(R\setminus S)\leq2t\hdone(R\setminus S)\leq2t\hdone(R)$ for any $t>0$ so that $|\flux^\perp|(R)=0$.
If $R\subset\R^n$ is countably $1$-rectifiable, but $\hdone(R)=\infty$,
then we note that $R$ can be covered by a countable union $\bigcup_{k=1}^\infty R_k$ of $1$-rectifiable sets $R_k\subset\R^n$ with $\hdone(R_k)<\infty$
so that again $|\flux^\perp|(R)\leq\sum_{k=1}^\infty|\flux^\perp|(R_k)=0$.
\item $S$ is countably $1$-rectifiable: Indeed, $S=\bigcup_{j=1}^\infty S_j$ with $S_j=\{x\in S\,:\,\Theta^{*1}(|\flux|,x)\geq\frac1j\}$,
where every $S_j$ satisfies
$$|\flux|(\R^n)\geq|\flux|(S_j)\geq\int_{S_j}\Theta^{*1}(|\flux|,x)\,\de\mathcal C^1(x)\geq\frac1j\mathcal C^1(S_j)\geq\frac1j\hdone(S_j)$$ by \cite[(3.3)(a) \& (1.3)]{Ed94},
where $\mathcal C^1$ denotes the $1$-dimensional covering outer measure.
Thus, $\flux\restr S_j$ is a flat $1$-chain of finite mass (note that $B\mapsto\mass{\flux\restr B}$ is a Borel function \cite[Sec.\,4]{Fl66}) and finite Hausdorff size $\hdone(S_j)$
and therefore rectifiable by \cite[Thm.\,4.1]{Wh99b}.
\item Obviously, $\flux\restr S$ is a rectifiable flat $1$-chain or equivalently a rectifiable $1$-dimensional current.
Consequently, it can be written as $\flux\restr S=\theta\hdone\restr S$ for some $(\hdone\restr S)$-measurable $\theta:S\to\R^n$ with $\theta(x)$ in the approximate tangent space of $S$ at $x$.
(Note that a decomposition of the form $\flux=\theta\hdone\restr S+\flux^\perp$ also follows from the structure theorem of divergence measure fields \cite[Thm.\,5.5]{Si07}.)
\item We have $\tau(\flux^\perp)=\tau'(0)|\flux^\perp|(\R^n)$: Indeed, $\tau(\flux^\perp)\leq\tau'(0)|\flux^\perp|(\R^n)$ follows directly from $\tau(w)\leq\tau'(0)w$ for all $w$.
The opposite inequality is obtained by slicing.
Let $A_i$ be a sequence of polygonal chains converging in the flat norm to $\flux^\perp$ such that $\tau(A_i)\to\tau(\flux^\perp)$.
Now denote by $Gr(n,1)$ the Grassmannian of $1$-dimensional lines in $\R^n$, by $\gamma_{n,1}$ the Haar measure on $Gr(n,1)$,
by $p_V:\R^n\to\R^n$ the orthogonal projection onto the line spanned by $V\in Gr(n,1)$,
and by $A\cap p_V^{-1}(\{x\})$ the slicing of the flat $1$-chain through $p_V^{-1}(\{x\})$
(which can be understood as the flat $0$-chain as which $A$ is observed through $(n-1)$-dimensional eyes in the $(n-1)$-dimensional world $p_V^{-1}(\{x\})$; see e.\,g.\ \cite[Sec.\,4.3]{Fe69}).
There is a constant $c=1/\int_{Gr(n,1)}|u\cdot w_V|\,\de\gamma_{n,1}(V)>0$
with $w_V$ being the unit vector spanning $V$ (we assume that a consistent orientation of $w_V$ is chosen for all $V$) and $u\in\R^n$ being any fixed unit vector (the constant is independent of the particular choice) such that
\begin{align*}
\mass{\flux^\perp}
&=c\int_{Gr(n,1)}\int_\R\mass{\flux^\perp\cap p_V^{-1}(\{y\})}\,\de\hdone(y)\de\gamma_{n,1}(V)\,,\\
\tau(A_i)
&=c\int_{Gr(n,1)}\int_\R\tau(A_i\cap p_V^{-1}(\{y\}))\,\de\hdone(y)\de\gamma_{n,1}(V)\,.
\end{align*}
The second equality is \cite[(3.2)]{CoRoMa17}, the first is obtained as follows.
Taking $v=1$ and $f=p_V$ in \cite[4.3.2(2) and following statement]{Fe69} and identifying $V$ with $\R$ we obtain
\begin{equation*}
\int_\R\mass{\flux^\perp\cap p_V^{-1}(\{y\})}\,\de\hdone(y)=|\flux^\perp\cdot w_V|(\R^n)\,.
\end{equation*}
Thus, with Fubini we have
\begin{multline*}
\int_{Gr(n,1)}\int_\R\mass{\flux^\perp\cap p_V^{-1}(\{y\})}\,\de\hdone(y)\de\gamma_{n,1}(V)
=\int_{Gr(n,1)}|\flux^\perp\cdot w_V|(\R^n)\de\gamma_{n,1}(V)\\
=\int_{\R^n}\int_{Gr(n,1)}|\tfrac{\partial\flux^\perp}{\partial|\flux^\perp|}\cdot w_V|\,\de\gamma_{n,1}(V)\de|\flux^\perp|
=\int_{\R^n}\tfrac1c\,\de|\flux^\perp|
=\tfrac1c\mass{\flux^\perp}\,.
\end{multline*}
Now $\flux^\perp\cap p_V^{-1}(\{y\})$ is diffuse for almost every $y\in\R$ (that is, it is singular with respect to any Dirac mass).
Furthermore note that (potentially after choosing a subsequence)
$\flatNorm{A_i\cap p_V^{-1}(\{y\})-\flux^\perp\cap p_V^{-1}(\{y\})}\to0$ as $i\to\infty$ for $\hdone\otimes\gamma_{n,1}$-almost every $(y,V)\in\R\times Gr(n,1)$ \cite[step\,2 in proof of Prop.\,2.5]{CoRoMa17}.
Thus, using \cref{thm:diffuse0chainCost} and Fatou's lemma we have
\begin{align*}
\tau'(0)\mass{\flux^\perp}
&=c\int_{Gr(n,1)}\int_\R\tau'(0)\mass{\flux^\perp\cap p_V^{-1}(\{y\})}\,\de\hdone(y)\de\gamma_{n,1}(V)\\
&=c\int_{Gr(n,1)}\int_\R\tau(\flux^\perp\cap p_V^{-1}(\{y\}))\,\de\hdone(y)\de\gamma_{n,1}(V)\\
&\leq c\int_{Gr(n,1)}\int_\R\liminf_{i\to\infty}\tau(A_i\cap p_V^{-1}(\{y\}))\,\de\hdone(y)\de\gamma_{n,1}(V)\\
&\leq\liminf_{i\to\infty}c\int_{Gr(n,1)}\int_\R\tau(A_i\cap p_V^{-1}(\{y\}))\,\de\hdone(y)\de\gamma_{n,1}(V)\\
&=\liminf_{i\to\infty}\tau(A_i)
=\tau(\flux^\perp)\,.
\end{align*}
\item By \cite[Sec.\,6]{Wh99b} there is an isometry between $\hdone\restr S$-measurable functions $f:S\to\R$ of compact support and flat chains of finite mass and support in $S$,
where $|f|$ denotes the mass density, that is, $|f|=|\theta|$ for the rectifiable flat $1$-chain $\flux\restr S$.
Thus, by \cite[Sec.\,6]{Wh99b} or \cite{CoRoMa17} we further have $\tau(\flux\restr S)=\int_S\tau(|\theta(x)|)\,\de\hdone(x)$.
\item Finally, the previous two points imply that if $\flux=\theta\hdone\restr S+\flux^\perp$ with the desired properties, then $\tau(\flux)=\int_S\tau(|\theta|)\,\de\hdone+\tau'(0)|\flux^\perp|(\R^n)$ as desired.
\qedhere
\end{itemize}
\end{proof}

\notinclude{https://en.wikipedia.org/wiki/Exterior_algebra
http://math.stackexchange.com/questions/240491/what-is-a-covector-and-what-is-it-used-for
https://en.wikipedia.org/wiki/Differential_form
https://en.wikipedia.org/wiki/Exterior_derivative
https://web.math.princeton.edu/~ochodosh/GMTnotes.pdf
helps}

\section{Lagrangian model for transportation networks}\label{sec:MMS}

In this section we recall the alternative formulation of branched transport models due to \cite{Maddalena-Morel-Solimini-Irrigation-Patterns}, which we directly extend to our more general transportation costs.
For the major part of the well-posedness analysis we only consider concave transportation costs $\tau$;
the well-posedness for non-concave $\tau$ will follow in \cref{sec:equivalence} from the equivalence to the Eulerian model.
Many of the involved arguments are adaptations of \cite{Maddalena-Morel-Solimini-Irrigation-Patterns,Bernot-Caselles-Morel-Traffic-Plans,Bernot-Caselles-Morel-Structure-Branched,Maddalena-Solimini-Transport-Distances,BrWi15-equivalent}.

\subsection{Model definition}

Here, transport networks are described from a Lagrangian viewpoint, tracking for each transported mass particle its path, where the collection of all paths is typically denoted as irrigation pattern.

\begin{definition}[Irrigation patterns and loop-free patterns]\label{def:reference_space}
\begin{enumerate}
\item A \emph{reference space} is a measure space $(\reSpace,\Bcal(\reSpace),\reMeasure)$,
where $\reSpace$ is a complete separable uncountable metric space, $\Bcal(\reSpace)$ the $\sigma$-algebra of its Borel sets, and $\reMeasure$ a positive finite Borel measure on $\reSpace$ without atoms.

The reference space can be interpreted as the collection of all particles to be transported.
\item Let $I = [0,1]$. An \emph{irrigation pattern} is a measurable function $\chi : \reSpace \times I \to \R^n$ such that for almost all $p\in\reSpace$ we have $\chi_p\equiv\chi(p,\cdot) \in \AC(I;\R^n)$.
We say that a sequence of irrigation patterns $\chi_j$ \emph{converges uniformly} to $\chi$ if $\chi_j(p,\cdot)$ converges uniformly to $\chi(p,\cdot)$ for almost all $p\in\reSpace$.

The function $\chi_p$ can be viewed as the travel path of mass particle $p$.
\item Let $i_0^\chi,i_1^\chi:\reSpace \to \R^n$ be defined as $i_0^\chi(p) = \chi(p,0)$ and $i_1^\chi(p) = \chi(p,1)$.
The \emph{irrigating measure} and the \emph{irrigated measure} are defined as the pushforward of $\reMeasure$ via $i_0^\chi$ and $i_1^\chi$, respectively,
\begin{displaymath}
 \mu_+^\chi = \pushforward{(i_0^\chi)}{\reMeasure}\,, \quad \mu_-^\chi = \pushforward{(i_1^\chi)}{\reMeasure}\,.
\end{displaymath}
\item An irrigation pattern $\chi$ is \emph{loop-free} if $\chi_p$ is injective for $\reMeasure$-almost all $p\in\reSpace$.
\end{enumerate}
\end{definition}

\begin{remark}[Standard space]
Any complete separable metric measure space $(\reSpace,\Bcal(\reSpace),\reMeasure)$ in which the Borel measure $\reMeasure$ has no atoms
is known to be isomorphic as a measure space to the standard space $([0,1],\Bcal([0,1]),m\lebesgue^1 \restr [0,1])$ with $m=\reMeasure(\reSpace)$
(for a proof see \cite[Prop.\,12 or Thm.\,16 in Sec.\,5 of Chap.\,15]{Royden-Real-Analysis} or \cite[Chap.\,1]{Villani-Transport-Old-New}).
We may thus always assume our reference space to be the standard space without loss of generality.
\end{remark}

\begin{definition}[Cost functional]\label{def:patternCost}
\begin{enumerate}
\item For an irrigation pattern $\chi$ and every $x\in\R^n$ consider the set
\begin{equation*}
[x]_\chi = \{q \in \reSpace \ : \ x \in \chi_q(I)\}\label{eq:solidarity_classes}
\end{equation*}
of all particles flowing through $x$.
The total \emph{mass flux} through $x$ is denoted
\begin{equation*}
m_\chi(x) = \reMeasure([x]_\chi)\,.
\end{equation*}
\item For a transportation cost $\tau$ we define the \emph{marginal cost} per particle,
\begin{displaymath}
 r^\tau(w) = \begin{cases}
              \frac{\tau(w)}{w} & w > 0\,,\\
              \lim_{w \to 0^+} \frac{\tau(w)}{w} = \tau'(0) & w = 0\,.
             \end{cases}
\end{displaymath}
Note that the limit exists, but may be infinite \cite[Thm.\,16.3.3]{Ku09}.
\item The \emph{cost function} of an irrigation pattern $\chi$ is
\begin{displaymath}
 \JEnMMS(\chi) = \int_{\reSpace \times I} r^\tau(m_\chi(\chi(p,t)))|\dot\chi(p,t)|\,\de\reMeasure(p)\de t
\end{displaymath}
(the dot indicates differentiation with respect to the second, time-like argument).
We furthermore abbreviate
\begin{equation*}
\JEn^{\tau,\mu_+,\mu_-}[\chi]
=\begin{cases}
\JEnMMS(\chi)&\text{if $\mu_+^\chi = \mu_+$ and $\mu_-^\chi = \mu_-$},\\
\infty&\text{else.}
\end{cases}\end{equation*}
\item Given $\mu_+,\mu_- \in \fbm(\R^n)$, the \emph{transport problem} is
\begin{equation*}
 \min_{\chi:\reSpace\times I\to\R^n}\JEn^{\tau,\mu_+,\mu_-}[\chi]\,.
\end{equation*}
\end{enumerate}
\end{definition}

Just like in the Eulerian setting, without loss of generality we may restrict to measures in $\prob$ and irrigation patterns on $[-1,1]^n$.

\begin{lemma}[Mass and domain rescaling]\label{thm:patternRescaling}
Let $\tau$ be a transportation cost.
Given $\mu_+,\mu_-\in\fbm(\R^n)$ with $\mu_+(\R^n)=\mu_-(\R^n)=\reMeasure(\reSpace)=m$ and $\spt\mu_+,\spt\mu_-\subset[-s,s]^n$, for any irrigation pattern $\chi:\reSpace\times I\to\R^n$ we have
\begin{align*}
\JEn^{\tau,\mu_+,\mu_-}[\chi]&=m\JEn^{\overline\tau,\overline\mu_+,\overline\mu_-}[\chi]
&&\text{for }\overline\mu_\pm=\mu_\pm/m,\,\overline\tau(w)=\tau(mw),\,\overline\reMeasure=\reMeasure/m,\\
\JEn^{\tau,\mu_+,\mu_-}[\chi]&=s\JEn^{\tau,\underline\mu_+,\underline\mu_-}[\underline\chi]
&&\text{for }\underline\mu_\pm=\pushforward{(\tfrac1s\id)}{\mu_\pm},\,\underline\chi=\chi/s.
\end{align*}
Furthermore, if $\tau$ is concave, then for any $\chi$ with $\mu_\pm^\chi=\mu_\pm$ we can find a $\tilde\chi:\reSpace\times I\to[-s,s]^n$ with $\mu_\pm^{\tilde\chi}=\mu_\pm$ and non-greater cost.
\end{lemma}
\begin{proof}
The rescaling of the energies follows by direct calculation.
For the last statement we just consider $\tilde\chi(p,t)=\mathrm{proj}_{[-s,s]^n}\chi(p,t)$ to be the orthogonal projection onto $[-s,s]^n$.
The result then follows from the observation that $|\dot{\tilde\chi}(p,t)|\leq|\dot\chi(p,t)|$
and $r^\tau(m_{\tilde\chi}(\tilde\chi(p,t)))\leq r^\tau(m_{\chi}(\chi(p,t)))$ (due to $m_{\tilde\chi}(\tilde\chi(p,t))\geq m_{\chi}(\chi(p,t))$ and the concavity of $\tau$) for almost all $(p,t)\in\reSpace\times I$.
\end{proof}

\begin{remark}[Bounded patterns for non-concave $\tau$]
Note that the condition of concavity in the last statement of the previous lemma may in fact be dropped;
this can be proved using the more explicit cost characterization derived in \cref{sec:patternsGilbert}.
\end{remark}

\subsection{Existence of minimizers and their properties}

The following summary of the existence theory essentially follows the approach by Maddalena and Solimini \cite{Maddalena-Solimini-Transport-Distances,Maddalena-Solimini-Synchronic}.
In several results we restrict ourselves to concave transportation costs $\tau$, since otherwise the cost functional will not be lower semi-continuous (an example will be given below).
That the optimization problem admits minimizers even for non-concave $\tau$ will follow in the next section.
The following two statements deal with the lower semi-continuity of the cost; afterwards we shall consider its coercivity.

\begin{lemma}[Continuity properties of $m_\chi$]\label{prop:mass_is_upper_semi-continuous}
Let $\chi_k$ be a sequence of irrigation patterns converging uniformly to $\chi$, and let $t_k \in I$ such that $t_k \to t$. Then, for almost all $p \in \reSpace$,
\begin{align*}
 &m_\chi(\chi(p,t)) \geq \limsup_{k \to \infty} m_{\chi_k}(\chi_k(p,t_k))\,.\label{eq:mass_is_upper_semi-continuous}
\end{align*}
\end{lemma}

\begin{proof}
Fix $p \in \reSpace$ such that $\chi_p\in \AC(I;\R^n)$ and $\chi_k(p,\cdot)\to\chi(p,\cdot)$ uniformly, 
and define the sets
\begin{displaymath}
 A = \bigcap_{n=1}^\infty A_n\,, \quad A_n = \bigcup_{k \geq n} [\chi_k(p,t_k)]_{\chi_k}\,.
\end{displaymath}
Recall that $A = \limsup_{k\to\infty} [\chi_k(p,t_k)]_{\chi_k}$ and $\reMeasure([\chi_k(p,t_k)]_{\chi_k}) = m_{\chi_k}(\chi_k(p,t_k))$ so that
\begin{displaymath}
 \reMeasure(A) = \lim_{n\to\infty}\reMeasure(A_n) \geq \limsup_{k\to\infty} m_{\chi_k}(\chi_k(p,t_k))\,.
\end{displaymath}
We now show $A \subseteq [\chi(p,t)]_{\chi}$ up to a $\reMeasure$-nullset so that $m_\chi(\chi(p,t)) \geq \reMeasure(A) \geq \limsup_k m_{\chi_k}(\chi_k(p,t_k))$ as desired.
Indeed, let $q \notin[\chi(p,t)]_{\chi}$, then by continuity of $\chi$ we have $d=\dist(\chi(p,t),\chi(q,I)) > 0$.
Assuming $\chi_k(q,\cdot) \to \chi(q,\cdot)$ uniformly, this implies $\dist(\chi_k(p,t_k),\chi_k(q,I)) > \frac d2$ for all $k$ large enough
so that $q\notin A_k$ for any $k$ large enough and thus $q\notin A$.
\end{proof}

\begin{proposition}[Lower semi-continuity of $\JEnMMS$]\label{prop:urban_planning_energy_is_lower_semi-continuous}
For a concave transportation cost $\tau$, the functional $\JEnMMS$ is lower semi-continuous with respect to uniform convergence of patterns.
\end{proposition}

\begin{proof}
Let the sequence $\chi_k$ of irrigation patterns converge uniformly to $\chi$
and define the measures $\mu_k^p = |\dot\chi_k(p,\cdot)|\,\de t$ and $\mu^p = |\dot\chi(p,\cdot)|\de t$ for $p\in\reSpace$, $k=1,2,\ldots$.
By the uniform convergence of $\chi_k$ we have $\dot\chi_k(p,\cdot)\to\dot\chi(p,\cdot)$ in the distributional sense and thus
\begin{displaymath}
\mu^p(A)\leq\liminf_{k\to\infty}\mu_k^p(A)
\end{displaymath}
for any open $A\subset I$ and almost every $p\in\reSpace$.
Furthermore, since $r^\tau$ is non-increasing for concave $\tau$, \cref{prop:mass_is_upper_semi-continuous} implies $r^\tau(\chi(p,t)) \leq \liminf_{k \to \infty} r^\tau(m_{\chi_k}(\chi_k(p,t)))$ for almost all $p\in\reSpace$ and $t\in I$.
Thanks to \cite[Def.\,C.1 \& Thm.\,C.1]{Maddalena-Solimini-Synchronic} we thus have
\begin{displaymath}
\int_I r^\tau(m_\chi(\chi(p,t))) \,\de\mu^p(t) \leq \liminf_{k\to\infty} \int_I r^\tau(m_{\chi_k}(\chi_k(p,t))) \,\de\mu_k^p(t)\,.
\end{displaymath}
Integrating with respect to $\reMeasure$ and applying Fatou's Lemma ends the proof.
\end{proof}

\begin{remark}[Failure for non-concave $\tau$]\label{rem:nonconcaveTau}
The cost functional $\JEnMMS$ fails to be lower semi-continuous without requiring concavity of $\tau$.
As a counterexample consider $\tau(w)=\lceil\frac wa\rceil$ for some $a\in(0,1)$ (where $\lceil\cdot\rceil$ denotes rounding to the next largest integer) and the sequence of irrigation patterns $\chi_k:\reSpace\times I\to\R^2$, $\reSpace=[0,1]$,
\begin{equation*}
\chi_k(p,t)=\begin{cases}(3t,\frac tk)&\text{if }t\in[0,\frac13]\\(2-3t,\frac{1-2t}k)&\text{if }t\in[\frac13,\frac23]\\(3t-2,\frac{t-1}k)&\text{if }t\in[\frac23,1]\end{cases}\text{ for }p\in[0,a]\,,\qquad
\chi_k(p,t)=(t,0)\text{ for }p\in(a,1]\,,
\end{equation*}
which is illustrated in \cref{fig:nonconcaveTau} and converges uniformly to the obvious $\chi$.
One readily calculates
\begin{equation*}
\lim_{k\to\infty}\JEnMMS(\chi_k)=3\tau(a)+\tau(1-a)\,,\qquad
\JEnMMS(\chi)=(1+2a)\tau(1)\,.
\end{equation*}
Thus, for $a$ slightly smaller than $\frac12$ we have $\tau(a)=1$, $\tau(1-a)=2$, $\tau(1)=3$ and thus $\JEnMMS(\chi)\geq\lim_{k\to\infty}\JEnMMS(\chi_k)$.
\end{remark}

\begin{figure}
\centering
\setlength{\unitlength}{.6\linewidth}
\begin{picture}(1,.3)
\put(0,0){\includegraphics[width=\unitlength]{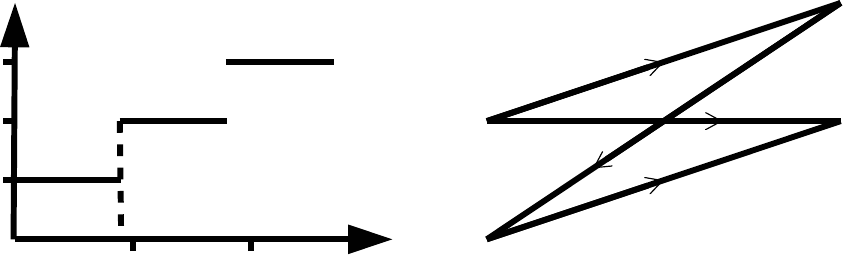}}
\put(.13,-.01){\small$a$}
\put(.15,-.03){\small$\frac12$}
\put(.29,-.03){\small$1$}
\put(.44,-.01){\small$w$}
\put(-.02,.08){\small$1$}
\put(-.02,.15){\small$2$}
\put(-.02,.22){\small$3$}
\put(-.06,.28){\small$\tau(w)$}
\put(.5,.15){\small$(0,0)$}
\put(1,.15){\small$(1,0)$}
\put(1,.27){\small$(1,\frac1{3k})$}
\put(.53,-.02){\small$(0,-\frac1{3k})$}
\end{picture}
\caption{Illustration of the transportation cost $\tau$ and irrigation pattern $\chi_k$ from \cref{rem:nonconcaveTau}.}
\label{fig:nonconcaveTau}
\end{figure}
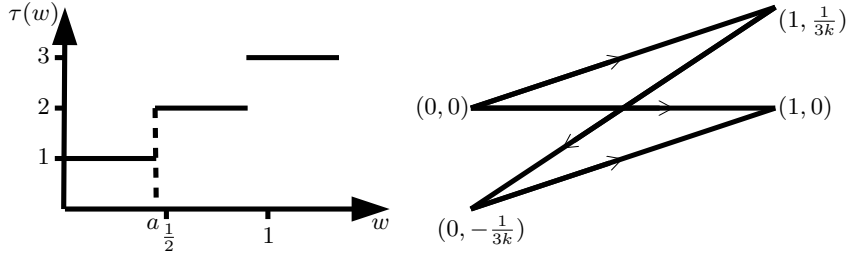

To examine the coercivity of the cost functional, we first collect some properties of patterns and \transportPath{} measures, which we will also need later on.
First note that the cost function of an irrigation pattern bounds its average path length and that any irrigation pattern can be assumed to be Lipschitz in its second argument.

\begin{lemma}[Bound on average path length]\label{thm:pathLength}
For any irrigation pattern $\chi$ and $\lambda^\tau$ from \cref{thm:averageCost} we have
\begin{equation*}
\JEnMMS(\chi)\geq\lambda^\tau(\reMeasure(\reSpace))\int_\reSpace\hdone(\chi_p(I))\,\de\reMeasure(p)\,.
\end{equation*}
\end{lemma}
\begin{proof}
This is a direct consequence of $\lambda^\tau(\reMeasure(\reSpace))\leq r^\tau(m_\chi(\chi(p,t)))$ and the fact $\int_0^1|\dot\chi_p(t)|\de t=\hdone(\chi_p(I))$.
\end{proof}

\begin{proposition}[Constant speed reparameterization of patterns]\label{thm:constSpeedPatternsUrbPl}
Irrigation patterns of finite cost $\JEnMMS$ can be reparameterized such that $\chi_p\in\Lip(I;\R^n)$ and $|\dot\chi_p|$ is constant for almost all $p\in\reSpace$ without changing the cost $\JEnMMS$.
\end{proposition}
\begin{proof}
The proof is analogous to \cite[Prop.\,2.3.2 to 2.3.4]{BrWi15-equivalent} (see also \cite[Lem.\,6.2]{Bernot-Caselles-Morel-Traffic-Plans} or \cite[Lem.\,4.1, Lem.\,4.2]{BeCaMo09}),
merely replacing $s_\alpha^\chi$ by $r^\tau(m_\chi(\cdot))$
and replacing the estimate $1 \leq \reMeasure(\reSpace)^{1-\alpha}s_\alpha^\chi(\chi(p,t))$ by $1 \leq (\lambda^\tau(\reMeasure(\reSpace)))^{-1} r^\tau(m_\chi(\chi(p,t)))$.
\end{proof}

The above reparameterization allows to identify irrigation patterns with \transportPath{} measures \cite[Def.\,2.5]{BuPrSoSt09}, for which compact sets have a simple characterization (as shown below in a simple variant of, for instance, \cite[Lem.\,4.3.5]{BrWi15-equivalent}).

\begin{definition}[\TransportPath{} measures]\label{def:transport_path_measures}
\begin{enumerate}
\item On the set $\Lip(I;\R^n)$ of Lipschitz curves we define the pseudometric
\begin{equation*}
d_\Theta(\theta_1,\theta_2)=\inf\left\{\max_{t\in I}|\theta_1(t)-\theta_2(\varphi(t))|\ :\ \varphi:I\to I\text{ increasing and bijective}\right\}
\end{equation*}
and let $\Theta=\Lip(I;\R^n)/\sim$ be the quotient space under the equivalence relation $\theta_1\sim\theta_2$ if and only if $d_\Theta(\theta_1,\theta_2)=0$.

Without explicit mention we will frequently identify an element $\theta\in\Theta$ with a representation in $\Lip(I;\R^n)$.
For instance, we will write $\theta(0)$, $\theta(1)$, and $\int_I |\dot\theta(t)|\de t$ for start point, end point, and length of $\theta\in\Theta$,
since those expressions are all independent of the particular representation.
Likewise we will write $\int_I\varphi(\theta(t))\cdot\dot\theta(t)\,\de t$ and $\int_I\psi(\theta(t))|\dot\theta(t)|\,\de t$ for the integral of any continuous $\varphi:\R^n\to\R^n$ and $\psi:\R^n\to\R$ along the curve $\theta\in\Theta$.

\notinclude{We define $\Theta=\Lip(I;\R^n)$ to be the set of Lipschitz curves $I \to \R^n$ with the metric
\begin{equation*}
d_\Theta(\theta_1,\theta_2)=\inf\left\{\max_{t\in I}|\theta_1(t)-\theta_2(\varphi(t))|\ :\ \varphi:I\to I\text{ increasing and bijective}\right\}.
\end{equation*}
}
\item Given $C>0$ and $\Omega\subset\R^n$ compact, we define $\Theta_C^\Omega\subset\Theta$ to be the set of curves starting in $\Omega$ with length bounded by $C$,
\begin{displaymath}
 \Theta_C^\Omega =\left\{\theta\in\Theta\,:\,\theta(0)\in\Omega\text{ and }\textstyle\int_I |\dot\theta(t)|\de t \leq C\right\}\,.
\end{displaymath}
\item A \emph{\transportPath{} measure} is a nonnegative measure $\eta$ on $\Theta$ (endowed with the Borel $\sigma$-algebra).
If $\mu_+,\mu_- \in \fbm(\R^n)$, we say that $\eta$ moves $\mu_+$ onto $\mu_-$ if
\begin{displaymath}
 \pushforward{p_0}{\eta} = \mu_+,\quad \pushforward{p_1}{\eta} = \mu_-,
\end{displaymath}
where $p_t : \Theta \to \R^n$ is defined by $p_t(\theta) = \theta(t)$ for $t\in\{0,1\}$. We denote by $\TPM(\mu_+,\mu_-)$ the set of \transportPath{} measures moving $\mu_+$ onto $\mu_-$.
\item Given an irrigation pattern $\chi:\reSpace\times I\to\R^n$ with finite cost,
by \cref{thm:constSpeedPatternsUrbPl} we may identify each $p\in\reSpace$ with an element of $\Theta$ via the map $\iota:\reSpace\to\Theta$, $p\mapsto\chi(p,\cdot)$.
The \emph{induced \transportPath{} measure} is defined as $\eta=\pushforward{\iota}{\reMeasure}$.
Two irrigation patterns are called \emph{equivalent}, if their induced \transportPath{} measures coincide.
\end{enumerate}
\end{definition}

\begin{lemma}[Compactness for \transportPath{} measures]\label{lem:compactness_lemma_for_transport_path_measures}
Let $\Omega\subset\R^n$ be compact.
For $\mu_\pm^k\in\fbm(\Omega)$ with $\mu_\pm^k\weakstarto\mu_\pm$
let $\eta_k\in\TPM(\mu_+^k,\mu_-^k)$ be a sequence of \transportPath{} measures such that
\begin{equation*}
\eta_k(\Theta\setminus\Theta_C^\Omega) \to 0
\end{equation*}
uniformly in $k$ as $C\to\infty$.
Then, up to a subsequence, $\eta_k \weakstarto \eta$ for some $\eta\in\TPM(\mu_+,\mu_-)$ in the sense
\begin{displaymath}
 \int_{\Theta} \varphi(\theta)\,\de\eta_k(\theta) \to \int_{\Theta} \varphi(\theta)\,\de\eta(\theta) \quad\text{for all}\ \varphi\in\contbdd(\Theta)\,,
\end{displaymath}
where $\contbdd(\Theta)$ denotes the set of bounded continuous functions on $\Theta$.
\end{lemma}

\begin{proof}
We first show that $\Theta_C^\Omega\subset\Theta$ is (sequentially) compact.
To this end let $\theta_k$, $n=1,2,\ldots$, be a sequence in $\Theta_C^\Omega$.
With a slight abuse of notation, $\theta_1,\theta_2,\ldots$ shall also denote corresponding representations in $\Lip(I;\R^n)$.
Upon reparameterization of each element (which does not change the sequence), the $\theta_k$ are uniformly Lipschitz.
The Arzel\`a--Ascoli Theorem now implies the existence of some $\theta\in\cont^0(I;\Omega)$ such that $\theta_k\to\theta$ uniformly up to a subsequence.
Furthermore, the Lipschitz constant of $\theta$ is bounded by the Lipschitz constant of the $\theta_k$ so that
\begin{displaymath}
 \int_I |\dot\theta(t)|\,\de t \leq \liminf_k \int_I |\dot\theta_k(t)|\,\de t \leq C\,.
\end{displaymath}

As a consequence, the $\eta_k$ are tight, that is, for every $\varepsilon>0$ there is a compact set, namely $\Theta_C^\Omega$ with $C$ large enough, so that $\eta_k(\Theta\setminus\Theta_C^\Omega)<\varepsilon$.
In addition, $\Theta$ is separable (which follows from the separability of $\Lip(I;\R^n)$).
Hence, up to a subsequence we get $\eta_k \weakstarto \eta$ due Prokhorov's Theorem \cite[Thm.\,5.1]{Bil99}, which assures weak compactness for a tight set of measures.

It remains to prove $\pushforward{p_0}{\eta} = \mu_+$ (the proof of $\pushforward{p_1}{\eta} = \mu_-$ works analogously).
Because of $\pushforward{p_0}{\eta_k} = \mu_+^k$ for all $k$ we have
\begin{equation*}
 \int_\Theta \varphi(p_0(\theta))\,\de\eta_k(\theta)
 = \int_\Omega \varphi(x) \,\de\mu_+^k(x)
 \quad\text{for all}\ \varphi \in \contbdd(\Omega)\,.
\end{equation*}
Due to $\eta_k \weakstarto \eta$ as well as $\varphi\circ p_0\in\contbdd(\Theta)$, letting $k\to\infty$ we finally arrive at
\begin{equation*}
 \int_\Theta \varphi(p_0(\theta))\,\de\eta(\theta)
 = \int_\Omega \varphi(x) \,\de\mu_+(x) \quad\text{for all}\ \varphi \in \contbdd(\Omega)\,,
\end{equation*}
that is, $\pushforward{p_0}{\eta} = \mu_+$.
\end{proof}

This compactness can now be employed to obtain coercivity of our cost functional.

\begin{proposition}[Coercivity of $\JEnMMS$]\label{thm:coercivity}
Let $\mu_+^k,\mu_-^k\in\fbm(\R^n)$ with bounded support such that $\mu_\pm^k\weakstarto\mu_\pm$ and $\chi_k$ be a sequence of irrigation patterns with $\mu_\pm^{\chi_k}=\mu_\pm^k$ and $\JEnMMS(\chi_k)<K$ for some $K<\infty$.
Then up to equivalence of irrigation patterns there exists a subsequence converging uniformly to some $\chi$ with $\mu_\pm^\chi=\mu_\pm$.
\end{proposition}
\begin{proof}
Let $\eta_k\in\TPM(\mu_+,\mu_-)$ be the sequence of \transportPath{} measures induced by $\chi_k$.
This sequence satisfies
\begin{multline*}
C\eta_k(\Theta\setminus\Theta_C^\Omega)
\leq\int_\Theta\int_I|\dot\theta(t)|\,\de t\de\eta_k(\theta)\\
=\int_\reSpace\int_I|\dot\chi_k(p,t)|\,\de t\de\reMeasure(p)
\leq\tfrac1{\lambda^\tau(\reMeasure(\reSpace))}\int_\reSpace\int_I r^\tau(m_{\chi_k}(\chi_k(p,t)))|\dot\chi_k(p,t)|\,\de t\de\reMeasure(p)
<K/\lambda^\tau(\reMeasure(\reSpace))\,.
\end{multline*}
Furthermore, for $\eta_k$-almost all $\theta$ we have $\theta(0)\in\Omega=\spt\mu_+$.
Thus, by \cref{lem:compactness_lemma_for_transport_path_measures} there exists a subsequence converging to some $\eta\in\TPM(\mu_+,\mu_-)$.
Now by Skorohod's theorem \cite[Thm.\,6.7]{Bil99} there exist irrigation patterns $\tilde\chi_k$, which also induce the $\eta_k$, that converge uniformly to some irrigation pattern $\chi$, which induces $\eta$.
\end{proof}

The existence of optimal irrigation patterns now is a direct consequence of the coercivity and lower semi-continuity.

\begin{theorem}[Existence]\label{thm:existence_of_minimizers_patterns}
Given $\mu_+,\mu_-\in\fbm(\R^n)$ with bounded support and a concave transportation cost $\tau$, the minimization problem
\begin{displaymath}
 \min_\chi \JEn^{\tau,\mu_+,\mu_-}[\chi]
\end{displaymath}
either has a solution, or $\JEn^{\tau,\mu_+,\mu_-}$ is infinite.
\end{theorem}

\begin{proof}
Consider a minimizing sequence of irrigation patterns $\chi_k$ such that $\lim_{k\to\infty}\JEn^{\tau,\mu_+,\mu_-}[\chi_k]=\inf_\chi \JEn^{\tau,\mu_+,\mu_-}[\chi]$ and assume this to be finite.
By \cref{thm:coercivity}, a subsequence of irrigation patterns converges uniformly to some pattern $\chi$ up to equivalence.
\Cref{prop:urban_planning_energy_is_lower_semi-continuous} now implies $\lim_{k\to\infty}\JEn^{\tau,\mu_+,\mu_-}[\chi_k]\geq\JEn^{\tau,\mu_+,\mu_-}[\chi]$, which ends the proof.
\end{proof}

\begin{theorem}[Existence of finite cost pattern]
If $\tau$ is an admissible transportation cost and $\mu_+,\mu_-\in\fbm(\R^n)$ with $\mu_+(\R^n)=\mu_-(\R^n)$ and bounded support,
then there exists an irrigation pattern $\chi$ with $\JEn^{\tau,\mu_+,\mu_-}[\chi]<\infty$.
\end{theorem}

\begin{proof}
The proof is analogous to \cref{thm:finiteCostGraph} and needs a completely analogous construction to the one of \cref{def:nadicGraph}.
In brief, by \cref{thm:patternRescaling} we may assume $\mu_+,\mu_-\in\prob$.
Now, for a sequence $\alpha_k$ of $n$-tuples in $\{-1,1\}^n$ recursively define the (Lipschitz, constant speed) particle path $\theta_\alpha$ via $\theta_\alpha(0)=0$ and
\begin{equation*}
\theta_\alpha(t)=\theta_\alpha(1-2^{-k})+\alpha_k(t-1+2^{-k})\text{ for }t\in[1-2^{-k},1-2^{-k-1}],\,k=0,1,\ldots.
\end{equation*}
Note that a single $\theta_\alpha$ describes a path on the $n$-adic graph from \cref{def:nadicGraph} and that one can connect the origin to any point in $[-1,1]^n$ via such a path.
Furthermore, for two sequences $\alpha_k,\gamma_k$ in $\{-1,1\}^n$ we set
\begin{equation*}
\theta_{\alpha\gamma}(t)=\begin{cases}\theta_\gamma(1-\frac t2)&\text{if }t\in[0,\frac12]\,,\\\theta_\alpha(t)&\text{if }t\in[\frac12,1]\,.\end{cases}
\end{equation*}
Now consider a measurable map $\iota:\reSpace\to[-1,1]^n\times[-1,1]^n$ such that $\pushforward{\iota}{\reMeasure}=\mu_+\otimes\mu_-$ (which shall denote the completion of the product measure)
\notinclude{(for instance the mapping induced by the optimal transport or a construction, where on each level the hypercubes are ordered, and the mass in each hypercube is the interval of $\reSpace$ that will be assigned to it)}
and set $\chi(p,t)=\theta_{\alpha\gamma}(t)$, where $\theta_\alpha$ and $\theta_\gamma$ connect the origin with both components of $\iota(p)$, respectively.
Obviously, $\chi$ has irrigating and irrigated measure $\mu_+$ and $\mu_-$, respectively. Furthermore, by the concavity of the upper bound $\beta$ of $\tau$,
\begin{multline*}
\JEnMMS(\chi)
=\int_\reSpace\int_Ir^\tau(m_\chi(\chi(p,t)))\sqrt n\,\de t\,\de\reMeasure(p)
=\sqrt n\sum_{k=1}^\infty\int_\reSpace\int_{1-2^{1-k}}^{1-2^{-k}}r^\tau(m_\chi(\chi(p,t)))\,\de t\,\de\reMeasure(p)\\
=\sqrt n\sum_{k=1}^\infty\int_{1-2^{1-k}}^{1-2^{-k}}\int_\reSpace r^\tau(m_\chi(\chi(p,t)))\,\de\reMeasure(p)\,\de t
\leq\sqrt n\sum_{k=1}^\infty\int_{1-2^{1-k}}^{1-2^{-k}}\int_\reSpace r^\beta(m_\chi(\chi(p,t)))\,\de\reMeasure(p)\,\de t\\
\leq\sqrt n\sum_{k=1}^\infty\int_{1-2^{1-k}}^{1-2^{-k}}2^{nk}\beta(2^{-nk})\,\de t
=\sqrt n\sum_{k=1}^\infty2^{(n-1)k}\beta(2^{-nk})\,,
\end{multline*}
which is finite by \cref{thm:admTau}.
\end{proof}

\begin{corollary}[Existence]
Given $\mu_+,\mu_-\in\fbm(\R^n)$ with bounded support and an admissible concave $\tau$, the minimization problem
$\min_\chi \JEn^{\tau,\mu_+,\mu_-}[\chi]$
has a solution.
\end{corollary}

\begin{remark}[Existence for general transportation cost]
The existence of an optimal irrigation pattern for general admissible transportation costs $\tau$ will follow from the model equivalence to the Eulerian model in \cref{thm:modelEquivalence}.
\end{remark}

\subsection{Explicit formula for cost function}\label{sec:patternsGilbert}
We will show that optimal irrigation patterns are loop free.
To this end we follow \cite[Ch.\,4]{Bernot-Caselles-Morel-Traffic-Plans} and first show a reformulation of the cost function as a generalized Gilbert energy.
This requires to show a rectifiability property of sets with positive flux.
Here we can follow the argument in \cite[Lem.\,4.6.4]{Be05} or \cite[Lem.\,6.3]{Bernot-Caselles-Morel-Traffic-Plans} and note that it can be rephrased not to make use of a bound on the cost functional.

\begin{lemma}[Rectifiability of positive flux set]\label{thm:rectifiabilityFlux}
Given an irrigation pattern $\chi$, the set
\begin{equation*}
S_\chi=\{x\in\R^n\,:\,m_\chi(x)>0\}
\end{equation*}
is countably $1$-rectifiable.
\end{lemma}
\begin{proof}
We shall first cover the preimage of $S_\chi$ under $\chi$,
\begin{equation*}
D=\chi^{-1}(S_\chi)=\{(p,t)\in\reSpace\times I\,:\,m_\chi(\chi_p(t))>0\}\,,
\end{equation*}
by a countable union of preimages of particle paths $\chi_{p_k}(I)$, $p_k\in\reSpace$,
from which we then derive the rectifiability of $S_\chi$.
In detail, for $k=1,2,\ldots$ we inductively define the particle $p_k\in\reSpace$ and the set $D_{p_k}^k$ as follows.
Given $p_1,\ldots,p_{k-1}\in\reSpace$ we set
\begin{equation*}
\textstyle
D_p^k=\left\{(q,t)\in D\,:\,\chi_q(t)\in\chi_p(I)\setminus\bigcup_{j=1}^{k-1}\chi_{p_j}(I)\right\}\subset\chi^{-1}(\chi_p(I))\,,
\end{equation*}
and we choose $p_k\in\reSpace$ such that $\reMeasure\otimes\lebesgue^1(D_{p_k}^k)\geq d^k/2$ for $d^k=\sup_{p\in\reSpace}\reMeasure\otimes\lebesgue^1(D_{p}^k)$,
where $\reMeasure\otimes\lebesgue^1$ denotes the completion of the product measure on $\reSpace\times I$.
Note that
\begin{equation*}
\sum_{k=1}^\infty d^k\leq2\sum_{k=1}^\infty\reMeasure\otimes\lebesgue^1(D_{p_k}^k)\leq2\reMeasure\otimes\lebesgue^1(D)\leq2\reMeasure(\reSpace)
\end{equation*}
so that necessarily $d^k\to0$ as $k\to\infty$.
Now by the Fubini--Tonelli theorem we have
\begin{align*}
\reMeasure(\reSpace)d^k
\geq\int_\reSpace\reMeasure\otimes\lebesgue^1(D_p^k)\,\de\reMeasure(p)
&=\int_\reSpace\int_\reSpace\int_0^1\setchar{\chi_p(I)\setminus\bigcup_{j=1}^{k-1}\chi_{p_j}(I)}(\chi_q(t))\setchar{D}(q,t)\,\de t\de\reMeasure(q)\de\reMeasure(p)\\
&=\int_\reSpace\int_0^1\int_\reSpace\setchar{\chi_p(I)}(\chi_q(t))\setchar{\R^n\setminus\bigcup_{j=1}^{k-1}\chi_{p_j}(I)}(\chi_q(t))\setchar{D}(q,t)\,\de\reMeasure(p)\de t\de\reMeasure(q)\\
&=\int_\reSpace\int_0^1m_\chi(\chi_q(t))\setchar{\R^n\setminus\bigcup_{j=1}^{k-1}\chi_{p_j}(I)}(\chi_q(t))\setchar{D}(q,t)\,\de t\de\reMeasure(q)\\
&=\int_{D\setminus\bigcup_{j=1}^{k-1}D_{p_j}^j}m_\chi(\chi_q(t))\,\de t\de\reMeasure(q)
\geq\int_{D\setminus\bigcup_{j=1}^\infty D_{p_j}^j}m_\chi(\chi_q(t))\,\de t\de\reMeasure(q)\,.
\end{align*}
Letting $k\to\infty$, the left-hand side tends to zero, which due to $m_\chi(\chi_q(t))>0$ for all $(q,t)\in D$ implies that $D\setminus\bigcup_{j=1}^\infty D_{p_j}^j$ is a nullset.

It remains to show that $S_\chi\setminus\bigcup_{j=1}^\infty\chi_{p_j}(I)$ is a $\hdone$-nullset.
Assume to the contrary the existence of a non-$\hdone$-negligible $C\subset S_\chi\setminus\bigcup_{j=1}^\infty\chi_{p_j}(I)$, then
\begin{multline*}
0<\int_Cm_\chi(x)\,\de\hdone(x)
=\int_C\int_\reSpace\setchar{\chi_p(I)}(x)\,\de\reMeasure(p)\de\hdone(x)\\
=\int_\reSpace\int_C\setchar{\chi_p(I)}(x)\,\de\hdone(x)\de\reMeasure(p)
=\int_\reSpace\int_0^1\setchar{C}(\chi_p(t))|\dot\chi_p(t)|\,\de t\de\reMeasure(p)
=\int_{\tilde D}|\dot\chi_p(t)|\,\de t\de\reMeasure(p)
\end{multline*}
for $\tilde D=\{(p,t)\in\reSpace\times I\,:\,\chi_p(t)\in C\}$.
By definition of $C$ we have $\tilde D\subset D\setminus\bigcup_{j=1}^\infty D_{p_j}^j$ so that $\tilde D$ is negligible,
which contradicts $\int_{\tilde D}|\dot\chi_p(t)|\,\de t\de\reMeasure(p)>0$.
\end{proof}

We can now reformulate our cost function as a generalized Gilbert energy, following \cite[Prop.\,4.8-4.9]{Bernot-Caselles-Morel-Traffic-Plans}.
Note that \cite[Prop.\,4.8]{Bernot-Caselles-Morel-Traffic-Plans} uses the Fubini--Tonelli theorem without explicitly verifying the $\sigma$-additivity of the involved measures.
For this reason we first followed \cite{Be05} in showing rectifiability, which will then justify the use of the Fubini--Tonelli theorem.

\begin{proposition}[Loop-free irrigation patterns and generalized Gilbert energy]\label{thm:GilbertPatterns}
For any irrigation pattern $\chi$ there exists a loop-free irrigation pattern $\tilde\chi$ with $\mu_\pm^{\tilde\chi}=\mu_\pm^\chi$, $m_\chi\geq m_{\tilde\chi}$, and
\begin{equation*}
\JEnMMS(\chi)
\geq\JEnMMS(\tilde\chi)
=\int_{S_{\tilde\chi}}\tau(m_{\tilde\chi}(x))\,\de\hdone(x)+\tau'(0)\int_\reSpace\hdone(\tilde\chi_p(I)\setminus S_{\tilde\chi})\,\de\reMeasure(p)\,.
\end{equation*}
\end{proposition}
\begin{proof}
Using the one-dimensional area formula, we can calculate the cost function as
\begin{align*}
\JEnMMS(\chi)
&=\int_\reSpace\int_0^1r^\tau(m_\chi(\chi_p(t)))|\dot\chi_p(t)|\,\de t\de \reMeasure(p)\\
&\geq\int_\reSpace\int_{\chi_p(I)}r^\tau(m_\chi(x))\,\de\hdone(x)\de \reMeasure(p)\\
&=\int_\reSpace\int_{S_\chi}\setchar{\chi_p(I)}(x)r^\tau(m_\chi(x))\,\de\hdone(x)\de \reMeasure(p)
+\int_\reSpace\int_{\chi_p(I)\setminus S_\chi}r^\tau(0)\,\de\hdone(x)\de \reMeasure(p)\\
&=\int_{S_\chi}\int_\reSpace\setchar{\chi_p(I)}(x)r^\tau(m_\chi(x))\,\de \reMeasure(p)\de\hdone(x)
+\tau'(0)\int_\reSpace\hdone(\chi_p(I)\setminus S_\chi)\de \reMeasure(p)\\
&=\int_{S_\chi}\tau(m_\chi(x))\,\de\hdone(x)
+\tau'(0)\int_\reSpace\hdone(\chi_p(I)\setminus S_\chi)\de \reMeasure(p)\,,
\end{align*}
where the inequality becomes an equality for loop-free irrigation patterns.
Now \cite[Prop.\,4.6]{Bernot-Caselles-Morel-Traffic-Plans} says that one can remove all loops from an irrigation pattern $\chi$,
resulting in a loop-free irrigation pattern $\tilde\chi$ with same irrigating and irrigated measure and with $m_{\tilde\chi}(x)\leq m_{\chi}(x)$ for all $x\in\R^n$.
Thus we have
\begin{multline*}
\JEnMMS(\chi)
\geq\int_{S_\chi}\tau(m_\chi(x))\,\de\hdone(x)
+\tau'(0)\int_\reSpace\hdone(\chi_p(I)\setminus S_\chi)\de \reMeasure(p)\\
\geq\int_{S_{\tilde\chi}}\tau(m_{\tilde\chi}(x))\,\de\hdone(x)
+\tau'(0)\int_\reSpace\hdone(\tilde\chi_p(I)\setminus S_{\tilde\chi})\de \reMeasure(p)
=\JEnMMS(\tilde\chi)\,.\qedhere
\end{multline*}
\end{proof}


\section{Model equivalence}\label{sec:equivalence}

Here we show that the Eulerian and the Lagrangian model formulations are actually equivalent,
that is, they produce the same infimal value of their cost function, and their minimizers, so they exist, can be related to each other.
The well-posedness of the Lagrangian model is a simple consequence of this result.
Using the explicit characterizations \cref{thm:GilbertFlux} and \cref{thm:GilbertPatterns} of the cost functionals as well as classical characterizations of $1$-currents,
this new proof becomes rather straightforward.
We first show that the Lagrangian model achieves lower cost function values than the Eulerian version and then the opposite inequality.

\subsection{Irrigation patterns have lower cost function than \transportPaths{}}
The proof directly constructs an irrigation pattern from a \transportPath{}, making use of a decomposition result for $1$-currents by Smirnov \cite{Sm93}.


\begin{proposition}[Irrigation patterns have lower cost]\label{prop:constructPatternFromFlux}
Let $\tau$ be a transportation cost.
For every \transportPath{} $\flux$ with finite cost $\JEn^{\tau,\mu_+,\mu_-}[\flux]$ there exists an irrigation pattern $\chi_\flux$ with
\begin{equation}\label{eq:urban_chi_leq_urban_flux}
\JEn^{\tau,\mu_+,\mu_-}[\chi_\flux] \leq \JEn^{\tau,\mu_+,\mu_-}[\flux]\,.
\end{equation}
Furthermore, if $\flux$ is optimal, then $\chi_\flux$ and $\flux$ are related via
\begin{equation}\label{eq:constructPatternFromFlux}
\int_{\R^n}\varphi\cdot\de\flux=\int_\reSpace\int_I\varphi(\chi_\flux(p,t))\cdot\dot\chi_\flux(p,t)\,\de t\,\de \reMeasure(p) \;\text{for all}\; \varphi\in \cont_c(\R^n;\R^n).
\end{equation}
\end{proposition}
\begin{proof}
Let $\flux$ be a mass flux between $\mu_+$ and $\mu_-$. By \cite[Thm.\,4.2]{Sil08}
\notinclude{Let $\flux$ have finite cost. By \cref{thm:GilbertFlux}}%
we have $\flux=\theta\hdone\restr S+\flux^\perp$
for a countably $1$-rectifiable set $S$, an $\hdone\restr S$-measurable function $\theta:S\to\R^n$ which is tangent to $S$ $\hdone$-almost everywhere, and $\flux^\perp$ a diffuse part\,.
Furthermore, by \cref{thm:GilbertFlux}
\begin{equation*}
\JEn^{\tau,\mu_+,\mu_-}[\flux]
=\int_{S}\tau(|\theta(x)|)\,\de\hdone(x)+\tau'(0)|\flux^\perp|(\R^n)\,.
\end{equation*}
Now by \cite[Thm.\,C]{Sm93}, $\flux$ can be decomposed into the sum of two vector-valued measures $P,Q\in\rca(\R^n;\R^n)$ with $\dv P=0$ and $\dv Q=\mu_+-\mu_-$
such that the total variation measure satisfies $|\flux|=|P+Q|=|P|+|Q|$.
This implies $|Q|\leq|\flux|=|\theta|\hdone\restr S+|\flux^\perp|$
so that $Q$ must have the form $Q=\theta_Q\hdone\restr S+\flux_Q^\perp$ for some $|\theta_Q|\leq|\theta|$ and $|\flux_Q^\perp|\leq|\flux^\perp|$.
Consequently, $\JEn^{\tau,\mu_+,\mu_-}[\flux]\geq\JEn^{\tau,\mu_+,\mu_-}[Q]$ so that we may actually assume $Q=\flux$ and $P=0$.
Again by \cite[Thm.\,C]{Sm93} there exists a complete decomposition of $\flux$ into curves of finite length, that is, there exists a \transportPath{} measure $\eta$ on $\Theta$ such that
\begin{align*}
\int_{\R^n}\varphi\cdot\de\flux
&=\int_\Theta\int_I\varphi(\vartheta(t))\cdot\dot\vartheta(t)\,\de t\de\eta(\vartheta)
\qquad\text{for all }\varphi\in\cont_c(\R^n;\R^n)\,,\\
\int_{\R^n}\psi\,\de|\flux|
&=\int_\Theta\int_I\psi(\vartheta(t))|\dot\vartheta(t)|\,\de t\de\eta(\vartheta)
\qquad\text{for all }\psi\in\cont_c(\R^n)
\end{align*}
as well as $\eta\in\TPM(\mu_+,\mu_-)$.
By the construction in \cite{Sm93} we may assume $\eta$ to be supported on loop-free paths.
Now by Skorohod's theorem \cite[Thm.\,6.7]{Bil99} there exists an irrigation pattern $\chi_\flux$ between $\mu_+$ and $\mu_-$ which induces $\eta$ so that
\begin{align*}
\int_{\R^n}\varphi\cdot\de\flux
&=\int_\reSpace\int_I\varphi(\chi_\flux(p,t))\cdot\dot\chi_\flux(p,t)\,\de t\de\reMeasure(p)
\qquad\text{for all }\varphi\in\cont_c(\R^n;\R^n)\,,\\
\int_{\R^n}\psi\,\de|\flux|
&=\int_\reSpace\int_I\psi(\chi_\flux(p,t))|\dot\chi_\flux(p,t)|\,\de t\de\reMeasure(p)
\qquad\text{for all }\psi\in\cont_c(\R^n)\,.
\end{align*}
Note that for $\hdone$-almost all $x\in S$ we have $m_{\chi_\flux}(x)=|\theta(x)|$ and $m_{\chi_\flux}=0$ $\hdone\restr R$-almost everywhere for any countably $1$-rectifiable $R\subset\R^n\setminus S$.
Indeed, for a contradiction assume first $m_{\chi_\flux}(x)>0$ on a non-$\hdone$-negligible set $R\subset\R^n\setminus S$, then
\begin{multline*}
0
<\int_Rm_{\chi_\flux}(x)\,\de\hdone(x)
=\int_R\int_\Gamma\setchar{\chi_\flux(p,I)}(x)\,\de\reMeasure(p)\de\hdone(x)\\
=\int_\Gamma\int_{\chi_\flux(p,I)}\setchar{R}(x)\,\de\hdone(x)\de\reMeasure(p)
\leq|\flux|(R)
=|\flux^\perp|(R)
=0\,,
\end{multline*}
the desired contradiction. Likewise, the same calculation yields
\begin{equation*}
\int_Rm_{\chi_\flux}(x)\,\de\hdone(x)
=|\flux|(R)
=\int_R|\theta(x)|\,\de\hdone(x)
\end{equation*}
for any $\hdone$-measurable $R\subset S$ so that $m_{\chi_\flux}=|\theta|$ on $S$.
Therefore, by \cref{thm:GilbertPatterns} and using
\begin{equation*}
|\flux^\perp|(A)
=|\flux|(A\setminus S)
=\int_\reSpace\int_I\setchar{A\setminus S}(\chi_\flux(p,t))|\dot\chi_\flux(p,t)|\,\de t\de\reMeasure(p)
=\int_\reSpace\hdone(\chi_\flux(p,I)\cap A\setminus S)\,\de\reMeasure(p)
\end{equation*}
for any Borel set $A$ we have
\begin{multline*}
\JEn^{\tau,\mu_+,\mu_-}[\chi_\flux]
=\int_S\tau(m_{\chi_\flux}(x))\,\de\hdone(x)+\tau'(0)\int_\reSpace\hdone(\chi_\flux(p,I)\setminus S)\,\de\reMeasure(p)\\
=\int_{S}\tau(|\theta(x)|)\,\de\hdone(x)+\tau'(0)|\flux^\perp|(\R^n)
=\JEn^{\tau,\mu_+,\mu_-}[\flux]\,.\qedhere
\end{multline*}
\end{proof}

\subsection{\TransportPath{}s have lower cost function than irrigation patterns}
Unlike in \cite{BrWi15-equivalent} or \cite{Maddalena-Solimini-Transport-Distances,Maddalena-Solimini-Synchronic} we here again use the generalized Gilbert energy to establish the remaining inequality.

\begin{proposition}[\TransportPath{} of a pattern]
For a given irrigation pattern $\chi$ define the \transportPath{} $\flux_\chi$ via
\begin{equation*}
\int_{\R^n}\varphi\cdot\de\flux_\chi=\int_\reSpace\int_I\varphi(\chi_p(t))\cdot\dot\chi_p(t)\,\de t\,\de \reMeasure(p) \;\text{for all}\; \varphi\in \cont_c(\R^n;\R^n).
\end{equation*}
Then we have $\dv\flux_\chi=\mu_+^\chi-\mu_-^\chi$ as well as
$\flux_\chi=\theta\hdone\restr S_\chi+\flux_\chi^\perp$ for $S_\chi$ from \cref{thm:rectifiabilityFlux}, a measurable $\theta:S_\chi\to\R^n$ with $|\theta(x)|\leq m_\chi(x)$ for all $x\in S_\chi$ and $\theta(x)$ tangent to $S_\chi$ for $\hdone$-almost all $x\in S_\chi$, and a measure
\begin{equation*}
\int_{\R^n}\varphi\cdot\de\flux_\chi^\perp=\int_\reSpace\int_I\setchar{\R^n\setminus S_\chi}(\chi_p(t))\varphi(\chi_p(t))\cdot\dot\chi_p(t)\,\de t\,\de \reMeasure(p) \;\text{for all}\; \varphi\in \cont_c(\R^n;\R^n),
\end{equation*}
singular with respect to $\hdone\restr R$ for any countably $1$-rectifiable $R\subset\R^n$.
\end{proposition}
\begin{proof}
For $\varphi\in \cont_c(\R^n;\R^n)$ and $\tau_{\chi_p(I)}(x)$ the unit tangent vector to $\chi_p(I)$, which is defined for almost all $x\in\chi_p(I)$, we have
\begin{multline*}
\int_{\R^n}\varphi\cdot(\flux_\chi-\flux_\chi^\perp)
=\int_\reSpace\int_{\chi_p(I)\cap S_\chi}\varphi(x)\cdot\tau_{\chi_p(I)}(x)\,\de\hdone(x)\,\de \reMeasure(p)
\leq\int_\reSpace\int_{S_\chi}\setchar{\chi_p(I)}(x)|\varphi(x)|\,\de\hdone(x)\,\de \reMeasure(p)\\
=\int_{S_\chi}\int_\reSpace\setchar{\chi_p(I)}(x)|\varphi(x)|\,\de \reMeasure(p)\,\de\hdone(x)
=\int_{S_\chi}m_\chi(x)|\varphi(x)|\,\de\hdone(x)\,.
\end{multline*}
Since this holds for all $\varphi\in \cont_c(\R^n;\R^n)$, we must have $\flux_\chi-\flux_\chi^\perp=\theta\hdone\restr S_\chi$ for some $|\theta|\leq m_\chi$.
Since $\tau_{\chi_p(I)}(x)$ is tangent to $S_\chi$ for almost all $x\in\chi_p(I)\cap S_\chi$
(indeed, the derivatives of two Lipschitz curves coincide almost everywhere on their intersection%
\notinclude{: Let $D$ be the intersection points at which both curves are differentiable; only countably many points $x\in D$ can be isolated; the rest has coinciding derivative, which can be computed by the limit of the finite difference taken in $D$.}),
we even obtain that $\theta(x)$ is tangent to $S_\chi$ for $\hdone$-almost all $x\in S_\chi$.
Furthermore, letting $R\subset\R^n$ be countably $1$-rectifiable,
\begin{multline*}
\int_{\R^n}\varphi\cdot\de\flux_\chi^\perp\restr R
=\int_\reSpace\int_{R\setminus S_\chi}\setchar{\chi_p(I)}(x)\varphi(x)\cdot\tau_{\chi_p(I)}(x)\,\de\hdone(x)\,\de\reMeasure(p)\\
=\int_{R\setminus S_\chi}\int_\reSpace\setchar{\chi_p(I)}(x)\varphi(x)\cdot\tau_{\chi_p(I)}(x)\,\de\reMeasure(p)\,\de\hdone(x)
=0\,.
\end{multline*}
Finally, for any $\psi\in\cont_c(\R^n)$ we have
\begin{multline*}
\int_{\R^n}\psi\,\de\dv\flux_\chi
=-\int_{\R^n}\nabla\psi\cdot\de\flux_\chi
=\int_\reSpace\int_I\nabla\psi(\chi_p(t))\cdot\dot\chi_p(t)\,\de t\,\de \reMeasure(p)\\
=\int_\reSpace\psi(\chi_p(1))-\psi(\chi_p(0))\,\de \reMeasure(p)
=\int_{\R^n}\psi\,\de(\mu_+^\chi-\mu_-^\chi)\,.\qedhere
\end{multline*}
\end{proof}

\begin{corollary}[\TransportPath{}s have lower cost]\label{thm:PatternsLowerCost}
Let $\tau$ be a transportation cost.
For every irrigation pattern $\chi$ with finite cost $\JEn^{\tau,\mu_+,\mu_-}[\chi]$, the \transportPath{} $\flux_\chi$ satisfies
\begin{displaymath}
\JEn^{\tau,\mu_+,\mu_-}[\flux_\chi]\leq\JEn^{\tau,\mu_+,\mu_-}[\chi]\,.
\end{displaymath}
\end{corollary}
\begin{proof}
This directly follows from the fact that $\flux_\chi$ is a \transportPath{} between $\mu_+$ and $\mu_-$ with
\begin{multline*}
\JEn^{\tau,\mu_+,\mu_-}[\flux_\chi]
=\int_{S_\chi}\tau(|\theta(x)|)\,\de\hdone(x)+\tau'(0)|\flux_\chi^\perp|(\R^n)\\
\leq\int_{S_\chi}\tau(m_\chi(x))\,\de\hdone(x)+\tau'(0)\int_\reSpace\hdone(\chi_p(I)\setminus S_\chi)\,\de\reMeasure(p)
\leq\JEn^{\tau,\mu_+,\mu_-}[\chi]\,,
\end{multline*}
where we have used \cref{thm:GilbertPatterns} and \cref{thm:GilbertFlux}.
\end{proof}

\begin{corollary}[Model equivalence]\label{thm:modelEquivalence}
We have
\begin{displaymath}
\inf_\flux\JEn^{\tau,\mu_+,\mu_-}[\flux]=\inf_{\chi}\JEn^{\tau,\mu_+,\mu_-}[\chi]\,.
\end{displaymath}
If one model admits a minimizer (for instance, if $\tau$ is admissible and $\mu_+,\mu_-$ have bounded support), then so does the other, and for any optimal irrigation pattern $\chi$, the \transportPath{} $\flux_\chi$ is optimal,
while for any optimal \transportPath{} $\flux$, the irrigation pattern $\chi_\flux$ is optimal.
\end{corollary}
\begin{proof}
The equality of the infima follows from \cref{prop:constructPatternFromFlux} and \cref{thm:PatternsLowerCost}.
The statement about optimal patterns and \transportPaths{} follows from $\JEn^{\tau,\mu_+,\mu_-}[\flux_\chi]\leq\JEn^{\tau,\mu_+,\mu_-}[\chi]$
and $\JEn^{\tau,\mu_+,\mu_-}[\chi_\flux]\leq\JEn^{\tau,\mu_+,\mu_-}[\flux]$ as derived in the proof of \cref{thm:PatternsLowerCost} and \cref{prop:constructPatternFromFlux}.
\end{proof}

\section{Final remarks and open problems}
We have introduced a Eulerian and a Lagrangian model for ramified transportation networks.
The models generalize the well-known branched transport model in that the most general class of reasonable transportation costs is considered instead of merely the branched transport choice $\tau(w)=w^\alpha$.
The corresponding cost functionals are quite explicitly characterized with the help of the concepts of $1$-currents and flat $1$-chains.
This then makes it rather straightforward to establish the equivalence between both models.

There are multiple directions for further investigation.

In an ongoing study the authors establish yet a further equivalent formulation
which generalizes the original formulation of the urban planning model and which has a flavour to it more like classical optimal transport.

Another direction consists in modifying the cost functional to obtain different models.
For instance, if instead of scalar measures $\mu_+,\mu_-$ one uses vector-valued measures in $\rca(\R^n;\R^n)$
and if one correspondingly replaces the mass preservation conditions \eqref{eqn:massPreservation} by
\begin{equation*}
\mu_+(\{v\}) + \sum_{\substack{e \in E(G)\\e^- = v}} \hat ew(e) = \mu_-(\{v\}) + \sum_{\substack{e \in E(G)\\e^+ = v}} \hat ew(e)
\end{equation*}
for all graph vertices (where $\hat e$ denotes the edge orientation), one arrives at a model for a structure transporting a vector-valued quantity, for instance an elastic structure bearing a mechanical load.
The corresponding flux then has the interpretation of a divergence-free stress field and becomes
\begin{equation*}
\flux_G=\sum_{e\in E(G)}w(e)\hat e\otimes\hat e\hdone\restr e\,.
\end{equation*}
As an example, the Wasserstein transportation cost $\tau(w)=w$ then would yield the well-known Michell trusses as optimal structures \cite{Mi94}.
An alternative formulation of the same model would be based on directed graphs with vector-valued weight function $w$,
in which case an edge $e$ would be assigned the cost $\tau(|w(e)|)$ if $w(e)$ is parallel to $\hat e$ and $\infty$ else;
a corresponding description via flat chains would then use flat $1$-chains with $\R^n$ as coefficient group.

Further model variants are obtained by letting each edge length $l(e)$ enter nonlinearly into the cost functional (for instance in the form $\sqrt{l(e)}$, modelling an efficiency gain on long distances),
by adding vertex costs on top (for instance as in \cite{BrPoWi12}, modelling a cost for direction changes),
and by allowing but penalizing mass loss or gain during the transport (as in unbalanced optimal transport, for instance \cite{DNSTransportDistances09}).

Finally, it is now possible to introduce the landscape function in this general setting.
Originally introduced by Santambrogio in \cite{Santambrogio-Landscape} and also studied by Brancolini and Solimini in \cite{Brancolini-Solimini-Hoelder} and Xia in \cite{Xia-Landscape},
the landscape function is related to erosion problems in geophysics \cite{Banavar2001,FractalBasins}.
Its value in $x$ is the cost to transport a mass particle from its initial position to the point $x$,
and its regularity properties are important in order to investigate the regularity of the minimizers of the branched transport functional \cite{Brancolini-Solimini-Fractal}.
A general regularity theory of the landscape function could eventually give an insight into the regularity of the minimizers of the functionals studied in this paper.

\section*{Acknowledgements}
This work was supported by the Deutsche Forschungsgemeinschaft (DFG), Cells-in-Motion Cluster of Excellence (EXC 1003-CiM), University of M\"unster, Germany.
B.W.'s research was supported by the Alfried Krupp Prize for Young University Teachers awarded by the Alfried Krupp von Bohlen und Halbach-Stiftung.

\bibliographystyle{alpha}
\bibliography{BrWi14}

\newcommand{\etalchar}[1]{$^{#1}$}
\begin{thebibliography}{CDRMS17}

\bibitem[BB05]{Brancolini-Buttazzo}
Alessio Brancolini and Giuseppe Buttazzo.
\newblock Optimal networks for mass transportation problems.
\newblock {\em ESAIM Control Optim. Calc. Var.}, 11(1):88--101 (electronic),
  2005.

\bibitem[BCF{\etalchar{+}}01]{Banavar2001}
Jayanth~R. Banavar, Francesca Colaiori, Alessandro Flammini, Amos Maritan, and
  Andrea Rinaldo.
\newblock Scaling, optimality, and landscape evolution.
\newblock {\em Journal of Statistical Physics}, 104(1):1--48, 2001.

\bibitem[BCM05]{Bernot-Caselles-Morel-Traffic-Plans}
Marc Bernot, Vicent Caselles, and Jean-Michel Morel.
\newblock Traffic plans.
\newblock {\em Publ. Mat.}, 49(2):417--451, 2005.

\bibitem[BCM08]{Bernot-Caselles-Morel-Structure-Branched}
Marc Bernot, Vicent Caselles, and Jean-Michel Morel.
\newblock The structure of branched transportation networks.
\newblock {\em Calc. Var. Partial Differential Equations}, 32(3):279--317,
  2008.

\bibitem[BCM09]{BeCaMo09}
Marc Bernot, Vicent Caselles, and Jean-Michel Morel.
\newblock {\em Optimal transportation networks}, volume 1955 of {\em Lecture
  Notes in Mathematics}.
\newblock Springer-Verlag, Berlin, 2009.
\newblock Models and theory.

\bibitem[Ber05]{Be05}
Marc Bernot.
\newblock {\em Optimal transport and irrigation}.
\newblock PhD thesis, \'Ecole normale sup\'erieure de Cachan, 2005.
\newblock https://tel.archives-ouvertes.fr/tel-00132078/.

\bibitem[Bil99]{Bil99}
Patrick Billingsley.
\newblock {\em Convergence of probability measures}.
\newblock Wiley Series in Probability and Statistics: Probability and
  Statistics. John Wiley \& Sons, Inc., New York, second edition, 1999.
\newblock A Wiley-Interscience Publication.

\bibitem[Blu70]{MR0268781}
Leonard~M. Blumenthal.
\newblock {\em Theory and applications of distance geometry}.
\newblock Second edition. Chelsea Publishing Co., New York, 1970.

\bibitem[BPSS09]{BuPrSoSt09}
Giuseppe Buttazzo, Aldo Pratelli, Sergio Solimini, and Eugene Stepanov.
\newblock {\em Optimal urban networks via mass transportation}, volume 1961 of
  {\em Lecture Notes in Mathematics}.
\newblock Springer-Verlag, Berlin, 2009.

\bibitem[BPW13]{BrPoWi12}
Kristian Bredies, Thomas Pock, and Benedikt Wirth.
\newblock Convex relaxation of a class of vertex penalizing functionals.
\newblock {\em Journal of Mathematical Imaging and Vision}, 47(3):278--302,
  2013.

\bibitem[BRW18]{BrRoWi16}
Alessio Brancolini, Carolin Rossmanith, and Benedikt Wirth.
\newblock Optimal micropatterns in 2{D} transport networks and their relation
  to image inpainting.
\newblock {\em Archive for Rational Mechanics and Analysis}, 228(1):279--308,
  2018.

\bibitem[BS11]{Brancolini-Solimini-Hoelder}
Alessio Brancolini and Sergio Solimini.
\newblock On the {H}\"older regularity of the landscape function.
\newblock {\em Interfaces Free Bound.}, 13(2):191--222, 2011.

\bibitem[BS14]{Brancolini-Solimini-Fractal}
Alessio Brancolini and Sergio Solimini.
\newblock Fractal regularity results on optimal irrigation patterns.
\newblock {\em J. Math. Pures Appl. (9)}, 102(5):854--890, 2014.

\bibitem[BW16]{BrWi15-equivalent}
Alessio Brancolini and Benedikt Wirth.
\newblock Equivalent formulations for the branched transport and urban planning
  problems.
\newblock {\em J. Math. Pures Appl.}, 106(4):695--724, 2016.

\bibitem[BW17]{BrWi15-micropatterns}
Alessio Brancolini and Benedikt Wirth.
\newblock Optimal energy scaling for micropatterns in transport networks.
\newblock {\em SIAM J. Math. Anal.}, 49(1):311--359, 2017.

\bibitem[CDRMS17]{CoRoMa17}
M.~Colombo, A.~De~Rosa, A.~Marchese, and S.~Stuvard.
\newblock On the lower semicontinuous envelope of functionals defined on
  polyhedral chains.
\newblock preprint on http://cvgmt.sns.it/paper/3347/, 2017.

\bibitem[DNS09]{DNSTransportDistances09}
Jean Dolbeault, Bruno Nazaret, and Giuseppe Savar\'e.
\newblock A new class of transport distances between measures.
\newblock {\em Calc. Var. Partial Differential Equations}, 34(2):193--231,
  2009.

\bibitem[Edg95]{Ed94}
G.~A. Edgar.
\newblock Fine variation and fractal measures.
\newblock {\em Real Anal. Exchange}, 20(1):256--280, 1994/95.

\bibitem[Fed69]{Fe69}
Herbert Federer.
\newblock {\em Geometric measure theory}.
\newblock Die Grundlehren der mathematischen Wissenschaften, Band 153.
  Springer-Verlag New York Inc., New York, 1969.

\bibitem[Fle66]{Fl66}
Wendell~H. Fleming.
\newblock Flat chains over a finite coefficient group.
\newblock {\em Trans. Amer. Math. Soc.}, 121:160--186, 1966.

\bibitem[GK90]{MR1074005}
Kazimierz Goebel and W.~A. Kirk.
\newblock {\em Topics in metric fixed point theory}, volume~28 of {\em
  Cambridge Studies in Advanced Mathematics}.
\newblock Cambridge University Press, Cambridge, 1990.

\bibitem[Kic89]{Ki89}
Satyanad Kichenassamy.
\newblock Compactness theorems for differential forms.
\newblock {\em Communications on Pure and Applied Mathematics}, 42(1):47--53,
  1989.

\bibitem[Kuc09]{Ku09}
Marek Kuczma.
\newblock {\em An introduction to the theory of functional equations and
  inequalities}.
\newblock Birkh\"auser Verlag, Basel, second edition, 2009.
\newblock Cauchy's equation and Jensen's inequality, Edited and with a preface
  by Attila Gil\'anyi.

\bibitem[Laa62]{La62}
Richard~George Laatsch.
\newblock {\em Subadditive Functions of One Real Variable}.
\newblock PhD thesis, Oklahoma State University, 1962.
\newblock https://shareok.org/handle/11244/30626.

\bibitem[Mic04]{Mi94}
{A.G.M.} Michell.
\newblock {LVIII.} the limits of economy of material in frame-structures.
\newblock {\em Philosophical Magazine Series 6}, 8(47):589--597, 1904.

\bibitem[MS09]{Maddalena-Solimini-Transport-Distances}
Francesco Maddalena and Sergio Solimini.
\newblock Transport distances and irrigation models.
\newblock {\em J. Convex Anal.}, 16(1):121--152, 2009.

\bibitem[MS13]{Maddalena-Solimini-Synchronic}
Francesco Maddalena and Sergio Solimini.
\newblock Synchronic and asynchronic descriptions of irrigation problems.
\newblock {\em Adv. Nonlinear Stud.}, 13(3):583--623, 2013.

\bibitem[MSM03]{Maddalena-Morel-Solimini-Irrigation-Patterns}
Francesco Maddalena, Sergio Solimini, and Jean-Michel Morel.
\newblock A variational model of irrigation patterns.
\newblock {\em Interfaces Free Bound.}, 5(4):391--415, 2003.

\bibitem[MW19]{MaWi19}
Andrea Marchese and Benedikt Wirth.
\newblock Approximation of rectifiable 1-currents and weak-{$\ast$} relaxation
  of the {$h$}-mass.
\newblock {\em J. Math. Anal. Appl.}, 479(2):2268--2283, 2019.

\bibitem[RIR01]{FractalBasins}
Ignacio Rodríguez-Iturbe and Andrea Rinaldo.
\newblock {\em Fractal River Basins. Chance and Self-Organization}.
\newblock Cambridge University Press, 2001.

\bibitem[Roy88]{Royden-Real-Analysis}
Halsey~L. Royden.
\newblock {\em Real analysis}.
\newblock Macmillan Publishing Company, New York, third edition, 1988.

\bibitem[Rud87]{Ru87}
Walter Rudin.
\newblock {\em Real and complex analysis}.
\newblock McGraw-Hill Book Co., New York, third edition, 1987.

\bibitem[San07]{Santambrogio-Landscape}
Filippo Santambrogio.
\newblock Optimal channel networks, landscape function and branched transport.
\newblock {\em Interfaces Free Bound.}, 9(1):149--169, 2007.

\bibitem[{\v{S}}il08]{Sil08}
Miroslav {\v{S}}ilhav{\'y}.
\newblock Normal currents: Structure, duality pairings and div--curl lemmas.
\newblock {\em Milan Journal of Mathematics}, 76(1):275--306, 2008.

\bibitem[Sim83]{Si83}
Leon Simon.
\newblock {\em Lectures on geometric measure theory}, volume~3 of {\em
  Proceedings of the Centre for Mathematical Analysis, Australian National
  University}.
\newblock Australian National University, Centre for Mathematical Analysis,
  Canberra, 1983.

\bibitem[Smi93]{Sm93}
S.~K. Smirnov.
\newblock Decomposition of solenoidal vector charges into elementary solenoids,
  and the structure of normal one-dimensional flows.
\newblock {\em Algebra i Analiz}, 5(4):206--238, 1993.

\bibitem[Vil09]{Villani-Transport-Old-New}
C{\'e}dric Villani.
\newblock {\em Optimal transport}, volume 338 of {\em Grundlehren der
  Mathematischen Wissenschaften [Fundamental Principles of Mathematical
  Sciences]}.
\newblock Springer-Verlag, Berlin, 2009.
\newblock Old and new.

\bibitem[\v{S}07]{Si07}
Miroslav \v{S}ilhav\'y.
\newblock Divergence measure vectorfields: their structure and the divergence
  theorem.
\newblock In {\em Mathematical modelling of bodies with complicated bulk and
  boundary behavior}, volume~20 of {\em Quad. Mat.}, pages 217--237. Dept.
  Math., Seconda Univ. Napoli, Caserta, 2007.

\bibitem[Whi57]{Wh57}
Hassler Whitney.
\newblock {\em Geometric integration theory}.
\newblock Princeton University Press, Princeton, N. J., 1957.

\bibitem[Whi99a]{Wh99b}
Brian White.
\newblock The deformation theorem for flat chains.
\newblock {\em Acta Math.}, 183(2):255--271, 1999.

\bibitem[Whi99b]{Wh99}
Brian White.
\newblock Rectifiability of flat chains.
\newblock {\em Annals of Mathematics}, 150(1):165--184, 1999.

\bibitem[Xia03]{Xia-Optimal-Paths}
Qinglan Xia.
\newblock Optimal paths related to transport problems.
\newblock {\em Commun. Contemp. Math.}, 5(2):251--279, 2003.

\bibitem[Xia04]{Xia-Interior-Regularity}
Qinglan Xia.
\newblock Interior regularity of optimal transport paths.
\newblock {\em Calc. Var. Partial Differential Equations}, 20(3):283--299,
  2004.

\bibitem[Xia14]{Xia-Landscape}
Qinglan Xia.
\newblock On landscape functions associated with transport paths.
\newblock {\em Discrete Contin. Dyn. Syst.}, 34(4):1683--1700, 2014.

\end{thebibliography}

\end{document}